\documentclass[12pt]{amsart}

\pdfoutput=1

\usepackage{amsthm}
\usepackage{amsmath}
\usepackage{amstext}
\usepackage{amssymb}
\usepackage{sagetex}

\usepackage{caption}
\usepackage{verbatim}
\usepackage{setspace}
\usepackage[dvips,letterpaper,margin= 1.2 in]{geometry}

\usepackage{graphicx}
\usepackage{hyperref}
\usepackage{pdflscape}

\usepackage{float}
\usepackage[utf8]{inputenc}
\usepackage[final]{pdfpages}

\newtheorem{theorem}{Theorem}[section]
\newtheorem{example}[theorem]{Example}
\newtheorem{lemma}[theorem]{Lemma}

\newtheorem{remark}[theorem]{Remark}

\newtheorem{definition}[theorem]{Definition}

\newtheorem{problem}[theorem]{Problem}

\newcommand{\ABC}{\mathcal{A}}

\title[Beyond Aztec Castles]{Beyond Aztec Castles:  \\ Toric Cascades in the $dP_3$ Quiver}

\begin{document}

\author{Tri Lai}
\address{Department of Mathematics, University of Nebraska - Lincoln, Lincoln, NE 68588}
\email{tlai3@unl.edu}

\author{Gregg Musiker}
\address{School of Mathematics, University of Minnesota, Minneapolis, Minnesota 55455}
\email{musiker@math.umn.edu}


\date{June 29, 2016}


\maketitle

\begin{abstract}
Given one of an infinite class of supersymmetric quiver gauge theories, string theorists can associate a corresponding toric variety (which is a Calabi-Yau 3-fold) as well as an associated combinatorial model known as a brane tiling.  In combinatorial language, a brane tiling is a bipartite graph on a torus and its perfect matchings are of interest to both combinatorialists and physicists alike.  A cluster algebra may also be associated to such quivers and in this paper we study the generators of this algebra, known as cluster variables, for the quiver associated to the cone over the del Pezzo surface $dP_3$.  In particular, mutation sequences involving mutations exclusively at vertices with two in-coming arrows and two out-going arrows are referred to as toric cascades in the string theory literature.  Such toric cascades give rise to interesting discrete integrable systems on the level of cluster variable dynamics.  We provide an explicit algebraic formula for all cluster variables which are reachable by toric cascades as well as a combinatorial interpretation involving perfect matchings of subgraphs of the $dP_3$ brane tiling for these formulas in most cases.
\end{abstract}

\tableofcontents

\section{Introduction} 
\label{intro}
This paper discusses an algebraic combinatorics problem that arises in the theory of cluster algebras but whose study is also spurred on by geometric models appearing in string theory.
The engineering of four-dimensional quantum field theories in string theory has developed a great deal beyond $N=4$ supersymmetric Yang-Mills.  Doubly-periodic bipartite planar graphs, also known as {\bf brane tilings}, are of special interest to physicists since perfect matchings on such graphs provide information regarding the geometry of certain toric varieties which are Calabi-Yau $3$-folds.  Such varieties are the moduli spaces of certain quiver representations modulo relations as given by the Jacobian of a quiver with potential.  In the language of theoretical physics, the AdS/CFT correspondence \cite{AdS1,AdS2,AdS3} associates a $N=1$ superconformal quantum field theory known as a supersymmetric quiver gauge theory to a toric Sasaki-Einstein $5$-manifold.

In this paper, we focus on a specific example of such a theory, which is associated to the cone over the del Pezzo surface of degree $6$ ($\mathbb{CP}^2$ blown up at three points) which we refer to as $\mathbf{dP_3}$ in this paper following the conventions of the physics literature \cite{BP,feng}.  As previously studied, via the AdS/CFT correspondence, the corresponding quiver gauge theory is built using a highly symmetric six-vertex quiver (illustrated in Figure \ref{fig:quiv_brane} (Middle)) and the potential
\begin{eqnarray} \label{eq:pot} W &=& A_{16}A_{64}A_{42}A_{25}A_{53}A_{31}
~+~ A_{14}A_{45}A_{51} ~+~ A_{23}A_{36}A_{62} \\ \nonumber &-&  A_{16}A_{62}A_{25}A_{51}
~-~ A_{36}A_{64}A_{45}A_{53} ~-~ A_{14}A_{42}A_{23}A_{31}.\end{eqnarray}

About a decade ago, using methods from dimer theory first developed by mathematicians and physicists alike in the context of statistical mechanics \cite{crystal,KOS}, physicists described how to associate a brane tiling to such a quiver gauge theory \cite{brane_dimer,Vafa}.  See \cite{GK} for a more recent mathematical treatment.  For the $dP_3$ case, this brane tiling is illustrated in Figure \ref{fig:quiv_brane} (Right).  We refer to this as the $dP_3$ brane tiling, or $\mathcal{T}$ for short.

\vspace{1em}

Given a quiver, e.g. the one illustrated in Figure \ref{fig:quiv_brane} (Middle), one may define a {\bf cluster algebra} following the construction of Fomin and Zeleveinsky \cite{FZ}.  (See Section \ref{Sec:Prelim} for details.)  In general, a cluster algebra will have an infinite number of generators, called {\bf cluster variables}, which are grouped into overlapping subsets called {\bf clusters}.  Cluster variables do not freely generate a cluster algebra, rather they satisfy certain binomial exchange relations which are deduced from the data of the quiver defining that given cluster algebra.

A cluster algebra is known as {\bf mutation-finite} if there is a finite list of quivers each of which would define that cluster algebra.  A cluster algebra is known as {\bf mutation-infinite} otherwise.  Combinatorial formulas are known for many mutation-finite cluster algebras \cite{MSW} since most such cluster algebras come from surfaces (see \cite{FeShTu}).  For mutation-infinite cluster algebras, providing a combinatorial formula for every single cluster variable of that algebra is quite a challenge.  However, for some specific mutation-infinite cluster algebras, there has been significant work which provides combinatorial formulas for cluster variables lying in certain subsets \cite{BMPW, DiF, glick, speyer}.

Our work yields further progress towards this combinatorial goal for a rich example which helps to illustrate new features that were unseen in previous descriptions of subsets of cluster varibles for mutation-infinite cluster algebras.  In particular, we describe a three-parameter family of subgraphs of $\mathcal{T}$ with the property that a certain weighted enumeration of {\bf dimers} (also known as {\bf perfect matchings}) yields most of the Laurent expansions of the {\bf toric cluster variables}, which are a subset of the generators of the associated $dP_3$-cluster algebra.  This is stated more precisely as Theorem \ref{thm:main} in terms of the language developed in our paper.  We present our combinatorial formula and methods in full detail with an intention to generalize these methods to further examples in future work.

\vspace{1em}

We now say a little about the history of this problem before discussing the organization of this paper.  Even before the connections between the mathematics and physics of brane tilings were well-known, mathematicians had been studying perfect matchings of $\mathcal{T}$ since the turn of the millennium, as by Jim Propp, Ben Wieland \cite[Problem 15]{EnumPropp} and Mihai Ciucu \cite[Section 7]{perfect}.  (More precisely they were studying tilings of the dual graph of $\mathcal{T}$ which consists of regular hexagons, squares, and triangles.)  In the modern language, Propp conjectured (proven by unpublished work of Wieland and published work by Ciucu) that a certain one-parameter family of subgraphs of $\mathcal{T}$ called {\bf Aztec Dragons}, as an analogue of {\bf Aztec Diamonds} \cite{Elkies}, had the property that the number of perfect matchings of each subgraph was a power of two.

While leading an REU (Research Experience for Undergraduates) at the University of Minnesota in 2012, the second author proposed a problem motivated by Sergey Fomin and Andrei Zelevinsky's theory of cluster algebras \cite{FZ} and the above mentioned theoretical physics.  Definitions of cluster algebras and mutation appear in Section \ref{Sec:Mutations}.  The goal was to obtain combinatorial formulas for the Laurent expansions of cluster variables obtained by certain sequences of mutations of quivers of interest to string theorists such as those associated to reflexive polygons \cite{hanany_polygons}.  As part of this REU, Sicong Zhang was inspired by a paper by Cyndie Cottrell and Benjamin Young \cite{CY} and proved that by weighting the perfect matchings of Aztec Dragons in the appropriate way, it followed that the resulting partition functions (which are Laurent polynomials in this case) agreed with the corresponding cluster variables \cite{zhang}.

In the subsequent summer, 2013 REU students Megan Leoni, Seth Neel, and Paxton Turner worked with the second author to provide a combinatorial interpretation for a two-parameter family of mutation sequences by extending to a family of subgraphs beyond Aztec Dragons \cite{LMNT}.  These were referred to as NE (Northeast) Aztec Castles and SW (Southwest) Aztec Castles.
Simultaneously, in Indiana, an alternative motivation for generalizing Aztec Dragons was being studied by Ciucu and the first author.

Similar to the Aztec Diamonds and the Aztec Dragons, James Propp also introduced Aztec Dungeons on a different lattice (the lattice corresponding to the affine Coxeter group $G_2$) and conjectured an explicit tiling formula for these regions, which is a power of $13$ or twice a power of $13$. This conjecture was later proven by Ciucu (see \cite[Section 8]{perfect}).  Inspired by the Aztec Dungeons, Matt Blum considered a hexagonal counterpart of them called \textbf{Hexagonal Dungeons}.  Blum had conjectured a striking pattern of the number of tilings of a hexagonal dungeon of side-lengths $a,2a,b,a,2a,b$ in cyclic order ($b\geq 2a$), which is $13^{2a^2}14^{\lfloor a^2/2\rfloor}$ (see \cite[Problem 25]{EnumPropp}). The first author and Ciucu proved and generalized Blum's conjecture \cite{lai'} by enumerating tilings of two families of regions restricted in a six-sided contour. This proof inspired a number of similar regions on different lattices, including {\bf Dragon Regions} denoted by $DR^{(1)}$ and $DR^{(2)}$ in \cite{LaiNewDungeon}. These Dragon Regions generalized Propp's Aztec Dragons \cite{perfect,EnumPropp,CY} and the NE/SW-Aztec Castles of \cite{LMNT}.

The present work provides a vivid picture of the cluster algebra associated to the $dP_3$ quiver and is a culmination of the above work on one-parameter, two-parameter, unweighted families of subgraphs of the $dP_3$ brane tiling $\mathcal{T}$.  We describe a general construction that yields the Dragon regions of \cite{LaiNewDungeon} such that the Laurent polynomials obtained from the weighted enumeration of perfect matchings on such subgraphs of $\mathcal{T}$ are exactly a three-parameter family of cluster variables in the $dP_3$-cluster algebra.  For lack of a better name, we have decided to refer to the three-parameter family of subgraphs $\mathcal{T}$ constructed in this paper as (general) {\bf Aztec Castles}.  The work of the first author in \cite{LaiNewDungeon} can be considered as a special case of the unweighted version of the main result, Theorem \ref{thm:main}, of the present paper.  

We prove in Section \ref{sec:gentoric} that this {\bf three-parameter family of cluster variables} is indeed the {\bf set of toric cluster variables}.   Though our work begins with a description of certain mutation sequences known as {\bf generalized $\tau$-mutation sequences}, we show, see Lemma \ref{lem:gentoric}, that for all toric cluster variables $X$, there exists a cluster reachable via a generalized $\tau$-mutation sequence which contains $X$.  Consequently, the aforementioned three-dimensional parameterization is indeed a parameterization of the entire set of toric cluster variables.

The combinatorial interpretation developed in this paper is in a similar spirit as David Speyer's {\bf crosses-and-wrenches} graphical interpretation for solutions to the Octahedron Recurrence \cite{speyer}, Philippe Di Francesco's dimer solutions to {\bf T-Systems} \cite{DiF}, or Bosquet-Melou-Propp-West's {\bf Pinecone} graphs \cite{BMPW}.  However as we highlight, both throughout the paper and especially in Section \ref{sec:open}, the $dP_3$ cluster algebra and Aztec Castles provide several new combinatorial features that were not seen in these earlier examples, and yet, due to the symmetry of the $dP_3$ quiver, this example is ideal for detailed analysis.

\vspace{1em}

In Section 2, we begin with the relevant background material on cluster algebras and then provide a geometric viewpoint on certain cluster mutations (toric mutations) of the $dP_3$ cluster algebra.  Section 3 presents our first theorem (Theorem \ref{thm:explicit}) which is an explicit algebraic formula for the Laurent expansion of any cluster variable reachable from toric mutations of the $dP_3$ quiver.  This provides a three-parameter family which extends the one-parameter and two-parameter families of cluster variables discussed in \cite{zhang} and \cite{LMNT}, respectively.  In Section 4, we illustrate how to construct Aztec Castles via six-tuples which is motivated by the constructions of \cite{LaiNewDungeon} and \cite{LMNT}.  Section 5 discusses why it is sufficient to consider a certain three-parameter family of six-tuples and provides a complete illustration of all of the possible shapes of Aztec Castles.  This leads us to our main theorem (Theorem \ref{thm:main}) which yields a combinatorial interpretation in terms of the $dP_3$ brane tiling for most cluster variables reachable from toric mutations.  There is an issue regarding self-intersecting contours that prevents us from getting a combinatorial interpretation for all reachable cluster variables in terms of perfect matchings even though the algebraic formula of Theorem \ref{thm:explicit} still applies in such cases.  Sections 6 and 7 provide the proof of our main theorem, while Section 8 includes some examples illustrating the proof.  We finish with open problems and directions for further research in Section 9.  
\vspace{1em}

\section{From Mutations To Alcove Walks} \label{Sec:Mutations}

We begin by reviewing the definition of quiver mutations and cluster mutations.  This leads us to study a special subcollection of mutation sequences known as toric mutation sequences.  This special subcollection includes the $\tau$-mutation sequences from \cite{LMNT} as special cases.  We provide a geometric interpretation that allows us to examine toric mutation sequences essentially as alcove-walks on the $\mathbb{Z}^3$ lattice, which extends the visualization of $\tau$-mutation sequences as $\mathbb{Z}^2$-alcove walks from \cite{LMNT}.  We exploit this identification later in this paper to obtain explicit algebraic formulas, Theorem \ref{thm:explicit}, and combinatorial interpretations, Theorem \ref{thm:main}, for the resulting cluster variables.

\subsection{Quiver and Cluster Mutations}
\label{subsec:quivbrane}
A \textbf{quiver} $Q$ is a directed finite graph with a set of vertices $V$ and a set of edges $E$ connecting them whose direction is denoted by an arrow. For our purposes $Q$ may have multiple edges connecting two vertices but may not contain any loops or $2-$cycles. We can relate a cluster algebra with initial seed $\{x_{1},x_{2},\ldots,x_{n}\}$ to $Q$ by associating a cluster variable $x_{i}$ to every vertex labeled $i$ in $Q$ where $|V| = n$. The cluster is the union of the cluster variables at each vertex.

\begin{definition}[\textbf{Quiver Mutation}] Mutating at a vertex $i$ in $Q$ is denoted by $\mu_{i}$ and corresponds to the following actions on the quiver:
\begin{itemize}
\item For every 2-path through $i$ (e.g. $j \rightarrow i \rightarrow k$), add an edge from $j$ to $k$.
\item Reverse the directions of the arrows incident to $i$
\item Delete any 2-cycles created from the previous two steps.
\end{itemize}
\end{definition}
\noindent When we mutate at a vertex $i$, the cluster variable at this vertex is updated and all other cluster variables remain unchanged \cite{FZ}. The action of $\mu_{i}$ on the cluster leads to the following binomial exchange relation:
\begin{equation*}
\label{eq: exchange relation}
x'_{i}x_{i} = \prod_{i \rightarrow j \; \mathrm{in} \; Q}x_{j}^{a_{i \rightarrow j}} + \prod_{j \rightarrow i \; \mathrm{in} \; Q}x_{j}^{b_{j \rightarrow i}}
\end{equation*}
where $x_i'$ is the new cluster variable at vertex $i$, $a_{i \rightarrow j}$ denotes the number of edges from $i$ to $j$, and $b_{j \rightarrow i}$ denotes the number of edges from $j$ to $i$.

\subsection{The Del Pezzo 3 Quiver and its Brane Tiling} \label{Sec:Prelim}

With quiver and cluster mutation in mind, we introduce our main character, the quiver $Q$ associated to the third del Pezzo surface ($\mathbf{dP_3}$), illustrated in Figure \ref{fig:quiv_brane} with its associated brane tiling \cite{brane_dimer,Vafa,FHMSVW,HV}.  We focus on one of the four possible toric phases of this quiver.  In particular, $Q$ is typically referred to as Model 1, as it is in Figure 27 of \cite{franco_eager}.  (However, for completeness, we mention that these were first described earlier in the physics literature in \cite{BP,feng}).)
Its {\bf Toric Diagram} is the the convex hull of the vertices $\{(-1,1), (0,1), (1,0), (1,-1), (0,-1), (-1,0), (0,0)\}$.  See Figure \ref{fig:quiv_brane} (Left).

\begin{figure}\centering
\includegraphics[width=12cm]{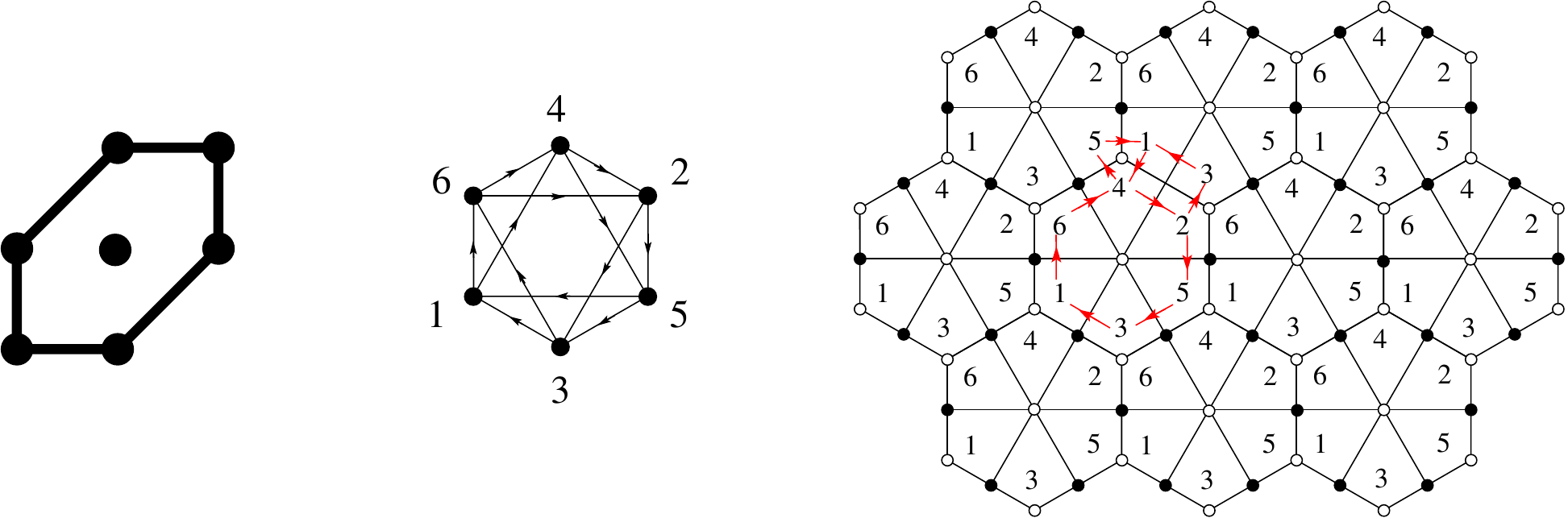}
 \caption{\small The $dP_3$ toric diagram, its quiver $Q$, and its associated brane tiling $\mathcal{T}$.}
    \label{fig:quiv_brane}
\end{figure}

In a toric supersymmetric gauge theory, a formal linear combination of closed cycles of the quiver where each edge in the unit cell of the brane tiling appears exactly twice, once for clockwise (positive) orientation and once for counter-clockwise (negative) orientation, is
known as a \textbf{superpotential}.  Using a pair $(Q,W)$ where $Q$ is a quiver that can be drawn on a torus and $W$ is a related superpotential, we can uniquely build a $2$-dimensional cell complex using potential terms as $2$-faces, quiver arrows as $1$-faces, and quiver vertices as $0$-faces.  We construct this cell complex on a torus (unfolded on its universal cover the Euclidean plane), and then take its planar dual to get the {\bf brane tiling} associated to $(Q,W)$.

Proceeding in this way for the case of the $dP_3$ quiver with the superpotential $W$ given in (\ref{eq:pot}), the associated brane tiling is illustrated on the right-hand-side of Figure \ref{fig:quiv_brane}.  We denote the $dP_3$ brane tiling as $\mathcal{T}$.

\subsection{Toric Mutations} \label{Sec:Toric}

We say that a vertex of a quiver $Q$ is {\bf toric} if it has both in-degree and out-degree $2$.  A {\bf toric mutation} is a mutation at a toric vertex.  Beginning with the $dP_3$ quiver introduced in Section \ref{Sec:Prelim}, mutation at any vertex is a toric mutation.  We call this initial quiver Model 1.  After any such mutation, up to graph isomorphism, we have a quiver as illustrated in the top-right of Figure \ref{fig:dP3QuiverModels} (the graphics are made with \cite{sage}).

In a Model 2 quiver, four out of the six vertices are toric at this point.  Two of these toric vertices come as an antipodal pair on the equator of the octahedron.  Mutation at one of them leads back to the original Model 1 quiver and mutation at the antipode leads to a Model 1 quiver where some of the vertices have been permuted.  The remaining two toric vertices lie at the poles of the octahedron.  Mutation at either of those toric vertices leads to a Model 3 quiver or the reverse of a Model 3 quiver.  See the bottom-left of Figure \ref{fig:dP3QuiverModels}. By abuse of notation we will refer to both of these as Model 3 quivers.

In a Model 3 quiver, there is a unique vertex with in-degree and out-degree $3$.  All other vertices in a Model 3 quiver are toric.  Three of these toric vertices are incident to a double-arrow.  Mutation at those three vertices leads to a Model 4 quiver, illustrated in the bottom-right of Figure \ref{fig:dP3QuiverModels}.  The remaining two toric vertices yields Model 2 quivers after they are mutated.  One of these Model 2 quivers is the quiver previously visited since mutation is an involution.

Finally, in a Model 4 quiver, there are three toric vertices, which are graph isomorphic to one-another.  Mutation at either of them leads to a Model 3 quiver.  To summarize all of the adjacencies, we borrow the following figure from \cite[Figure 27]{franco_eager}.  See Figure \ref{fig:ModelConnections}.  We now focus on special examples of toric mutation sequences before returning to the general case.

\begin{figure}
\includegraphics[width=2.5in]{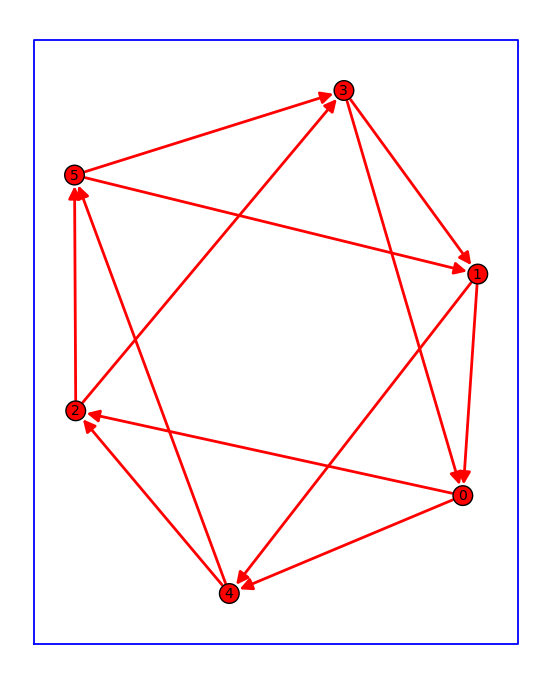}
\includegraphics[width=3.1in]{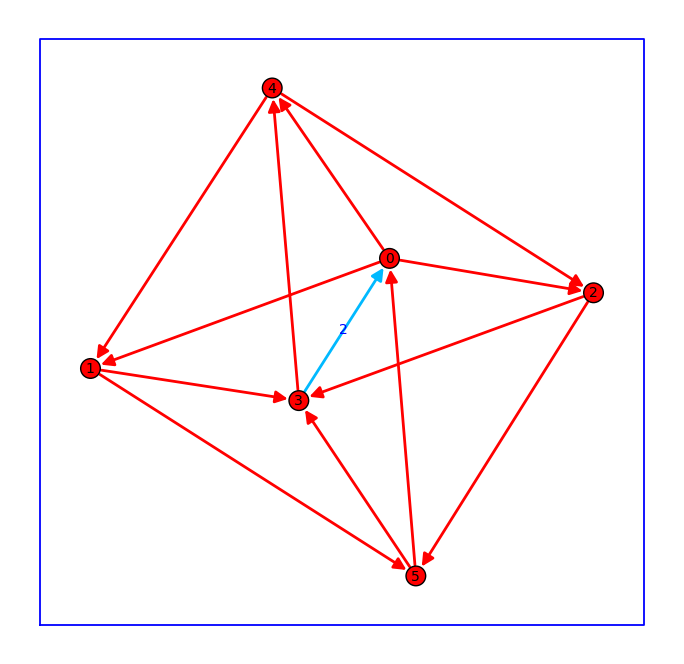}
\includegraphics[width=2.5in]{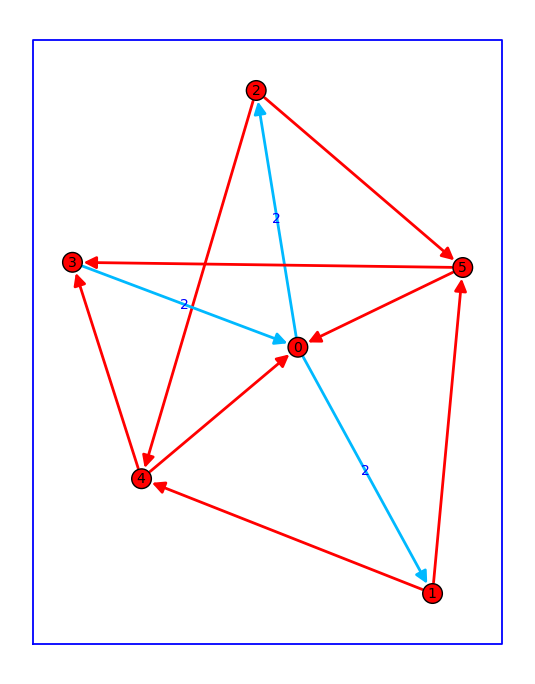}
\includegraphics[width=2.5in]{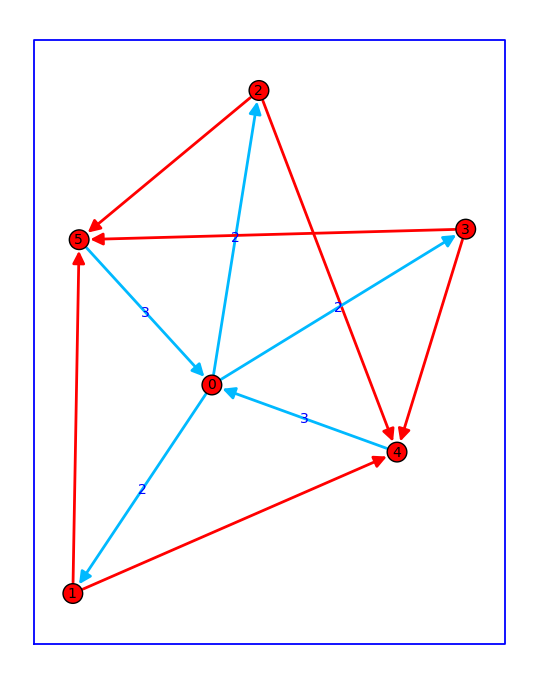}
\caption{Models 1, 2, 3, and 4 of the $dP_3$ quiver.  Two quivers are considered to be the same model if they are equivalent under (i) graph isomorphism and (ii) reversal of all edges.}  
\label{fig:dP3QuiverModels}
\end{figure}
\begin{figure}\centering
\includegraphics[width=12cm]{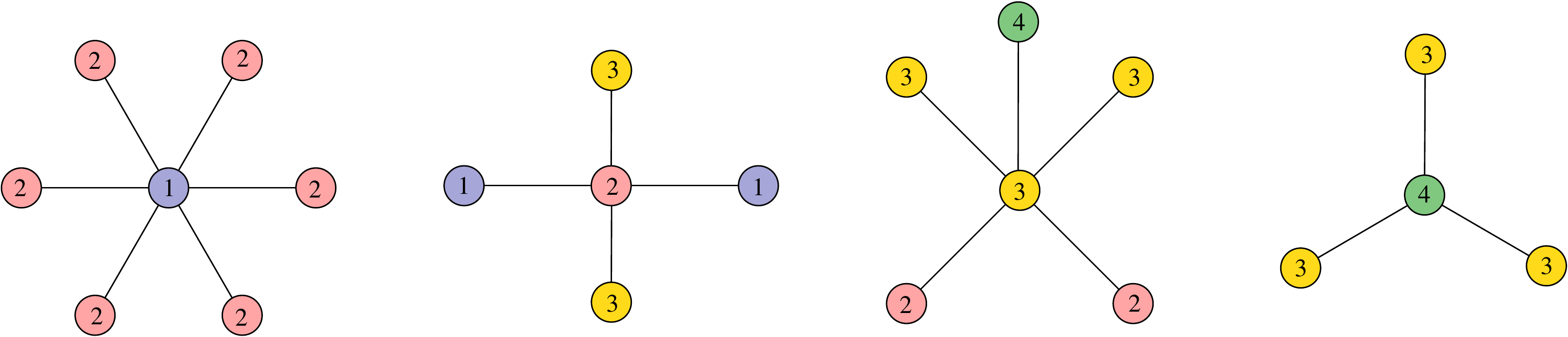}
\caption{Adjacencies between the different models.}
\label{fig:ModelConnections}
\end{figure}

\subsection{Generalized $\tau$-mutation Sequences}
\label{sec:tau}

In this subsection, we define a class of mutation sequences on $Q$, which we refer to as generalized $\tau$-mutation sequences. This extends the definition from \cite{LMNT}, which defined $\tau_1$, $\tau_2$, and $\tau_3$\footnote{Note: the notations for $\tau_i$ and $\tau_i'$-mutations are reversed as compared with \cite{LMNT} since in this paper, the sequences involving combinations of mutations and permutations are more central.}.

\begin{definition}
\label{def:tau0}
Define the following pairs of mutations on $Q$.
\begin{itemize}
\centering
\item[] $\tau_{1}'=\mu_{1} \circ \mu_{2}$
\item[] $\tau_{2}'=\mu_{3} \circ \mu_{4}$
\item[] $\tau_{3}'=\mu_{5} \circ \mu_{6}$
\item[] $\tau_{4}'= \mu_{1} \circ \mu_{4} \circ \mu_1 \circ \mu_5 \circ \mu_1$
\item[] $\tau_{5}'= \mu_{2} \circ \mu_{3} \circ \mu_2 \circ \mu_6 \circ \mu_2$
\end{itemize}

\end{definition}

Since antipodal vertices share no common edges, we observe that $\mu_{2i - 1}$ and $\mu_{2i}$ commute for $i \in \{1,2,3\}$.
Furthermore, for such $i$, the action of $\tau_{i}$ on the quiver exchanges the labels on vertices $2i-1$ and $2i$.  This motivates us to define
the following actions on cluster seeds which are slight variants of the $\tau'$-mutations.

\begin{definition}
\label{def:tau}
Define the following actions on $Q$.
\begin{itemize}
\centering
\item[] $\tau_{1}=\mu_{1} \circ \mu_{2} \circ (12)$
\item[] $\tau_{2}=\mu_{3} \circ \mu_{4} \circ (34)$
\item[] $\tau_{3}=\mu_{5} \circ \mu_{6} \circ (56)$
\item[] $\tau_{4}= \mu_{1} \circ \mu_{4} \circ \mu_1 \circ \mu_5 \circ \mu_1 \circ (145)$
\item[] $\tau_{5}= \mu_{2} \circ \mu_{3} \circ \mu_2 \circ \mu_6 \circ \mu_2 \circ (236)$.
\end{itemize}
where we apply a graph automorphism of $Q$ and permutation to the labeled seed after the sequence of mutations.
\end{definition}

One can then check that on the level of quivers and labeled seeds (i.e. ordered clusters), we have the following identities:
For all $i,j$ such that $1 \leq i \not = j \leq 3$
\begin{equation}
\label{eq: tau_relations}
\begin{split}
 \tau_1(Q) = \tau_2(Q) = \tau_3(Q) = \tau_4(Q) = \tau_5(Q) &= Q  \\
(\tau_{i})^{2} \{x_1,x_2\dots, x_6\} = (\tau_{4})^{2} \{x_1,x_2\dots, x_6\} = (\tau_{5})^{2} \{x_1,x_2\dots, x_6\}&= \{x_1,x_2\dots, x_6\} \\
(\tau_{i}\tau_{j})^{3} \{x_1,x_2\dots, x_6\}&= \{x_1,x_2\dots, x_6\}, \\
\tau_i \tau_4 \{x_1,x_2\dots, x_6\} &= \tau_4 \tau_i \{x_1,x_2\dots, x_6\}, \\
\tau_i \tau_5 \{x_1,x_2\dots, x_6\} &= \tau_5 \tau_i \{x_1,x_2\dots, x_6\}.
\end{split}
\end{equation}

Lastly, letting $\tau_4\circ \tau_5$ act on the labeled seed $\{x_1,x_2\dots, x_6\}$, we see that this sequence has infinite order.  From these relations, it follows that $\langle \tau_1,\tau_2,\dots, \tau_5\rangle$ generate a subgroup\footnote{We later show, see Remark \ref{Rem:unique}, that in fact there are no other relations and $\langle \tau_1,\tau_2,\dots, \tau_5\rangle$ generate the full reflection group of this type.  We do not require this equality for the remainder of Section \ref{Sec:Mutations}.} of the reflection group of type $\tilde{A}_2 \times I_\infty$ where $\tilde{A}_2$ is the affine symmetric group on $\{0,1,2\}$ and  $I_\infty$ is the infinite dihedral group.   

We define a \textbf{generalized} $\mathbf{\tau}$-\textbf{mutation sequence} $S$ to be a mutation sequence of the form $\tau_{a_1} \tau_{a_2} \ldots \tau_{a_k}$ with the $a_i$'s in $\{1,2,3,4,5\}$. 

\begin{definition} [\bf Ordered Cluster from a generalized $\tau$-mutation sequence]
Starting with an initial cluster of $[x_1,x_2,x_3,x_4,x_5,x_6]$, we let $Z^S = [z^S_1, z^S_2, z^S_3, z^S_4, z^S_5, z^S_6]$ denote the ordered cluster resulting from the generalized $\tau$-mutation sequence $S$.
\end{definition}

\subsection{Viewing Generalized $\tau$-mutation Sequences as Prism Walks} \label{sec:walks}

Before saying more about the cluster variables arising from a generalized $\tau$-mutation sequence, we develop a two-parameter, and later three-parameter, coordinate system motivated by the relations satisfied by the $\tau_i$'s.  In particular, we have the following.

\begin{remark} Since we intertwine the permutations with the mutations when applying a generalized $\tau$-mutation sequence, all of the $\tau_i$'s fix quiver $Q$.  It follows that  $Z^S$ is well-defined up to the $\tilde{A}_2 \times I_\infty$ relations of (\ref{eq: tau_relations}).
\end{remark}

Let $L^\Delta$ denote the square lattice triangulated with $45^\circ$-~$45^\circ$-~$90^\circ$ triangles, as in Figure \ref{fig:sqlattice}.  Note that $L^\Delta$ is isomorphic to the affine $\widetilde{A}_2$ Coxeter lattice, which we let $\tau_1$, $\tau_2$, and $\tau_3$ act on as simple reflections, equivalently as steps in an alcove walk as in \cite{Rou} (also compare with \cite[Figure 5]{LMNT}).
We let the triangle $\{(0,-1), (-1,0), (0,0)\}$ be the initial alcove and use the convention
that $\tau_1$ (resp. $\tau_2$, $\tau_3$) corresponds to a flip across the horizontal (resp. vertical, diagonal) edge of the initial alcove.  For the remaining alcoves, we obtain the correspondence between edges and $\tau_i$'s by noting that around each vertex in $L^\Delta$, the assignments alternate between a given $\tau_i$ and $\tau_j$ (for $i,j \in \{1,2,3\}$).  Using two sides of the triangle in the initial alcove allows us to uniquely extend the assignment to all other edges.

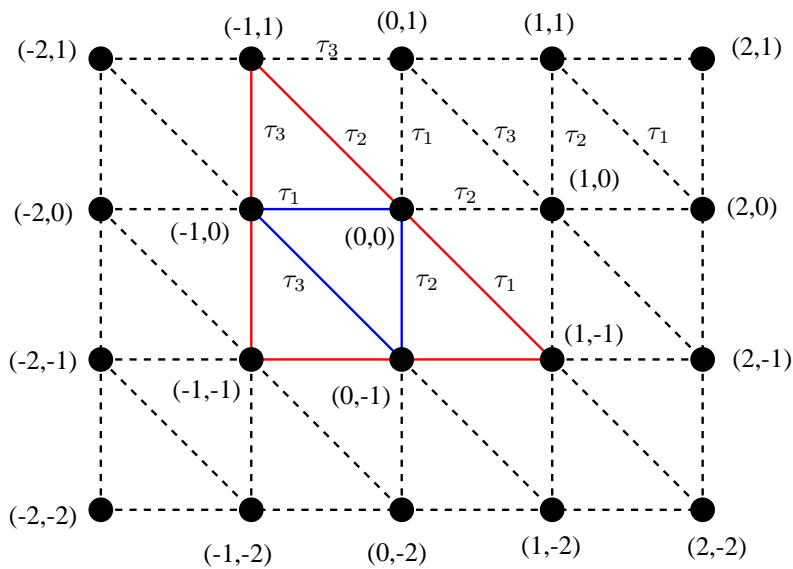
\begin{figure}
    \centering
    
\setlength{\unitlength}{3947sp}%
\begingroup\makeatletter\ifx\SetFigFont\undefined%
\gdef\SetFigFont#1#2#3#4#5{%
  \reset@font\fontsize{#1}{#2pt}%
  \fontfamily{#3}\fontseries{#4}\fontshape{#5}%
  \selectfont}%
\fi\endgroup%
      \scalebox{2.0}{\scalebox{0.5}{
  \begin{picture}(0,0)%
\includegraphics{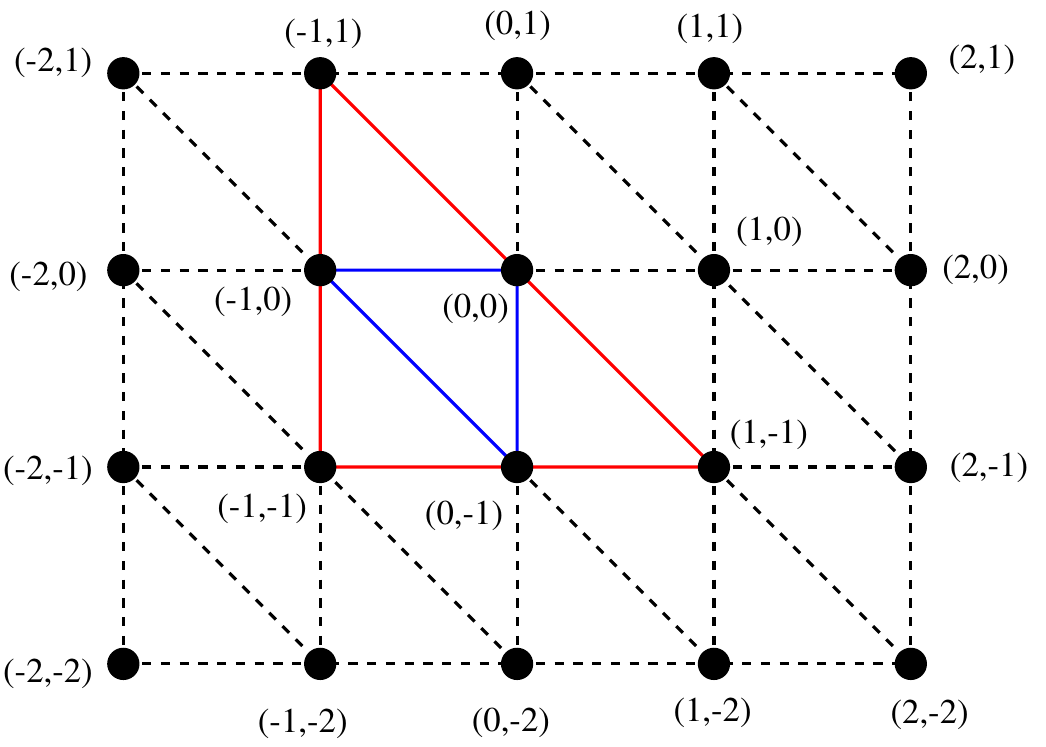}%
\end{picture}%

\begin{picture}(5005,3569)(826,-3337)
\put(2506,-1539){\makebox(0,0)[lb]{\smash{{\SetFigFont{10}{12.0}{\familydefault}{\mddefault}{\updefault}{ $\tau_3$}%
}}}}
\put(3339,-1531){\makebox(0,0)[lb]{\smash{{\SetFigFont{10}{12.0}{\familydefault}{\mddefault}{\updefault}{ $\tau_2$}%
}}}}
\put(2476,-991){\makebox(0,0)[lb]{\smash{{\SetFigFont{10}{12.0}{\familydefault}{\mddefault}{\updefault}{ $\tau_1$}%
}}}}
\put(3834,-1539){\makebox(0,0)[lb]{\smash{{\SetFigFont{10}{12.0}{\familydefault}{\mddefault}{\updefault}{ $\tau_1$}%
}}}}
\put(3583,-984){\makebox(0,0)[lb]{\smash{{\SetFigFont{10}{12.0}{\familydefault}{\mddefault}{\updefault}{ $\tau_2$}%
}}}}
\put(3834,-601){\makebox(0,0)[lb]{\smash{{\SetFigFont{10}{12.0}{\familydefault}{\mddefault}{\updefault}{ $\tau_3$}%
}}}}
\put(4276,-631){\makebox(0,0)[lb]{\smash{{\SetFigFont{10}{12.0}{\familydefault}{\mddefault}{\updefault}{ $\tau_2$}%
}}}}
\put(4794,-609){\makebox(0,0)[lb]{\smash{{\SetFigFont{10}{12.0}{\familydefault}{\mddefault}{\updefault}{ $\tau_1$}%
}}}}
\put(3313,-616){\makebox(0,0)[lb]{\smash{{\SetFigFont{10}{12.0}{\familydefault}{\mddefault}{\updefault}{ $\tau_1$}%
}}}}
\put(2896,-624){\makebox(0,0)[lb]{\smash{{\SetFigFont{10}{12.0}{\familydefault}{\mddefault}{\updefault}{ $\tau_2$}%
}}}}
\put(2386,-594){\makebox(0,0)[lb]{\smash{{\SetFigFont{10}{12.0}{\familydefault}{\mddefault}{\updefault}{ $\tau_3$}%
}}}}
\put(2716,-61){\makebox(0,0)[lb]{\smash{{\SetFigFont{10}{12.0}{\familydefault}{\mddefault}{\updefault}{ $\tau_3$}%
}}}}
\end{picture}    
      }}
	\caption{The lattice $L^\Delta$ with the initial alcove labeled as $(0,-1), (-1,0)$, and $(0,0)$.  The three $\tau$-mutations $\tau_1$, $\tau_2$, $\tau_3$ correspond to the directions of the triangular flips.}  
	\label{fig:sqlattice}
\end{figure}

Using (\ref{eq: tau_relations}), any generalized $\tau$-mutation sequence $S$ can be written as $S=S_1S_2$, where $S_1$ consists entirely of $\tau_1$'s, $\tau_2$'s, and $\tau_3$'s (i.e. a $\tau$-mutation sequence as defined in \cite{LMNT}), and $S_2$ is an alternating product of $\tau_4$ and $\tau_5$.  For such an $S_1$, we associate an {\bf alcove walk} by starting in the initial alcove for $S_1 = \emptyset$ and then for each $\tau_i$ in $S_1$ (from left-to-right), we apply the associated reflection, which yields one of the three neighboring alcoves.  The following remark is easy to verify and will be useful for the descriptions of triangular flips we use below.

\begin{remark} \label{rem:Delta} The lattice $L^\Delta$ consists of two orientations of triangles, a NE-pointing triangle, $[(i,j), (i-1,j+1), (i,j+1)]$, and a SW-pointing one, $[(i,j),(i-1,j+1),(i-1,j)]$.  Further, the image of $[(i_1,j_1),(i_2,j_2),(i_3,j_3)]$ under the map $\alpha : \mathbb{Z}^2 \to \mathbb{Z}/3\mathbb{Z}$ defined as $(I,J) \mapsto I - J \mod 3$ is a permutation of $[1,2,3]$.  See Figure \ref{fig:3levels}.  In fact, $L^\Delta$ is completely and disjointly tiled by the set of NE-pointing triangles whose NE vertex $(I,J)$ satisfies $\alpha(I,J) = 3$.  Such triangles are translates of the triangle shown shaded.
\end{remark}

\begin{remark} \label{rem:order123} For convenience, consider the case that $|S_1|$ is even so that we have $T_{S_1} = [(i,j), (i-1,j+1), (i,j+1)]$, is a NE-pointing triangle, as in Remark \ref{rem:Delta}.  Without loss of generality, we also focus on the case $i - j \equiv 1 \mod 3$.   If we apply triangular flips in the three possible directions, we obtain adjacent triangles $T_{S_1\tau_1}$, $T_{S_1\tau_2}$, and $T_{S_1\tau_3}$ where we have exchanged one vertex for a new one: $(i,j) \leftrightarrow (i-1,j+2)$, $(i-1,j+1) \leftrightarrow (i+1,j)$, or $(i,j+1) \leftrightarrow (i-1,j)$, keeping the order otherwise the same.  By a case-by-case analysis, we see that the values of $\alpha(T_{S_1\tau_r})$ continue to be  $[1,2,3]$ in that order, see Figures \ref{fig:3levels} and \ref{fig:flips}.
\end{remark}

\begin{figure}
\includegraphics[width=3in]{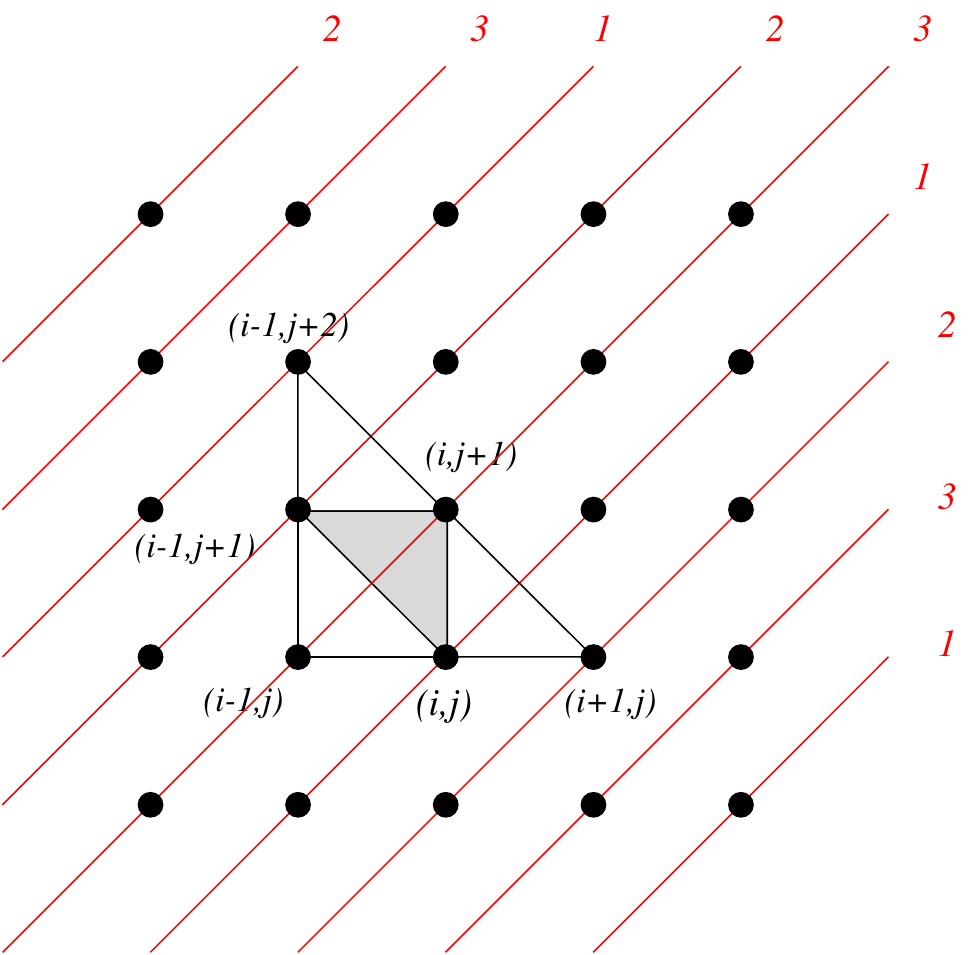}
\caption{Levels of $i-j \mod 3$ in Lattice $L^\Delta$ illustrated.}
\label{fig:3levels}
\end{figure}

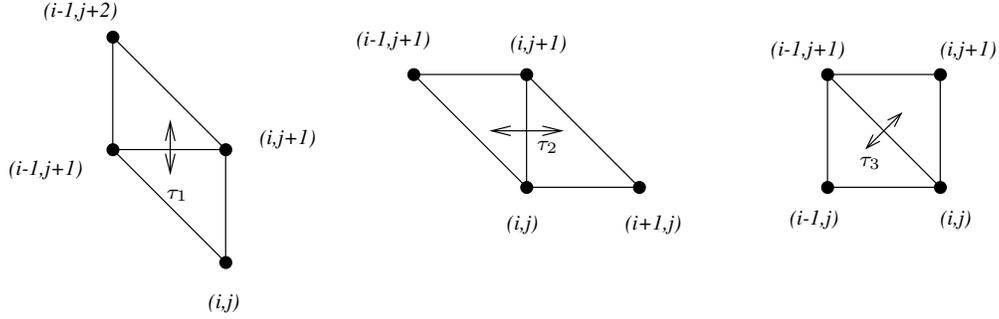
\begin{figure}
\begin{picture}(0,0)%
\includegraphics{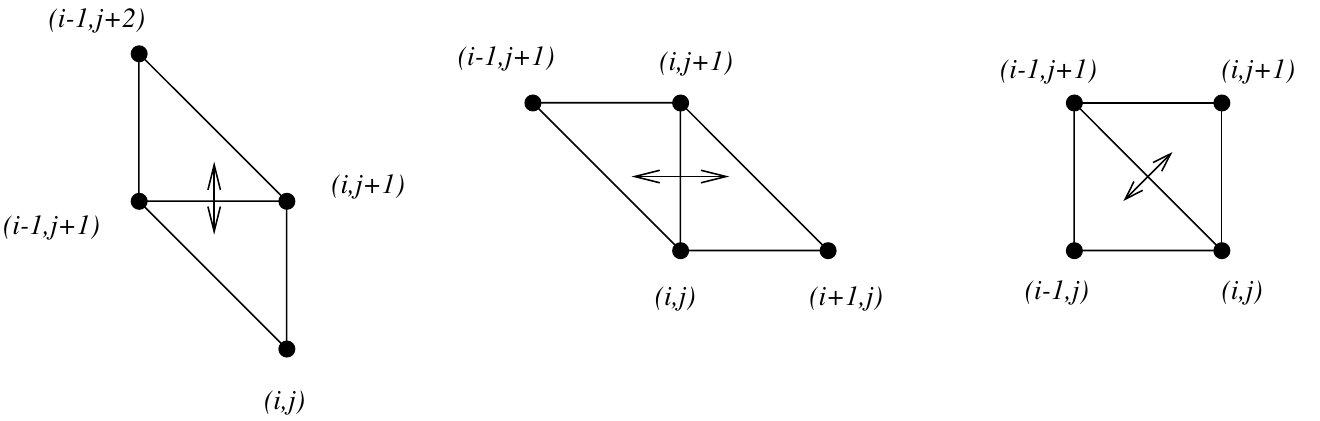}%
\end{picture}%
\setlength{\unitlength}{3947sp}%
\begingroup\makeatletter\ifx\SetFigFont\undefined%
\gdef\SetFigFont#1#2#3#4#5{%
  \reset@font\fontsize{#1}{#2pt}%
  \fontfamily{#3}\fontseries{#4}\fontshape{#5}%
  \selectfont}%
\fi\endgroup%
\begin{picture}(6286,1954)(813,-1579)
\put(6171,-629){\makebox(0,0)[lb]{\smash{{\SetFigFont{8}{9.6}{\rmdefault}{\mddefault}{\itdefault}{$\tau_3$}%
}}}}
\put(1819,-834){\makebox(0,0)[lb]{\smash{{\SetFigFont{8}{9.6}{\rmdefault}{\mddefault}{\itdefault}{$\tau_1$}%
}}}}
\put(4144,-531){\makebox(0,0)[lb]{\smash{{\SetFigFont{8}{9.6}{\rmdefault}{\mddefault}{\itdefault}{$\tau_2$}%
}}}}
\end{picture}
\caption{The three possible triangular flips illustrated.} \label{fig:flips}
\end{figure}

With coordinates now described for $\tau$-mutation sequences $S_1$ consisting of $\tau_1$'s, $\tau_2$'s and $\tau_3$'s, we now consider a generalized $\tau$-mutation sequence $S=S_1S_2$, where $S_2$ consists of $\tau_4$'s and $\tau_5$'s.  For such a generalized $\tau$-mutation sequence we associate a corresponding triangular prism in $\mathbb{Z}^3$.  By the relations (\ref{eq: tau_relations}) and the equivalence $L^\Delta \times \mathbb{Z} \cong \mathbb{Z}^3$, the following definition is well-defined.  In Section \ref{sec:subgraphs}, these prisms will be used to define a corresponding $6$-tuple of graphs, which will serve as our combinatorial interpretation for cluster variables.

\begin{definition} [\bf Prism $\Delta^S$ from a generalized $\tau$-mutation sequence $S$] \label{def:prism}
Factor $S$ as $S_1S_2$ as mentioned above, and let the triple
$[(i_1,j_1),(i_2,j_2),(i_3,j_3)]$ denote the triangle of $L^\Delta$ reached after applying the alcove walk associated to $S_1$ starting from the initial alcove.  We order this triple such that that vertices satisfy $i_r- j_r \equiv r \mod 3$.  Then
$$\Delta^{S} = \Delta^{S_1S_2} = [(i_1,j_1,k_1),(i_1,j_1,k_2),(i_2,j_2,k_2),(i_2,j_2,k_1),(i_3,j_3,k_1),(i_3,j_3,k_2)]$$
in $\mathbb{Z}^3 \cong L^\Delta \times \mathbb{Z}$ where we let $\{k_1,k_2\} = \{|S_2|, |S_2|+1\}$ (resp. $\{-|S_2|,-|S_2|+1\}$) if $S_2$ starts with $\tau_5$ (resp. $\tau_4$).  Here the correspondence between these sets depends on the parity of $|S_2|$: we let $k_1 = \pm |S_2|$ if $|S_2|$ is odd and $k_2 = \pm |S_2|$ otherwise.  
\end{definition}

In particular, notice that by Definition \ref{def:prism}, $$\Delta^\emptyset = [(0,-1,1), (0,-1,0), (-1,0,0), (-1,0,1), (0,0,1), (0,0,0)].$$

Since clusters are well-defined up to the relations (\ref{eq: tau_relations}), it follows that we can associate the ordered cluster $Z^S = [z_1^S,z_2^S,\dots, z_6^S]$, obtained from the generalized $\tau$-mutation sequence $S$, to the prism
$\Delta^S$.
Since both of these lists are ordered, this induces a mapping from the set of lattice points of $\mathbb{Z}^3$ to the set of cluster variables reachable via a generalized $\tau$-mutation sequence.
Accordingly, we use the notation $z_i^{j,k}$ to denote the cluster variable corresponding to the lattice point $(i,j,k)$, noting that we have not yet shown that this mapping is an injection.  
We will show injectivity in Section \ref{sec:explicit} but our arguments in Section \ref{sec:gentoric} do not require it.

\begin{figure}
\includegraphics[width=5in]{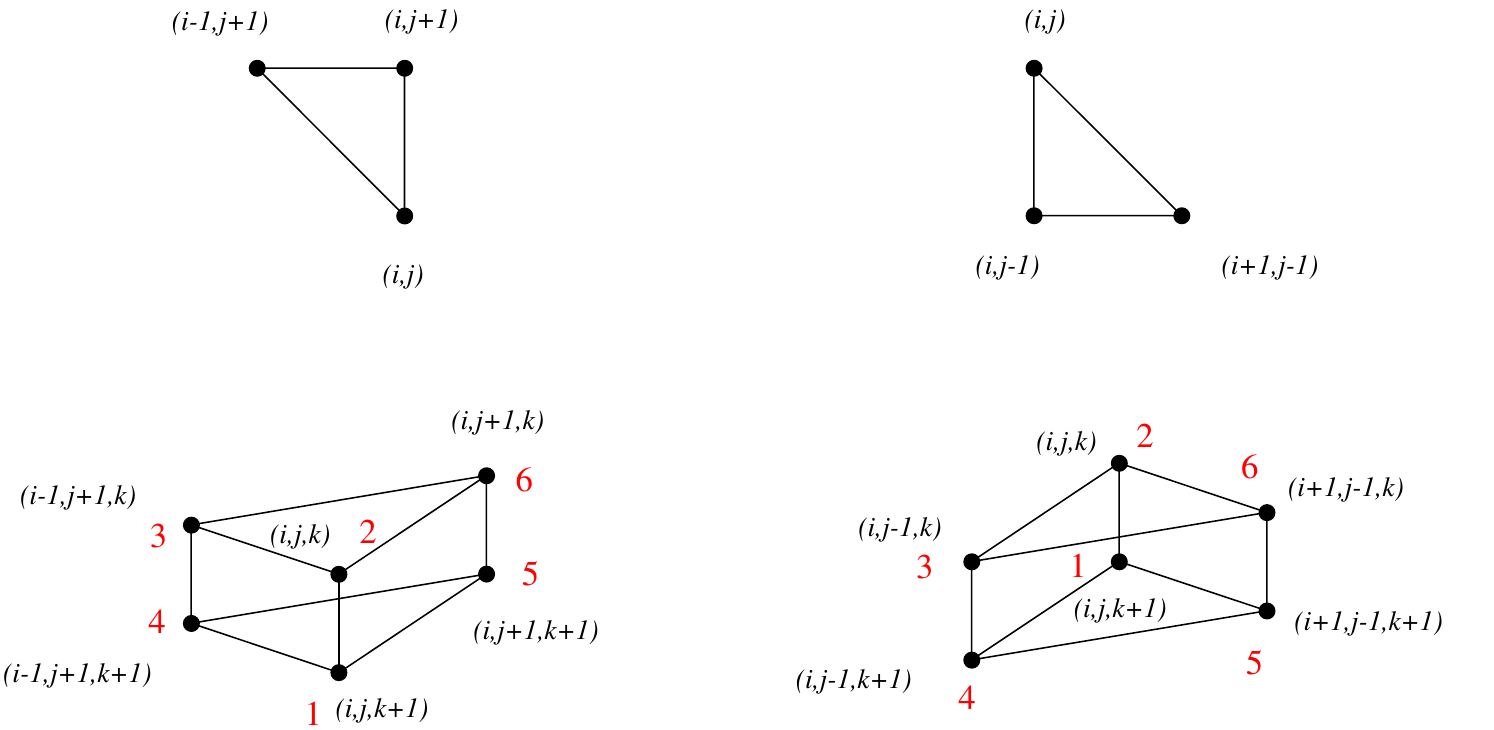}
\caption{As in Remark \ref{rem:Delta}, we have two possible orientations of a triangle, i.e. alcove, in $L^\Delta$.  We let these be shorthand for
the corresponding triangular prisms in $L^\Delta \times \mathbb{Z} \cong \mathbb{Z}^3$ reached by applying a $\tau$-mutation sequence $S_1$ involving
only $\tau_1$'s, $\tau_2$'s, and $\tau_3$'s.}
\label{fig:ModelI}
\end{figure}

\begin{figure}
\includegraphics[width=5.5in]{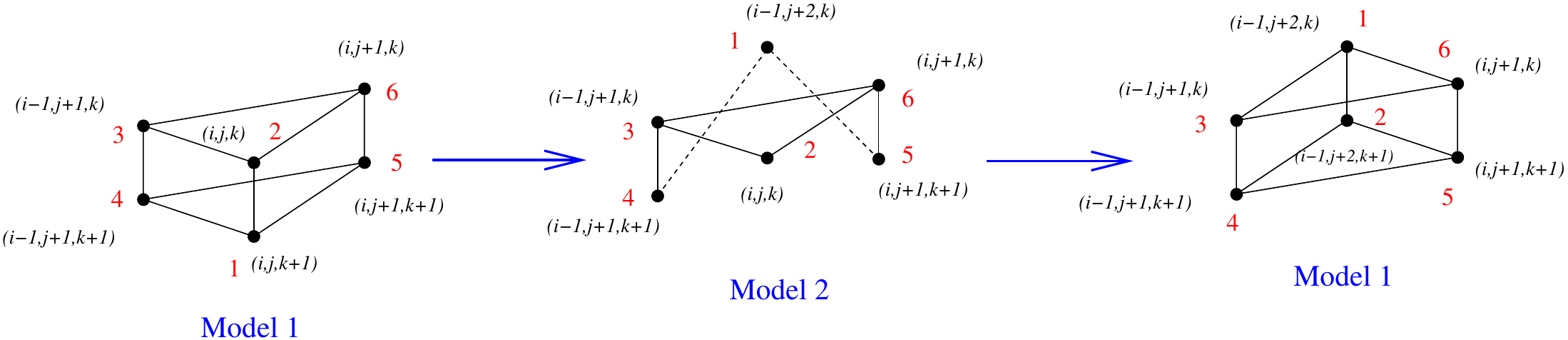}
\caption{Illustrating the prism as it is transformed by mutations $\mu_1$ and $\mu_2$, resulting in the image of $\tau_1' = \mu_1\circ\mu_2$.  Applying permutation $(12)$ yields a triangular prism of the above form.}
\label{fig:ModelII}
\end{figure}

\begin{figure}
\includegraphics[width=5.5in]{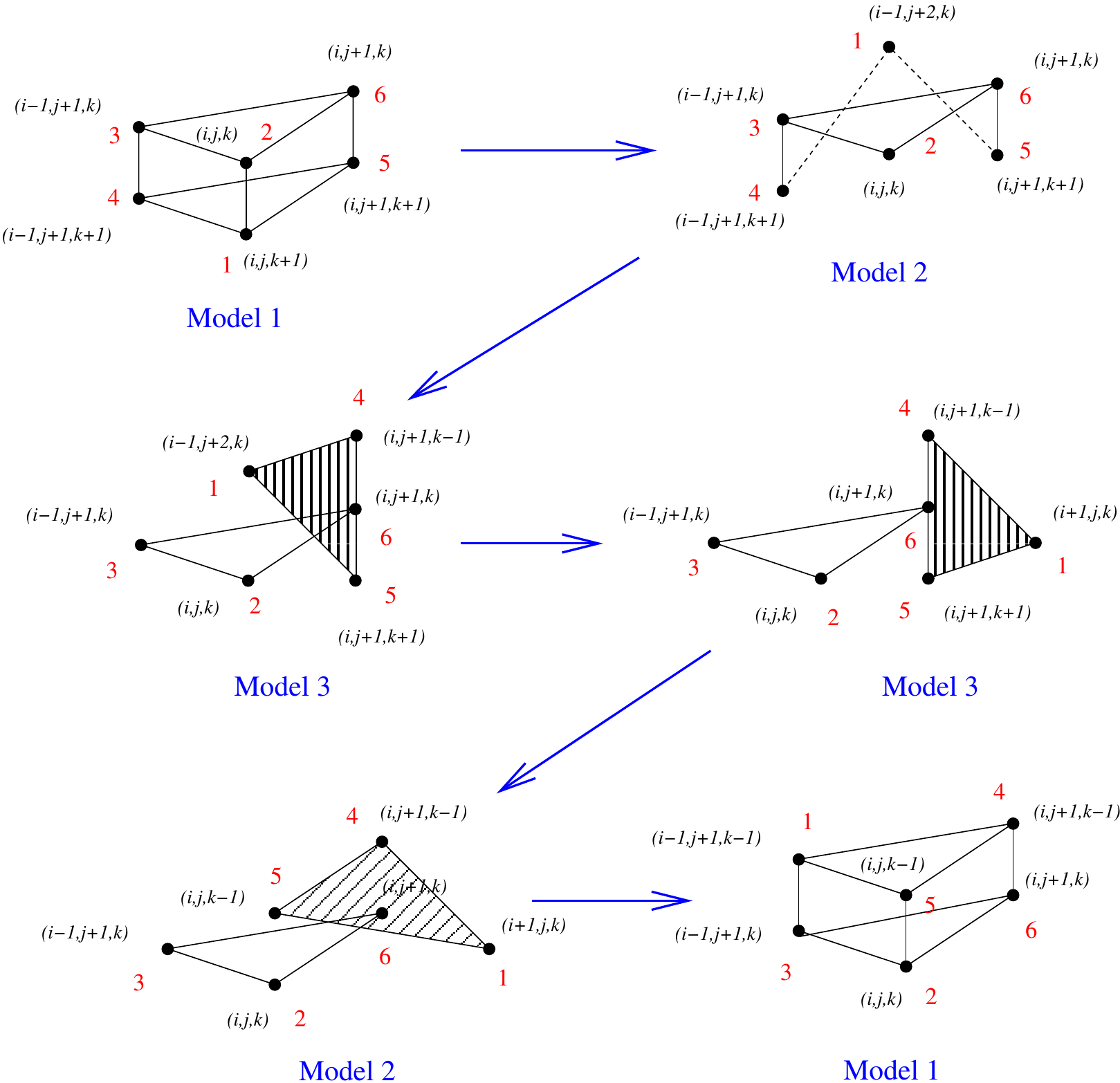}
\caption{We illustrate the $\mathbb{Z}^3$-transformations induced by the mutation sequence $\tau_4' = \mu_1\mu_4\mu_1\mu_5\mu_1$.  Applying permutation $(145)$ to this result
yields a triangular prism of the standard form but with a decreased third coordinate.}
\label{fig:ModelIII}
\end{figure}

\subsection{General Toric Mutation Sequences geometrically} \label{sec:gentoric}

As described at the end of the last section, the set of clusters that are reachable via a generalized $\tau$-mutation sequence can be modeled as prisms in the $\mathbb{Z}^3$ lattice.  In particular, the effect of a single $\tau_i$ corresponds to a glide reflection or translation on one of the two prisms illustrated in Figure \ref{fig:ModelI}.  Using these two types of transformations, we can move the initial prism, $[(0,-1,1), (0,-1,0), (-1,0,0), (-1,0,1), (0,0,1), (0,0,0)]$, to any other isometric prism in $\mathbb{Z}^3$ up to reflection or translation.  The natural ensuing question is what about the set of clusters that are reachable via other toric mutation sequences?  We answer this question with the following result.

\begin{lemma} \label{lem:gentoric}
For the cluster algebra associated to the $dP_3$ Quiver, the set of toric cluster variables, i.e. those reachable via a toric mutation sequence, coincides with the set of cluster variables that are reachable by generalized $\tau$-mutation sequences.  Consequently, we obtain a parameterization of such toric cluster variables by $\mathbb{Z}^3$.
\end{lemma}

\begin{proof} We prove this result by breaking up $\tau$-mutation sequences into their individual components.  From the above analysis, i.e. as induced by the relations (\ref{eq: tau_relations}), we know how each of the $\tau_i$'s transforms one cluster to another in terms of the $\mathbb{Z}^3$-parameterization.  However, since most of the individual mutations appearing in a $\tau_i$ appear therein uniquely, we may immediately deduce in most cases how an individual mutation transforms one cluster variable to another in terms of this parameterization as well.  To start with, in Figure \ref{fig:ModelII}, we illustrate how the components of $\tau_1$ act on $\mathbb{Z}^3$.  By symmetry, the components of $\tau_2$ and $\tau_3$ induce rotated versions of these moves. 

Next we present Figure \ref{fig:ModelIII}, which illustrates the intermediate steps that together yield the vertical translation induced by $\tau_4$. The validity of the second (resp. third) configuration in this figure follows from the above description of how the completed mutation sequences $\tau_1$ (resp. $\tau_4$) affects a cluster.  Similarly, if we run the $\tau_4$ mutation sequence backwards, the validity of the fifth and fourth configurations follow analogously.   By similar logic, the sequence $\tau_5$ would induce the opposite composition of moves.

\begin{figure}
\includegraphics[width=5in]{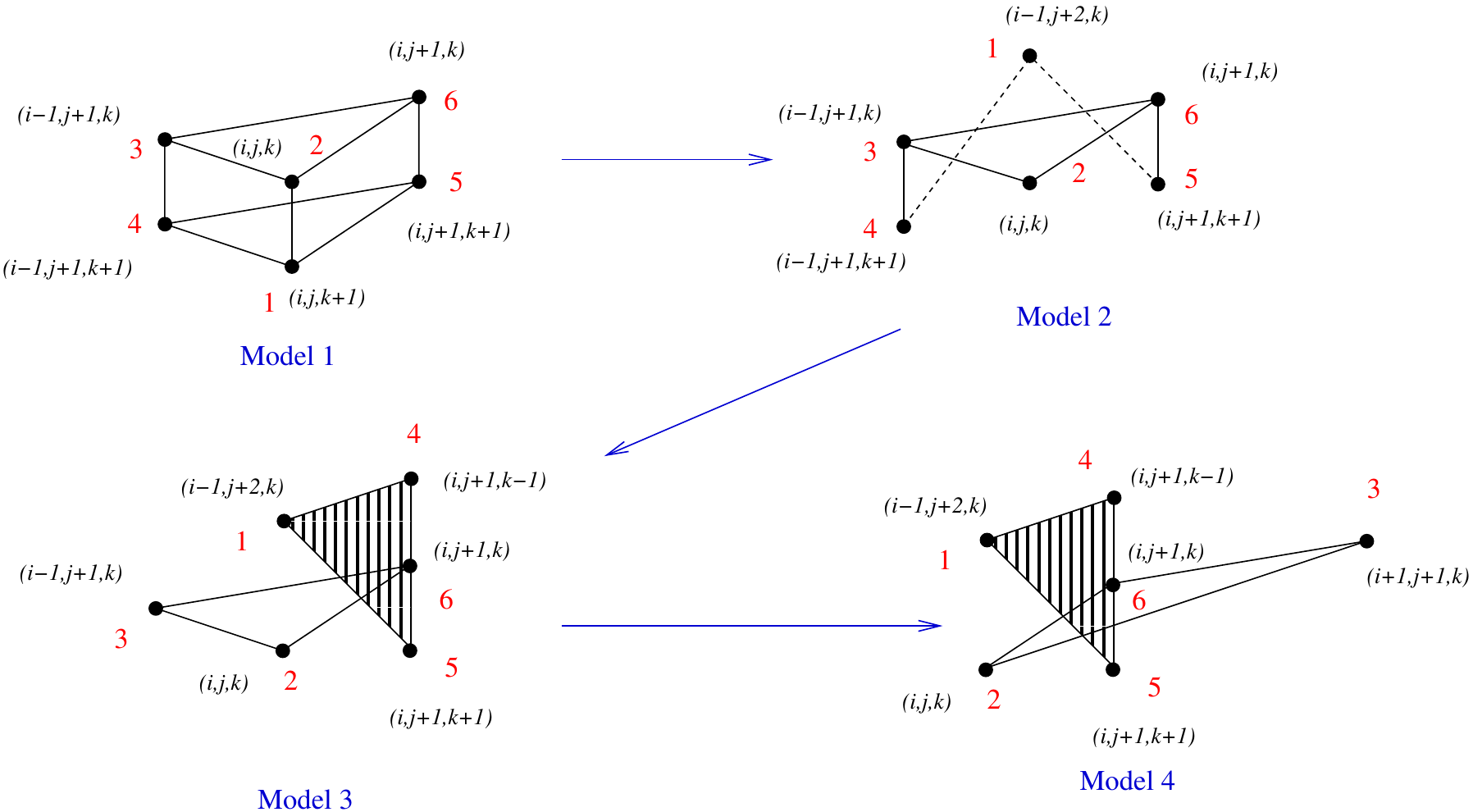}
\caption{The toric mutation sequence $\mu_1\mu_4\mu_3$ does not correspond to a product of $\tau_i$'s but illustrates a toric mutation between Models 3 and 4.}
\label{fig:ModelsGeom}
\end{figure}

By comparing these figures with the description of toric mutations in Section \ref{Sec:Toric}, we observe that we can obtain $\mathbb{Z}^3$-coordinates for all cluster variables obtained from a toric mutation sequence that passes through quivers of Model 1, Model 2, or Model 3 type.  In particular, up to geometric rotations or reflections, any single toric mutation starting from (1) a Model 1 quiver corresponds to the local transformation illustrated on the left-hand-side of Figure \ref{fig:ModelII}; (2) a Model 2 quiver corresponds either to the local transformation illustrated on the right-hand-side of Figure \ref{fig:ModelII} or the diagonal arrows of Figure \ref{fig:ModelIII}; (3) a Model 3 quiver corresponds to either the middle arrow of Figure \ref{fig:ModelIII}, the diagonal arrows on either side in Figure \ref{fig:ModelIII}, or is handled in the ensuing paragraph.

To complete this picture to include all toric mutation sequences, even those that visit a Model 4 type quiver, we include Figure \ref{fig:ModelsGeom} that shows the $\mathbb{Z}^3$-transformations induced by the toric mutations between Model 3 and Model 4.  In particular, the content of Figure \ref{fig:ModelsGeom} is that if we start with the initial cluster $\{x_1,x_2,\dots, x_6\}$ and mutate by $\mu_1$, $\mu_4$, then $\mu_3$ we obtain a new cluster whose third cluster variable agrees with the second cluster variable in the cluster after the generalized $\tau$-mutation sequence $\tau_2\tau_1$.  The resulting quiver is of Model 4 type and the resulting cluster variables are still parameterized by $\mathbb{Z}^3$.  Any toric mutation of a Model 4 quiver leads back to a Model 3 one, and as described in Section \ref{Sec:Toric}, all toric mutations between Model 3 and Model 4 are equivalent.  Thus up to symmetry, we have illustrated all possible single step toric mutations among any of the models of the $dP_3$ quiver and hence for any toric mutation sequence rather than just the generalized $\tau$-mutation sequences.
\end{proof}

\begin{remark} \label{rem:RG}  In Section 4 of \cite{FHU}, Franco, Hanany, and Uranga discussed certain mutation sequences, called {\bf duality cascades} in their language, of the $dP_3$ quiver that are significant for geometric and physical reasons.  After comparing their work to ours, we realized that their cascades are essentially the $\tau_i$'s defined above.  It is of interest that our combinatorial motivation, i.e. we looked for a family of mutation sequences whose combinations would satisfy Coxeter relations, aligned with their geometric objectives of understanding the Renormalization Group (RG) flow.  The second author has started investigating other examples such as $dP_2$ and $Y_{p,q}$'s with Franco to see how the use of fractional branes and beta functions could lead to other families of mutation sequences satisfying Coxeter relations.
\end{remark}

\begin{remark} \label{rem:zono} In Sections 3.2.2 and 8.3 of \cite{franco_eager}, Eager and Franco discuss a possible coordinate system for working with mutation sequences of quivers from brane tilings.  They are motivated by previous work in tilting theory \cite{BH} and the multi-dimensional octahedron recurrence \cite{HK,HS,speyer}, and sketch examples of coordinates for $dP_2$ and $dP_3$.  In their coordinate system, they describe certain duality cascades that act as translations of a zonotope.

As described in Section \ref{sec:walks}, the importance of the generalized $\tau$-mutation sequences is that out of the space of all possible toric mutation sequences, the generalized $\tau$-mutation sequences are the ones that map a cluster that looks like a prism (which is an example of a zonotope) to a translation, or a glide reflection, of the same prism.  Hence, up to a coordinate change, our generalized $\tau$-mutation sequences of the present paper should agree with the duality cascades described in \cite{franco_eager}.  With Eager, the second author is currently exploring the possibility of explicitly defining these coordinate systems and zonotopes for other examples.
\end{remark}

\begin{remark} Unpublished work by the researchers Andr\'e Henriques, David Speyer, and Dylan Thurston on the multi-dimensional octahedron recurrence continues the point of view from \cite{HK,HS,speyer} and would provide an alternative construction for cluster variable coordinates that would be a variant of the ones discussed in this section.  We thank David Speyer for enlightening conversations on this topic.  More comments comparing our approaches to theirs are included in Section \ref{sec:open}.
\end{remark}

\section{Explicit formula for cluster variables} \label{sec:explicit}

By Lemma \ref{lem:gentoric}, we have a surjection between the lattice points of $\mathbb{Z}^3$ and the cluster variables reachable from a general toric mutation sequence.  Moreover, a general toric mutation sequence $S$, applied to the initial cluster, reaches a cluster $Z^S$ that either may be modeled as a prism $\Delta^S$ in $\mathbb{Z}^3$ or is at most three mutation steps away from such a cluster.  In this section, we continue to use the notation $z_i^{j,k}$ to denote the cluster variable corresponding to lattice point $(i,j,k)$.  With this notation in mind, we come to our first main result.

\begin{theorem} \label{thm:explicit}
Let $(i,j,k) \in \mathbb{Z}^3$ and $z_i^{j,k}$ be the associated cluster variable (reachable by a toric mutation sequence) as described above.

Define $A= \frac{x_3x_5+x_4x_6}{x_1x_2}$, $B=\frac{x_1x_6+x_2x_5}{x_3x_4}$, $C=\frac{x_1x_3+x_2x_4}{x_5x_6}$, $D=\frac{x_1x_3x_6+x_2x_3x_5+x_2x_4x_6}{x_1x_4x_5}$, and
$E = \frac{x_2x_4x_5 + x_1x_3x_5 + x_1x_4x_6}{x_2x_3x_6}$.
Then $z_i^{j,k}$ is given by the Laurent polynomial
{\large $$x_r ~~A^{\lfloor \frac{(i^2+ij+j^2+1) + i + 2j}{3}\rfloor} ~
B^{\lfloor \frac{(i^2+ij+j^2+1) + 2i + j}{3}\rfloor}~
C^{\lfloor \frac{i^2+ij+j^2+1}{3}\rfloor}~
D^{\lfloor \frac{(k-1)^2}{4}\rfloor}~
E^{\lfloor \frac{k^2}{4}\rfloor}$$} where

$r = 1$ if $2(i-j) + 3k \equiv 5$, ~ $r = 2$ if $2(i-j) + 3k \equiv 2$, ~ $r = 3$ if $2(i-j) + 3k \equiv 4$,

$r = 4$ if $2(i-j) + 3k \equiv 1$, ~ $r = 5$ if $2(i-j) + 3k \equiv 3$, ~ $r = 6$ if $2(i-j) + 3k \equiv 0$ working modulo $6$.  In particular, the variable $x_r$ is uniquely determined by the values of $(i-j)$ modulo $3$ and $k$ modulo $2$.
\end{theorem}

\begin{remark}
This nontrivial correspondence between the values of $r$ and $2(i-j)+3k \mod 6$ comes from our cyclic ordering
of the Model 1 $dP_3$ quiver in counter-clockwise order given in Figure \ref{fig:quiv_brane}.  In particular, as we rotate from vertex $r$ to $r'$ in clockwise order, the corresponding value of $2(i-j)+3k \mod 6$ increases by $1$ (circularly).
\end{remark}

\begin{remark} \label{Rem:unique}
As an application of Theorem \ref{thm:explicit}, observe that $z_i^{j,k} \not = z_{i'}^{j',k'}$ unless $i=i'$, $j=j'$, and $k=k'$.  This follows since each of these algebraic expressions are distinct for different $(i,j,k)\in \mathbb{Z}^3$.
\end{remark}

\noindent Before proving the result we note the following identities that will aid us in the proof.

\begin{lemma} \label{lem:ABCDE}
Let $c(i,j) = i^2 + ij + j^2 + 1$, $a(i,j) = c(i,j) + i + 2j$, and $b(i,j) = c(i,j) + 2i+j$.  Then we have the identities
{\footnotesize \begin{eqnarray} \label{eq:thirdA} \bigg\lfloor \frac{a(i,j)}{3}\bigg\rfloor + \bigg\lfloor \frac{a(i-1,j+2)}{3}\bigg\rfloor &=& \bigg\lfloor \frac{a(i-1,j+1)}{3}\bigg\rfloor + \bigg\lfloor \frac{a(i,j+1)}{3}\bigg\rfloor + \chi(i - j \equiv 1 \mod 3),  \\
\label{eq:thirdB} \bigg\lfloor \frac{b(i,j)}{3}\bigg\rfloor + \bigg\lfloor \frac{b(i-1,j+2)}{3}\bigg\rfloor &=& \bigg\lfloor \frac{b(i-1,j+1)}{3}\bigg\rfloor + \bigg\lfloor \frac{b(i,j+1)}{3}\bigg\rfloor + \chi(i - j \equiv 2 \mod 3), \\
\label{eq:thirdC} \bigg\lfloor \frac{c(i,j)}{3}\bigg\rfloor + \bigg\lfloor \frac{c(i-1,j+2)}{3}\bigg\rfloor &=& \bigg\lfloor \frac{c(i-1,j+1)}{3}\bigg\rfloor + \bigg\lfloor \frac{c(i,j+1)}{3}\bigg\rfloor + \chi(i - j \equiv 0 \mod 3), \\
\label{eq:fourthD} \bigg\lfloor \frac{k^2}{4}\bigg\rfloor +  \bigg\lfloor \frac{(k-2)^2}{4}\bigg\rfloor &=&
\bigg\lfloor \frac{(k-1)^2}{4}\bigg\rfloor + \bigg\lfloor \frac{(k-1)^2}{4}\bigg\rfloor + \chi(k \mathrm{~is~even}),  \mathrm{~and} \\
\label{eq:fourthE}\bigg\lfloor \frac{(k+1)^2}{4}\bigg\rfloor +  \bigg\lfloor \frac{(k-1)^2}{4}\bigg\rfloor &=&
\bigg\lfloor \frac{k^2}{4}\bigg\rfloor + \bigg\lfloor \frac{k^2}{4}\bigg\rfloor + \chi(k \mathrm{~is~odd})
\end{eqnarray} }
where $\chi(S)$ equals $1$ when statement $S$ is true and equals $0$ otherwise. The identities
{\footnotesize \begin{eqnarray} \label{eq:thirdAA} \bigg\lfloor \frac{b(i-1,j+1)}{3}\bigg\rfloor + \bigg\lfloor \frac{b(i+1,j)}{3}\bigg\rfloor &=& \bigg\lfloor \frac{b(i,j)}{3}\bigg\rfloor + \bigg\lfloor \frac{b(i,j+1)}{3}\bigg\rfloor + \chi(i - j \equiv 1 \mod 3) \\
\label{eq:thirdAAA} \bigg\lfloor \frac{c(i,j+1)}{3}\bigg\rfloor + \bigg\lfloor \frac{c(i-1,j)}{3}\bigg\rfloor &=& \bigg\lfloor \frac{c(i,j)}{3}\bigg\rfloor + \bigg\lfloor \frac{c(i-1,j+1)}{3}\bigg\rfloor + \chi(i - j \equiv 1 \mod 3) \end{eqnarray} }

\vspace{-1em}\noindent and analogous versions for $i-j \equiv 2$ or $0 \mod 3$ also hold if all instances of $b(i,j)$ (resp. $c(i,j)$) are switched with $c(i,j)$ or $a(i,j)$ (resp. $a(i,j)$ or $b(i,j)$).
 \end{lemma}

\begin{proof}

Firstly we note that  $a(i,j) \equiv  \begin{cases} 1 \mod 3 \mathrm{~if~}i-j \equiv 0 \mathrm{~or~} 2\mod 3 \\ 0 \mod 3 \mathrm{~if~}i - j \equiv 1\end{cases}.$

Similarly, we see that $b(i,j) \equiv  \begin{cases} 1 \mod 3 \mathrm{~if~}i-j \equiv 0 \mathrm{~or~} 1\mod 3 \\ 0 \mod 3 \mathrm{~if~}i - j \equiv 2\end{cases}$

and $c(i,j) \equiv \begin{cases} 1 \mod 3 \mathrm{~if~}i-j \equiv 0 \mod 3 \\ 2 \mod 3 \mathrm{~if~}i-j \equiv 1 \mathrm{~or~} 2\mod 3\end{cases}.$

It is easy to verify also that
\begin{equation}
\label{eq:IdA} a(i,j) + a(i-1,j+2) - a(i-1,j+1) - a(i,j+1) = 1
\end{equation}
and we get identical equalities for $b(i,j)$ and $c(i,j)$.
Subtracting the floor functions appearing on the RHS of (\ref{eq:thirdA}) from those on the LHS, we obtain
{\small \begin{eqnarray*}\frac{a(i,j)}{3} + \frac{a(i-1,j+2)}{3} &-& \frac{a(i-1,j+1)-1}{3} - \frac{a(i,j+1)-1}{3} \mathrm{~~when~}i - j \equiv 1 \mathrm{~mod~3}, \\
\frac{a(i,j)-1}{3} + \frac{a(i-1,j+2)-1}{3} &-& \frac{a(i-1,j+1)-1}{3} - \frac{a(i,j+1)}{3}   \mathrm{~~when~}i - j \equiv 2 \mathrm{~mod~3,~and} \\
\frac{a(i,j)-1}{3} + \frac{a(i-1,j+2)-1}{3} &-& \frac{a(i-1,j+1)}{3} - \frac{a(i,j+1)-1}{3} \mathrm{~~when~}i - j \equiv 0 \mathrm{~mod~3}.
\end{eqnarray*}}

Using identity (\ref{eq:IdA}), the result is $\chi(i-j \equiv 1 \mod 3)$ exactly as desired.
The proof of identity (\ref{eq:thirdB}) is essentially identical.  In the case of (\ref{eq:thirdC}), we see
{\small \begin{eqnarray*}\frac{c(i,j)-2}{3} + \frac{c(i-1,j+2)-2}{3} &-& \frac{c(i-1,j+1)-2}{3} - \frac{c(i,j+1)-1}{3} \mathrm{~~when~}i - j \equiv 1 \mathrm{~mod~3,} \\
\frac{c(i,j)-2}{3} + \frac{c(i-1,j+2)-2}{3} &-& \frac{c(i-1,j+1)-1}{3} - \frac{c(i,j+1)-2}{3}  \mathrm{~~when~}i - j \equiv 2 \mathrm{~mod~3,~and} \\
\frac{c(i,j)-1}{3} + \frac{c(i-1,j+2)-1}{3} &-& \frac{c(i-1,j+1)-2}{3} - \frac{c(i,j+1)-2}{3} \mathrm{~~when~}i - j \equiv 0 \mathrm{~mod~3},
\end{eqnarray*}}

\noindent resulting in $\chi(i-j \equiv 0 \mod 3)$.

We next consider the identities (\ref{eq:fourthD}) and (\ref{eq:fourthE}).  Notice that these identities agree up to a shift of $k$ by $1$ so it suffices to prove (\ref{eq:fourthE}).

We have $k^2 \mod 4 \equiv \begin{cases} 1 \mathrm{~if~k~is~odd} \\ 0 \mathrm{~if~k~is~even}\end{cases}$ and the reverse is true for $(k\pm 1)^2 \mod 4$.
Thus if $k$ is odd, taking the difference of the floor functions appearing on RHS from the floor functions appearing on the LHS we obtain
$$\frac{k^2+2k+1}{4} + \frac{k^2-2k+1}{4} - \frac{k^2-1}{4} - \frac{k^2-1}{4} = 1$$
and obtain $$\frac{k^2+2k}{4} + \frac{k^2-2k}{4} - \frac{k^2}{4} - \frac{k^2}{4} = 0$$ if $k$ is even.
Hence, this difference is exactly $\chi(k \mathrm{~is~odd})$ as desired.

Identities (\ref{eq:thirdAA}), (\ref{eq:thirdAAA}), and their analogues for $b(i,j)$ and $c(i,j)$ are proved by the same method as the proofs of
(\ref{eq:thirdA})-(\ref{eq:thirdC}).
\end{proof}

With this lemma in hand we now prove Theorem \ref{thm:explicit}.   Our proof extends the arguments appearing in the proof of Theorem 2.1 of \cite{LaiNewDungeon} and in Section 2 of \cite{LMNT}.

\begin{proof}[Proof of Theorem \ref{thm:explicit}]
We begin with the base cases for $(i,j) = (0,-1), (-1,0),$ or $(0,0)$ and $k=0$ or $1$.
We see that $(\lfloor \frac{a(i,j)}{3} \rfloor, \lfloor \frac{b(i,j)}{3} \rfloor, \lfloor \frac{c(i,j)}{3} \rfloor, \lfloor \frac{(k-1)^2}{4} \rfloor, \lfloor \frac{k^2}{4} \rfloor)= (0,0,0,0,0)$ for all six of these cases and hence $$[z_0^{-1,1}, z_0^{-1,0}, z_{-1}^{0,0}, z_{-1}^{0,1}, z_0^{0,1}, z_0^{0,0}] = Z^\emptyset =  [x_1,x_2,\dots, x_6]$$ as desired.

We prove the explicit formula in general by showing that recurrences induced by mutations are also satisfied by the relevant products of $x_r$, $A$, $B$, $C$, $D$, and $E$.  For brevity, let $\ABC_i^{j} = A^{\lfloor \frac{a(i,j)}{3}\rfloor}
B^{\lfloor \frac{b(i,j)}{3}\rfloor}
C^{\lfloor \frac{c(i,j)}{3}\rfloor}$ and $\mathcal{D}^k = D^{\lfloor \frac{(k-1)^2}{4}\rfloor}E^{\lfloor \frac{k^2}{4}\rfloor}$.
Let $S=S_1S_2$ be a generalized $\tau$-mutation sequence, factored as in Definition \ref{def:prism}.  We assume that $S_1$ and $S_2$ are both of even length and work in the case $i-j \equiv 1 \mod 3$.  By Remark \ref{rem:order123}, we can assume in this case that we have
$Z^{S} = [z_{i}^{j,k+1}, z_{i}^{j,k}, z_{i-1}^{j+1,k}, z_{i-1}^{j+1,k+1}, z_{i}^{j+1,k+1}, z_{i}^{j+1,k}]$.  Induction and the statement of Theorem \ref{thm:explicit} yields
$$Z^{S} = [x_1 \ABC_i^j\mathcal{D}^{k+1}, ~x_2 \ABC_i^j\mathcal{D}^{k}, ~x_3 \ABC_{i-1}^{j+1}\mathcal{D}^{k}, ~x_4 \ABC_{i-1}^{j+1}\mathcal{D}^{k+1}, ~x_5 \ABC_i^{j+1}\mathcal{D}^{k+1}, ~x_6 \ABC_i^{j+1}\mathcal{D}^{k}].$$
Since we are assuming that $S_2$ is of even length, we may assume also that $k$ is even.

The cases where $i - j \equiv 0$ or $2 \mod 3$, $S_1$ is of odd length, or $S_2$ is of odd length can be handled by similar logic.  We make comments below pointing out how to change the proofs in these other cases.

\vspace{0.5em}

By the definition of mutation by $\mu_1$, $\mu_4$, or $\mu_5$, we get recurrences relating cluster variables $z_i^{j,k}$'s together:
{\footnotesize
\begin{eqnarray}
\label{eq:Rec1} z_{i-1}^{j+2,k} z_i^{j,k+1} &=& z_{i-1}^{j+1,k} z_i^{j+1,k+1} + z_{i-1}^{j+1,k+1} z_i^{j+1,k} = (x_3 x_5 + x_4 x_6)~ \ABC_{i-1}^{j+1} ~\ABC_i^{j+1} \mathcal{D}^k \mathcal{D}^{k+1}, \\
\label{eq:Rec2} z_{i+1}^{j,k} z_{i-1}^{j+1,k+1} &=& z_{i}^{j,k} z_i^{j+1,k+1} + z_{i}^{j,k+1} z_i^{j+1,k} = (x_2x_5 + x_1x_6)~ \ABC_{i}^{j} ~\ABC_i^{j+1} \mathcal{D}^k \mathcal{D}^{k+1}, \mathrm{~and~}\\
\label{eq:Rec3} z_{i-1}^{j,k} z_i^{j+1,k+1} &=&  z_i^{j,k+1} z_{i-1}^{j+1,k} +  z_i^{j,k} z_{i-1}^{j+1,k+1} = (x_1x_3 + x_2x_4)~ \ABC_i^{j} ~\ABC_{i-1}^{j+1} \mathcal{D}^k \mathcal{D}^{k+1}.
\end{eqnarray}
}
\vspace{-1em}

Inductively, one of the factors on the LHS is already of the desired form (e.g. $z_{i}^{j, k+1} =  x_2\ABC_i^j \mathcal{D}^{k+1}$,
$ z_{i-1}^{j+1,k+1}  = x_3 \ABC_{i-1}^{j+1}\mathcal{D}^{k+1}$, or
 $z_{i}^{j+1,k+1}  = x_6 \ABC_i^{j+1}\mathcal{D}^{k+1}$) and to verify that the inductive step continues to hold as we mutate by $\mu_1$, $\mu_4$, or $\mu_5$, it suffices to verify that
\begin{eqnarray*}
z_{i-1}^{j+2,k} z_i^{j,k+1} = x_1 x_2 \ABC_{i-1}^{j+2} ~\ABC_i^{j} \mathcal{D}^k \mathcal{D}^{k+1}  &=& (x_3 x_5 + x_4 x_6)~ \ABC_{i-1}^{j+1} ~\ABC_i^{j+1} \mathcal{D}^k \mathcal{D}^{k+1}, \\
z_{i+1}^{j,k} z_{i-1}^{j+1,k+1} = x_3x_4 \ABC_{i+1}^{j} ~\ABC_{i-1}^{j+1} \mathcal{D}^k \mathcal{D}^{k+1} &=& (x_2x_5 + x_1x_6)~ \ABC_{i}^{j} ~\ABC_i^{j+1} \mathcal{D}^k \mathcal{D}^{k+1}, \mathrm{~and~}\\
z_{i-1}^{j,k} z_i^{j+1,k+1} = x_5x_6 \ABC_{i-1}^{j} ~\ABC_i^{j+1} \mathcal{D}^k \mathcal{D}^{k+1}  &=&  (x_1x_3 + x_2x_4)~ \ABC_i^{j} ~\ABC_{i-1}^{j+1} \mathcal{D}^k \mathcal{D}^{k+1}.
\end{eqnarray*}
By cross-multiplying, we reduce the problem to showing
$$\frac{\ABC_{i-1}^{j+2} ~\ABC_i^{j}}{\ABC_{i-1}^{j+1} ~\ABC_i^{j+1}}   = \frac{x_3 x_5 + x_4 x_6}{x_1 x_2}, ~~
\frac{\ABC_{i+1}^{j} ~\ABC_{i-1}^{j+1}}{\ABC_{i}^{j} ~\ABC_i^{j+1}}   = \frac{x_2 x_5 + x_1 x_6}{x_3 x_4}, ~~
\frac{\ABC_{i-1}^{j} ~\ABC_i^{j+1}}{\ABC_i^{j} ~\ABC_{i-1}^{j+1} }   = \frac{x_1 x_3 + x_2 x_4}{x_1 x_2}$$
Continuing to assume that $i-j \equiv 1 \mod 3$, and using the identities of Lemma \ref{lem:ABCDE}, we see that the LHS's equal $A^1B^0C^0$, $A^0B^1C^0$, and $A^0B^0C^1$, respectively.  This agrees with the RHS's by definition.  The proofs for $i-j \equiv 2$ or $0 \mod 3$ are handled similarly using the identities of Lemma \ref{lem:ABCDE} but with cyclic permutations of the initial cluster variables appearing in Equations (\ref{eq:Rec1})-(\ref{eq:Rec3}).

Mutations by $\mu_2$, $\mu_3$, or $\mu_6$ are analogous except with $z_{i-1}^{j+2,k+1} z_i^{j,k}$,
$z_{i+1}^{j,k+1} z_{i-1}^{j+1,k}$, or $z_{i-1}^{j,1} z_i^{j+1,k}$ on the LHS instead.  Doing these in pairs plus a transposition corresponds to the $\tau$-mutation sequences $\tau_1$, $\tau_2$, or $\tau_3$.

If $S_1$ is of odd length, then as in Remark \ref{rem:Delta}, we use the SW-pointed triangular prism instead and this just reverses the direction of the triangular flips, meaning that the factor on the LHS that is inductively known is the one with coordinate $k$ rather than $(k+1)$.  However, the proof is otherwise unaffected.

On the other hand, the $\tau_4$ and $\tau_5$-mutation sequences are more complicated.  Assuming that $i - j \equiv 1 \mod 3$ and $k$ is even, then $\tau_4 = \mu_1 \circ \mu_4 \circ \mu_1 \circ \mu_5 \circ \mu_1 \circ (145)$ corresponds to the
sequence of recurrences:
\begin{eqnarray}
\label{eq:kRec1}
z_{i-1}^{j+2,k}z_i^{j,k+1} &=& z_{i-1}^{j+1,k} z_i^{j+1,k+1} + z_{i-1}^{j+1,k+1} z_i^{j+1,k} \sim (\ref{eq:Rec1})\\
\label{eq:kRec2}
z_{i-1}^{j+1,k+1} z_{i}^{j+1,k-1} &=& z_{i-1}^{j+2,k} z_i^{j,k} + z_{i-1}^{j+1,k} z_i^{j+1,k} \\
\label{eq:kRec3}
z_{i-1}^{j+2,k} z_{i+1}^{j,k} &=& z_{i}^{j+1,k-1} z_i^{j+1,k+1} + (z_{i}^{j+1,k})^2 \\
\label{eq:kRec4}
z_{i}^{j+1,k+1} z_{i}^{j,k-1} &=& z_{i+1}^{j,k} z_{i-1}^{j+1,k} + z_{i}^{j,k} z_i^{j+1,k} \\
\label{eq:kRec5}
z_{i-1}^{j+1,k-1}z_{i+1}^{j,k}  &=& z_{i}^{j,k} z_i^{j+1,k-1} + z_{i}^{j,k-1} z_i^{j+1,k} \sim (\ref{eq:Rec2})
\end{eqnarray}
Notice that in the last recurrence, because of the cyclic permutation, it is as if we are mutating by $\mu_4$ rather than $\mu_1$ hence why this recurrence looks like (\ref{eq:Rec2}) instead of (\ref{eq:Rec1}).  Furthermore, focusing on the $(i,j)$-coordinates, the three terms in (\ref{eq:kRec2}) are a rearrangement of the terms in (\ref{eq:Rec1}) while the three terms in (\ref{eq:kRec4}) are a rearrangement of the terms in (\ref{eq:Rec2}).  (The situation for $\tau_5$ is analogous and left to the reader.)

Equations (\ref{eq:kRec1}) and (\ref{eq:kRec5}) were handled above, thus we now use recurrence (\ref{eq:kRec2}) to show that the inductive hypothesis continues even as $k$ decreases.  Note that in Equation (\ref{eq:kRec2}), all cluster variables are known to have the desired explicit form except for
$z_i^{j+1,k-1}$.  Hence, just as above, it suffices to show that
$$x_4 x_5 ~ \ABC_{i-1}^{j+1} ~\ABC_i^{j} \mathcal{D}^{k+1} \mathcal{D}^{k-1}  =
x_2 x_2 ~ \ABC_{i-1}^{j+2} ~\ABC_i^{j} \mathcal{D}^k \mathcal{D}^{k}
+ x_3 x_6 ~ \ABC_{i-1}^{j+1} ~\ABC_i^{j+1} \mathcal{D}^k \mathcal{D}^{k}.
$$
Dividing by $x_4x_5\ABC_{i-1}^{j+1} ~\ABC_i^{j+1} \mathcal{D}^k \mathcal{D}^{k}
$ yields
$$\frac{\mathcal{D}^{k+1} \mathcal{D}^{k-1} }{\mathcal{D}^k \mathcal{D}^{k}} =
\frac{x_2 x_2}{x_4x_5} ~ \frac{\ABC_{i-1}^{j+2} ~\ABC_i^{j} }{\ABC_{i-1}^{j+1} ~\ABC_i^{j+1} }
+ \frac{x_3 x_6}{x_4x_5} = \frac{x_2^2 A + x_3x_6}{x_4x_5}.$$
Then recurrences (\ref{eq:fourthD}) and (\ref{eq:fourthE}) of Lemma \ref{lem:ABCDE}  yield
$D^1E^0$ on the LHS since we assumed that $k$ was even.
By algebraic manipulations, we see that

{\footnotesize $$D = \frac{x_2^2 A + x_3x_6}{x_4x_5} = \frac{x_3^2 B + x_2 x_6}{x_1x_5} = \frac{x_6^2 C + x_2x_3}{x_1x_4}, \mathrm{~and~}
E =  \frac{x_1^2 A + x_4x_5}{x_3x_6} = \frac{x_4^2 B + x_1 x_5}{x_2x_6} = \frac{x_5^2 C + x_1x_4}{x_2x_3}$$}

\noindent and so we have the desired equality.
The other expressions for $D$ and $E$ allow us to prove the result for other values of $i-j \mod 3$ and when $k$ is odd instead of even.

The recurrence (\ref{eq:kRec3}) allows us to proceed with the induction as well.  Assuming that we know the explicit formula holds for five out of these six cluster variables, it suffices to check
$$x_2 x_3 ~ \ABC_{i-1}^{j+2} ~\ABC_{i+1}^{j} \mathcal{D}^{k} \mathcal{D}^{k}  =
x_5 x_5 ~ \ABC_{i}^{j+1} ~\ABC_i^{j+1} \mathcal{D}^{k-1} \mathcal{D}^{k+1}
+ x_6 x_6 ~ \ABC_{i}^{j+1} ~\ABC_i^{j+1} \mathcal{D}^k \mathcal{D}^{k}
$$
After dividing through and multiplying top and bottom of the LHS by
$\ABC_i^j\ABC_{i-1}^{j+1}$, we get
$$x_2 x_3 ~
\frac{\ABC_{i-1}^{j+2} \ABC_i^j}{
\ABC_{i-1}^{j+1}\ABC_{i}^{j+1}} ~
\frac{\ABC_{i+1}^{j}\ABC_{i-1}^{j+1}}{
\ABC_i^{j} \ABC_i^{j+1}} =
x_5 x_5 ~  \frac{\mathcal{D}^{k-1} \mathcal{D}^{k+1}}
{\mathcal{D}^k \mathcal{D}^{k}}
+ x_6 x_6,
$$
which equals
$$ x_2 x_3 A B = x_5^2 D + x_6^2$$
when $i-j \equiv 1 \mod 3$ and $k$ is even, and this is easily shown.

Lastly, recurrence (\ref{eq:kRec4}) is a variant of recurrence (\ref{eq:kRec2}) and the inductive hypothesis continues by an analogous verification: i.e.
$$x_5 x_1 \ABC_i^{j+1} \ABC_{i}^j \mathcal{D}^{k+1} \mathcal{D}^{k-1} =
x_3 x_3 \ABC_{i+1}^{j} \ABC_{i-1}^{j+1} \mathcal{D}^k \mathcal{D}^k
+ x_2 x_6 \ABC_{i}^{j} \ABC_{i}^{j+1} \mathcal{D}^k \mathcal{D}^k$$
by cross-multiplication
$$x_5 x_1  \frac{\mathcal{D}^{k+1} \mathcal{D}^{k-1}}{\mathcal{D}^k \mathcal{D}^k } =
x_3 x_3 \frac{\ABC_{i+1}^{j} \ABC_{i-1}^{j+1}}{\ABC_i^{j+1} \ABC_{i}^j}
+ x_2 x_6 $$
and the algebraic fact that
$x_5x_1D = x_3^2 B + x_2 x_6$ (which is true when $i-j \equiv 1 \mod 3$ and $k$ is even).

We have demonstrated the validity of this formula for all cluster variables reachable by a generalized $\tau$-mutation sequence.  However, because of Lemma \ref{lem:gentoric}, the set of such cluster variables is the same as the set of cluster variables reachable by a toric mutation sequence.  Hence our proof is complete.
\end{proof}

\begin{remark} We later will give combinatorial formulas for $z_i^{j,k}$ which will provide alternative interpretations of the recurrences (\ref{eq:Rec1})-(\ref{eq:kRec5}) in terms of a technique known as Kuo's Graphical Condensation \cite{kuo1}.
\end{remark}

\section{Contours} \label{Sec:Contours}

In this section we describe a method for constructing subgraphs of the brane tiling $\mathcal{T}$ corresponding to the $dP_3$ Quiver.  This construction is a variant and extension of the Dragon Regions defined in  \cite{LaiNewDungeon}, as well as generalizing the Aztec Castles from \cite{LMNT}.

Given a $6$-tuple $(a,b,c,d,e,f) \in \mathbb{Z}^6$, we consider a {\bf (six-sided) contour} whose side-lengths are $a,b, \dots , f$ in clockwise order (starting from the Northeast corner).  See the right-hand-side of Figure \ref{fig:hexagon}.  In the case of a negative entry we draw the contour in the opposite direction for the associated side.  By abuse of notation, we will refer to such entries as {\bf lengths} even when they are negative.  Several different qualitatively different looking contours are illustrated in Figure \ref{fig:contours}.  Let $\mathcal{C}(a,b,c,d,e,f)$ denote the corresponding (six-sided) contour, which we abbreviate as {\bf contour} in the remainder of the paper.

\begin{figure}\centering
 \includegraphics[width=12cm]{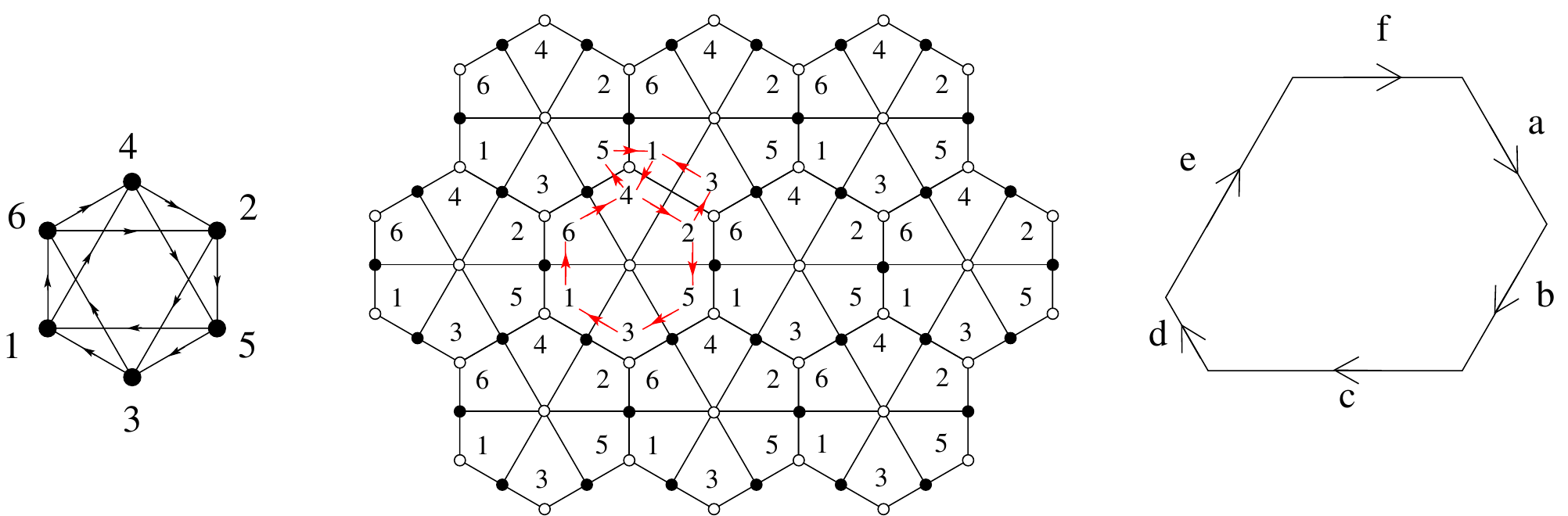} 
\caption{Illustrating the contour $\mathcal{C}(a,b,c,d,e,f)$ in the case that all entries are positive.}
\label{fig:hexagon}
\end{figure}

\begin{figure}
    \centering
    \scalebox{0.9}{\includegraphics[keepaspectratio=true, width=150 mm]{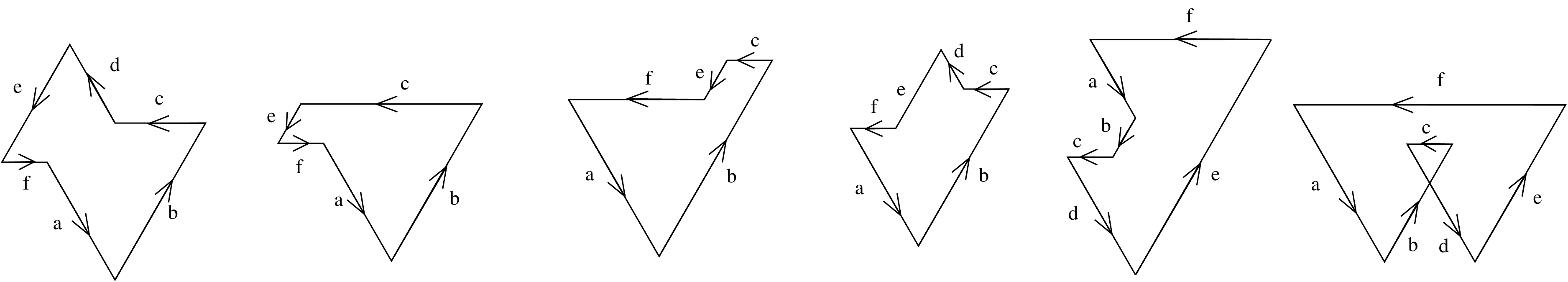}}
\caption{\small From left to right, the cases where $(a,b,c,d,e,f) = $ (1)  $(+, -, + , +, -, +)$, (2) $(+, -, +, 0, -, +)$, (3) $(+, -, +, 0, -, -)$, 
(4)$(+, -, + , +, -, -)$,  (5) $(+, +, +, -, +, -)$, (6) $(+, -, +, -, +, -)$.}
    \label{fig:contours}
\end{figure}

\begin{definition} [\textbf{Subgraphs $\widetilde{\mathcal{G}}(a,b,c,d,e,f)$ and $\mathcal{G}(a,b,c,d,e,f)$}]
\label{def:subgraphs} Suppose that the contour $\mathcal{C}(a,b,c,d,e,f)$ does not intersect itself.  (See the last picture in Figure \ref{fig:contours} for an example of a contour with self-intersections.) Under this assumption, we use the contour $\mathcal{C}(a,b,c,d,e,f)$ to define two subgraphs, which are cut-out by the contour, by the following rules:

Step 1: The brane tiling $\mathcal{T}$ consists of a subdivided triangular lattice.  We superimpose the contour $\mathcal{C}(a,b,c,d,e,f)$ on top of $\mathcal{T}$ so that its sides follow the lines of the triangular lattice, beginning the contour at a white vertex of degree $6$.  In particular, sides $a$ and $d$ are tangent to faces $1$ and $2$, sides $b$ and $e$ are tangent to faces $5$ and $6$, and sides $c$ and $f$ are tangent to faces $3$ and $4$.  We scale the contour so that a side of length $\pm 1$ transverses two edges of the brane tiling $\mathcal{T}$, and thus starts and ends at a white vertex of degree $6$ with no such white vertices in between.

Step 2: For any side of positive (resp. negative) length, we remove all black (resp. white) vertices along that side.

Step 3: A side of length zero corresponds to a single white vertex.  If one of the adjacent sides is of negative length, then that white vertex is removed during step 2.  On the other hand, if the side of length zero is adjacent to two sides of positive length, we keep the white vertex.  However, as a special case, if we have three sides of length zero in a row, we instead remove that white vertex\footnote{The reader might wonder what happens if we have two adjacent sides of length zero (between two positive values) or four or more adjacent sides of length zero.  As shown in Figures \ref{fig:decomposition} and \ref{fig:decomposition2}, for the subgraphs we care about in this paper, such a case cannot occur.}.

Step 4: We define $\widetilde{\mathcal{G}}(a,b,c,d,e,f)$ to be the resulting subgraph, which will contain a number of black vertices of valence one.  After matching these up with the appropriate white vertices and continuing this process until no vertices of valence one are left, we obtain a simply-connected graph $\mathcal{G}(a,b,c,d,e,f)$ which we call the {\bf Core Subgraph}, following notation of \cite{BMPW}.
\end{definition}

\begin{remark}
In the case that $(a,b,c,d,e,f) = (+,-,+,+,-,\pm)$, these subgraphs (without face-labels) agree with the $DR^{(1)}(a,-b,c)$ dragon regions of \cite{LaiNewDungeon}. Similarly, $(a,b,c,d,e,f) = (-,+,-,-,+,\pm)$ agree with the $DR^{(2)}(-a,b,-c)$ dragon regions.
\end{remark}

\begin{remark} \label{Rem:LMNT}
In the case that $a+b+c = 0$ or $1$, after negating $b$ and $e$, followed by subtracting $1$ from each entry, recovers the NE Aztec Castles in Section 3 of \cite{LMNT}.  The SW Aztec Castles in Section 4 of \cite{LMNT} are obtained analogously, but using the opposite sign convention and a different translation.
In particular, using the notation of \cite{LMNT}, with $i, j \geq 0$,
\begin{eqnarray*}\gamma_i^j &=& \mathcal{G}(j, -i-j, i, j+1, -i-j-1, i+1), \mathrm{~and} \\
\widetilde{\gamma}_{-i}^{-j} &=& \mathcal{G}(-j+1, i+j, -i, -j, i+j+1, -i-1). \end{eqnarray*}
\end{remark}

\begin{example} \label{ex:0}
 Here we provide examples of subgraphs arising from contours for each of the first five cases appearing in Figure \ref{fig:contours}.
The subgraphs illustrated are

(1) $\mathcal{G}(3, -4, 2, 2, -3, 1)$, (2) $\mathcal{G}(1, 1, 1, -4, 6, -4)$, (3) $\mathcal{G}(5,-6,4,0,-1,-1)$,

(4) $\mathcal{G}(6,-7,4,1,-2,-1)$, and (5) $\mathcal{G}(5,-8,6,0,-3,1)$, respectively.

\noindent See Figure \ref{fig:ex0}.

\end{example}

\begin{figure}
\includegraphics[width=5.3in]{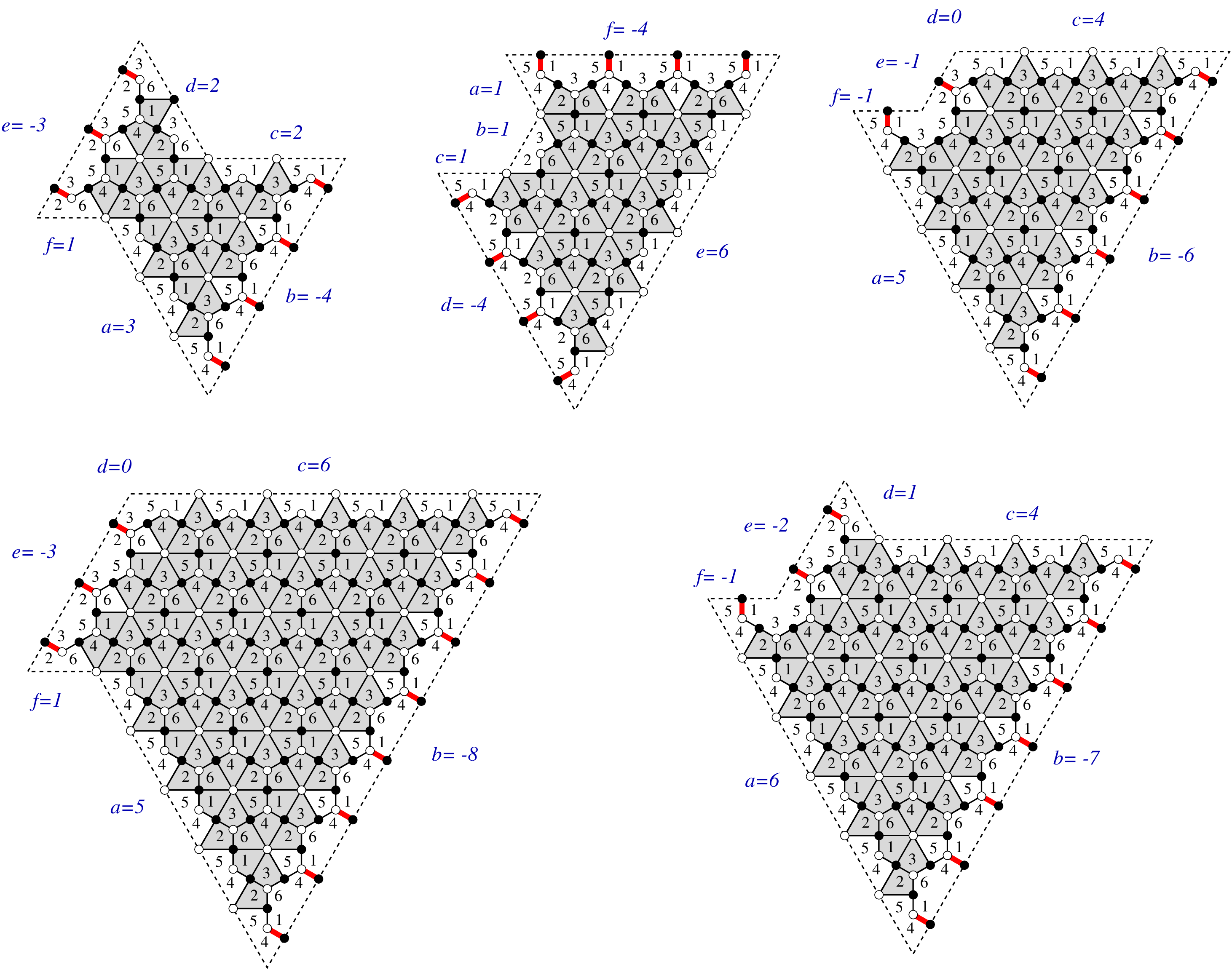}
\caption{The subgraphs associated to Example \ref{ex:0}.}
\label{fig:ex0}
\end{figure}

\begin{example} \label{ex:Dragons}
We can also produce Aztec Dragons \cite{perfect,CY,EnumPropp} in this notation.
Let $\sigma\mathcal{C}(a,b,c,d,e,f) = \mathcal{C}(a+1, b-1, c+1, d-1, e+1, f-1)$. We denote by
$\mathcal{C}_i^j$ the contour $\mathcal{C}(j, -i-j, i, j+1, -i-j-1, i+1)$.
For each $n \in \mathbb{Z}_{\geq 0}$,
\[D_{n + 1/2} = \mathcal{G}(\mathcal{C}_{-1}^{n+1}) = \mathcal{G}(n+1, -n, -1, n+2, -n-1, 0),\]
\[D_n = \mathcal{G}(\sigma\mathcal{C}_{0}^{n}) = \mathcal{G}(n+1, -n-1, 1, n, -n, 0).\]
See Figure \ref{fig:AztecDragons}.

As a special case of Definition \ref{def:6tuple}, we will see that $D_{n}$'s and $D_{n+1/2}$'s arise from the $\tau$-mutation sequence $\tau_1\tau_2\tau_3\tau_1\tau_2\dots$ continued periodically.
This $\tau$-mutation sequence corresponds to a vertical translation in the square lattice $L^\Delta$.
\end{example}

\begin{figure}
\includegraphics[width=5.3in]{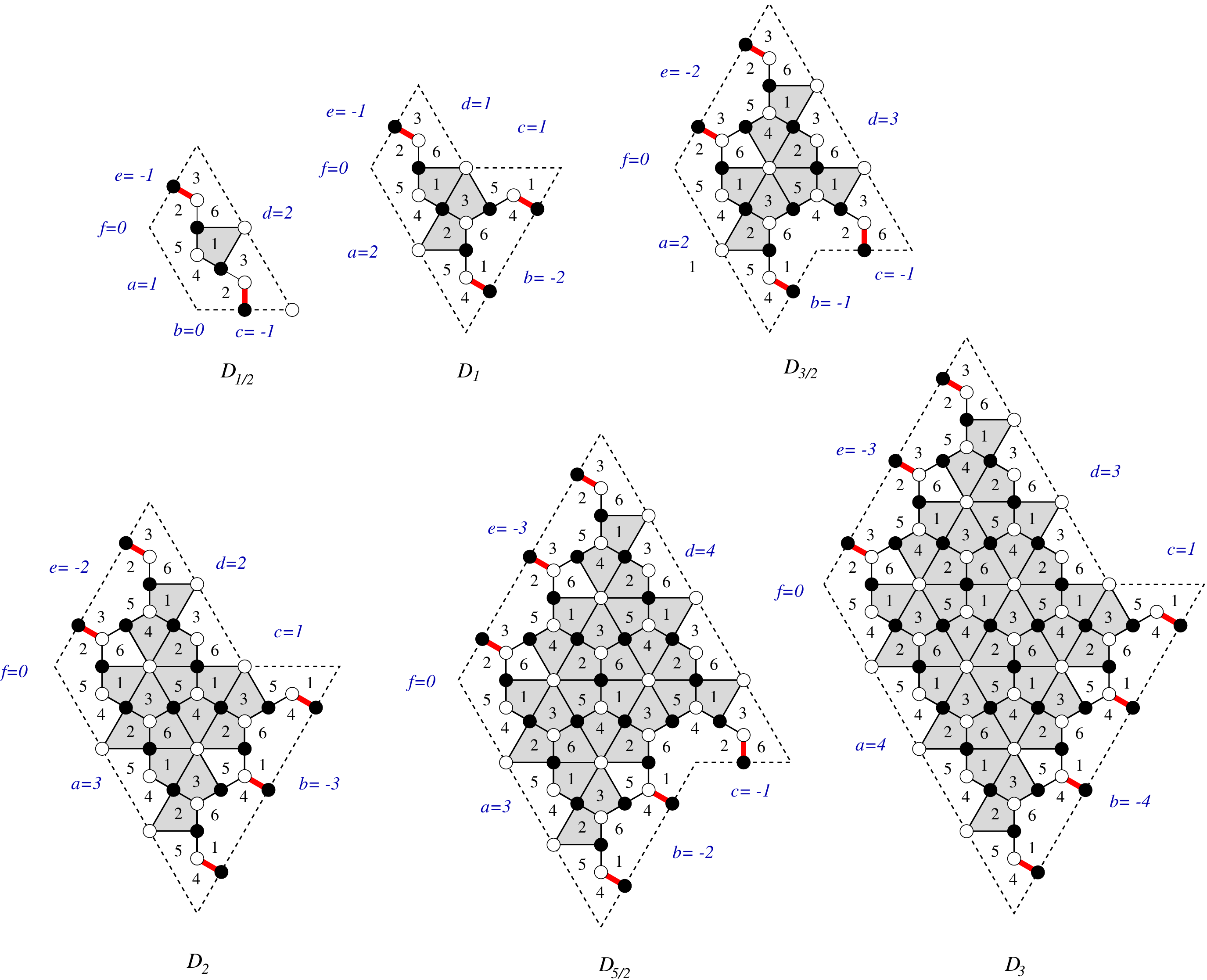}
\caption{The Aztec Dragon constructed as examples of $\mathcal{G}(a,b,c,d,e,f)$'s.  See Example \ref{ex:Dragons}.}
\label{fig:AztecDragons}
\end{figure}

\begin{example} \label{ex:antiparallel}
Another possible contour arises when we have two sides of length zero in a row.  On first glance, the two contours given below appear to be triangles with all sides of length $5$.  However, these
are actually degenerate quadrilaterals where two adjacent sides happen to be anti-parallel.  In particular, along one of the three sides of this contour, the pattern of including and excluding white and black vertices switches, signifying the invisible corner of an angle of $180^\circ$.  Thus, when building subgraphs corresponding to a contour we note that
$\mathcal{G}(5,-5,5,0,0,0)$, $\mathcal{G}(5,-5,4,0,0,-1)$, and $\mathcal{G}(5,-5,5,3,0,-2)$ would all be different subgraphs of $\mathcal{T}$.  We have illustrated
$\mathcal{G}(5,-5,3,0,0,-2)$ and $\mathcal{G}(-5,5,-2,0,0,3)$ in Figure \ref{fig:antiparallel}.
\end{example}

\begin{figure}
\includegraphics[width=5.3in]{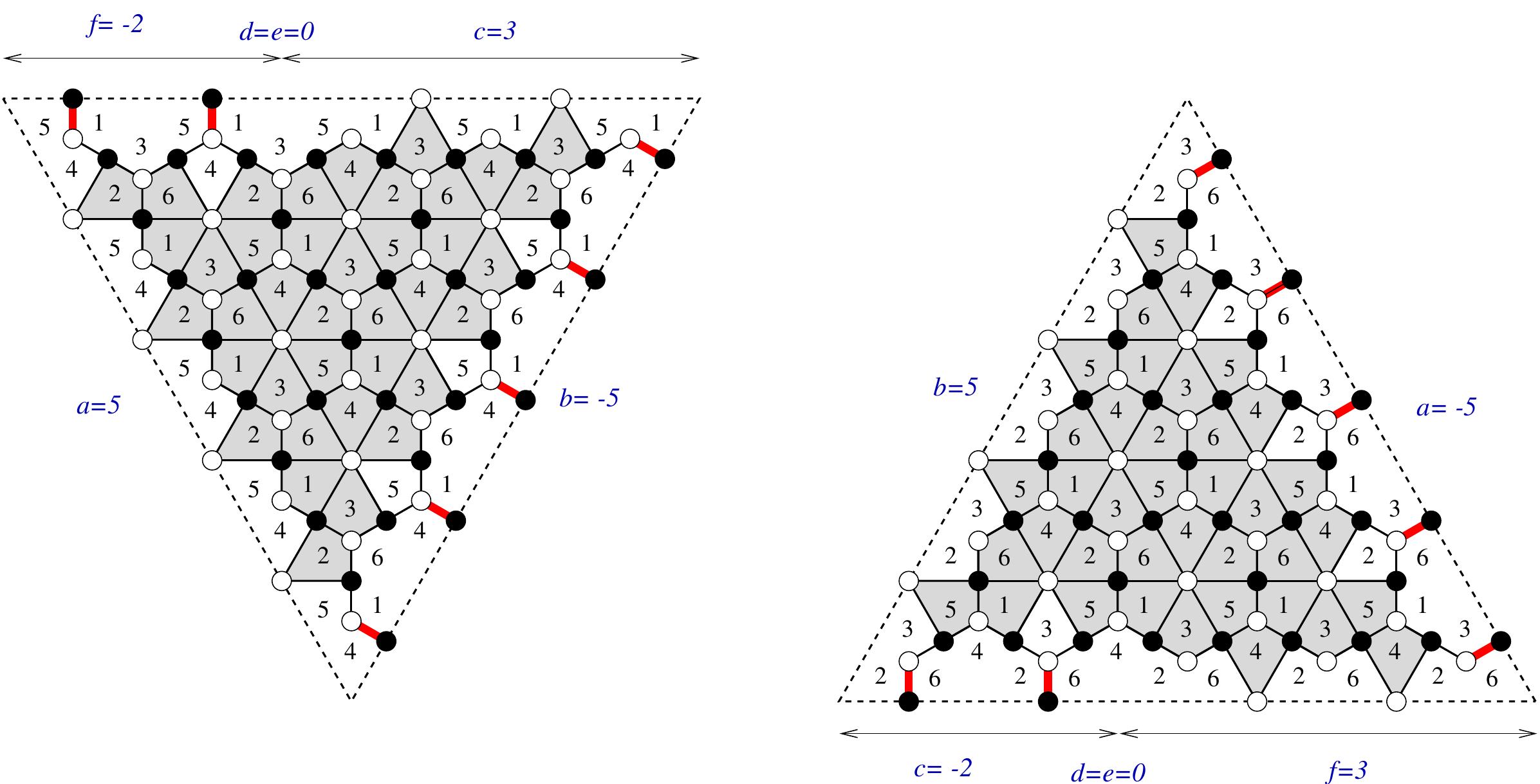}
\caption{Examples of graphs obtained from contours which are degenerate quadrilaterals, as described in Example \ref{ex:antiparallel}.}
\label{fig:antiparallel}
\end{figure}

\section{From Mutations to Subgraphs} \label{sec:subgraphs}

In this section, we come to our second main result, Theorem \ref{thm:main}, which provides a combinatorial interpretation for the Laurent polynomials $z_{i}^{j,k}$ defined in Theorem \ref{thm:explicit}.  Combining this with Definition \ref{def:prism}, this yields a direct combinatorial interpretation for most cluster variables reachable by a toric mutation sequence (using Lemma \ref{lem:gentoric} and the geometry of the $\mathbb{Z}^3$ lattice described in Section \ref{sec:gentoric}).

Motivated by the definition of NE and SW Aztec Castles, see Remark \ref{Rem:LMNT}, we extend the definition to three dimensions as follows.

\begin{definition} \label{Def:LMNT}
For all $i,j \in \mathbb{Z}$, we let
$$\mathcal{C}_i^j \mathrm{~be~the~contour~} \mathcal{C}(j, -i-j, i, j+1, -i-j-1, i+1).$$ Recall the map from Example \ref{ex:Dragons}, we let $\sigma^k\mathcal{C} = \mathcal{C}(a+k, b-k, c+k, d-k, e+k, f-k)$ for $\mathcal{C} = \mathcal{C}(a,b,c,d,e,f)$.
Combining this together, we let $$\mathcal{C}_{i}^{j,k} = \mathcal{C}(j+k, -i-j-k, i+k, j+1-k, -i-j-1+k, i+1-k).$$
\end{definition}

\begin{remark}In particular, note that $\sigma\mathcal{C}_i^j = \mathcal{C}(j+1, -i-j-1, i+1,j, -i-j, i)$.  This agrees with the fact that $\sigma \gamma_i^j$, was defined as $\gamma_i^j$ after $180^\circ$ rotation in \cite{LMNT}.  Furthermore, the SW Aztec Castles of \cite{LMNT} can
be described as $\widetilde{\gamma}_{-i}^{-j}~~=~~ \sigma\mathcal{C}_{-i-1}^{-j}$ and $\sigma\widetilde{\gamma}_{-i}^{-j}~~=~~ \mathcal{C}_{-i-1}^{-j}.$
\end{remark}

Definition \ref{Def:LMNT} motivates us to define the map $\phi: \mathbb{Z}^3 \to \mathbb{Z}^6$ defined by $\mathcal{C}_i^{j,k} = \mathcal{C}(a,b,c,d,e,f)$.  In other words,
$$a = j + k, ~~~b = -i-j-k, ~~~ c=  i+k, ~~~ d = j-k +1, ~~~ e = -i-j+k -1, ~~~ f = i - k +1.$$

\begin{lemma} \label{lem:closeup}
Out of all possible contours $\mathcal{C}(a,b,c,d,e,f)$, those which are of the form $\mathcal{C}_i^{j^k} = \mathcal{C}(j+k, -i-j-k, i+k, j+1-k, -i-j-1+k, i+1-k)$ actually comprise all contours that  (i) close up; and (ii) either have a self-intersection or yield a subgraph such that the number of black vertices is the number of white vertices.
\end{lemma}

\begin{proof}
First of all, since the plane has two directions, a contour must satisfy
$$a+b = d+e \mathrm{~~~and~~~} c+d = f + a$$
if it is going to close up.  These two relations also imply the equivalent third relation $b +c = e+f$.  Picking two of these three relations already restrict us from $\mathbb{Z}^6$ to $\mathbb{Z}^4$ (calculating the Smith normal form can be used to ensure that these relations do not introduce any torsion elements).  To reduce all the way to $\mathbb{Z}^3$, the balancing of the vertex colors yields the linearly independent relation
$$a+b+c+d+e+f=1$$ as we show now (again, no torsion is introduced by including this relation).

Firstly, note that a single triangle has three white vertices and three black vertices on the boundary, but an extra white vertex at its center.
We inductively build our contour by attaching a triangle to our shape along one or two sides.  Either way, this does not alter the difference between the number of white vertices and black vertices.

However, after building the full contour (which has one extra white vertex), we then remove vertices from the boundary as instructed by Definition \ref{def:subgraphs}.
Every side with positive length $p$ forces us to remove $p$ black vertices.  Every segment with total negative length $-n$ (by a segment we mean at least one, but possibly multiple consecutive sides, all of negative length with sides of positive length on either side) forces us to remove $(n+1)$ white vertices.

For the purposes of this proof, note that a segment of two or more zeros in a row between two sides of positive length also counts as a negative segment and would lead to removing exactly one white vertex.  (As indicated in Definition \ref{def:subgraphs}, this can only happen with three zeros in a row and such patterns only occur in contours corresponding to initial cluster variables).  On the other hand, a single side of length zero or a segment of a zeroes  adjacent to at least one side of negative length does not lead to the removal of any vertices of either color.

By considering all possible sign patterns without self-intersecting contours, see Figures \ref{fig:decomposition} and \ref{fig:decomposition2}, we see that there are always exactly two negative segments and hence we are removing
$2-a-b-c-d-e-f$ more white vertices than black vertices leading to a color-balanced subgraph if and only if we have $a+b+c+d+e+f=1$.
\end{proof}

\subsection{Possible shapes of Aztec Castles} \label{sec:shapes}

Since we have just shown that the contours $\mathcal{C}_i^{j,k}$ are of the most general form for our combinatorial purposes (i.e. we want subgraphs which have perfect matchings), we next describe the possible shapes of these contours.  In this section we work directly with the six-tuples $(a,b,c,d,e,f)$ rather than the contours $\mathcal{C}(a,b,c,d,e,f)$ themselves.  By direct computation we arrive at the following description: $32$ possible sign-patterns (excluding those with zeroes) organized into orbits of size $1$ or $6$ under the action of twisted rotation $\theta : (a,b,c,d,e,f) \to (-f, -a, -b, -c, -d, -e)$.  When $k \geq 1$ (resp. $k \leq 0$), we get only $19$ of these sign-patterns, $6$ of which are possible regardless of $k$.

1) First Possibility: $(a,b,c,d,e,f) = (+, -, -, +, -, -)$, $(-, +, +, -, +, +)$ or a cyclic rotation of one of these.
See Figure \ref{fig:shapes2}.

2) Second Possibility: $(a,b,c,d,e,f) = (+, -, -, +, +,  -)$, $(-, +, +, -, -, +)$ or a cyclic rotation.  See Figure \ref{fig:shapes1}.

3) Third Possibility: $(a,b,c,d,e,f) =  (+, -, -, -, +, -)$, $(-, +, +, +, -, +)$, or a cyclic rotation.  See Figure \ref{fig:shapes3}.

4) Lastly, there are degenerate cases which are a combination of two of these possibilities where one or more of the sides are length zero.

5) Six-tuples of the form $(a,b,c,d,e,f) = (+, -,+, -, +,-)$ or $(-, +, -, +, -, +)$ also appear, see Figure \ref{fig:shapes4}, but these always correspond to self-intersecting contours, which we do not give a combinatorial interpretation for in this paper.  This leaves a question for future work.  See Problem \ref{prob:self-int} in Section \ref{sec:open}.

\begin{figure}
\includegraphics[width=5.5in]{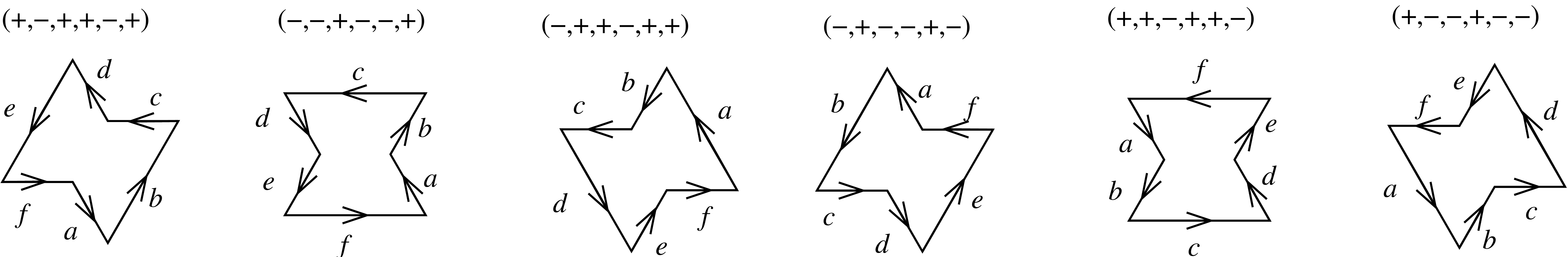}
\caption{Six unbounded sign-unbalanced regions.}
\label{fig:shapes2}
\end{figure}

\begin{figure}
\includegraphics[width=5in]{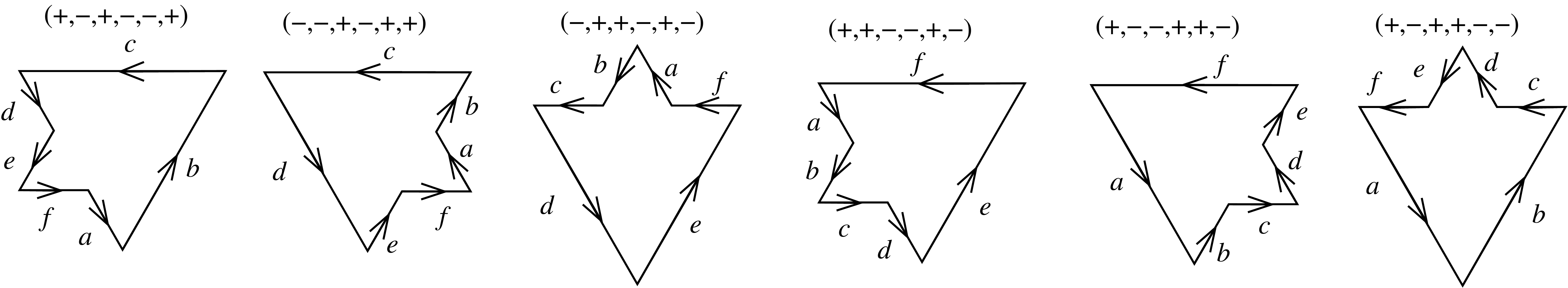}
\caption{Six unbounded sign-balanced regions.}
\label{fig:shapes1}
\end{figure}

\begin{figure}
\includegraphics[width=5in]{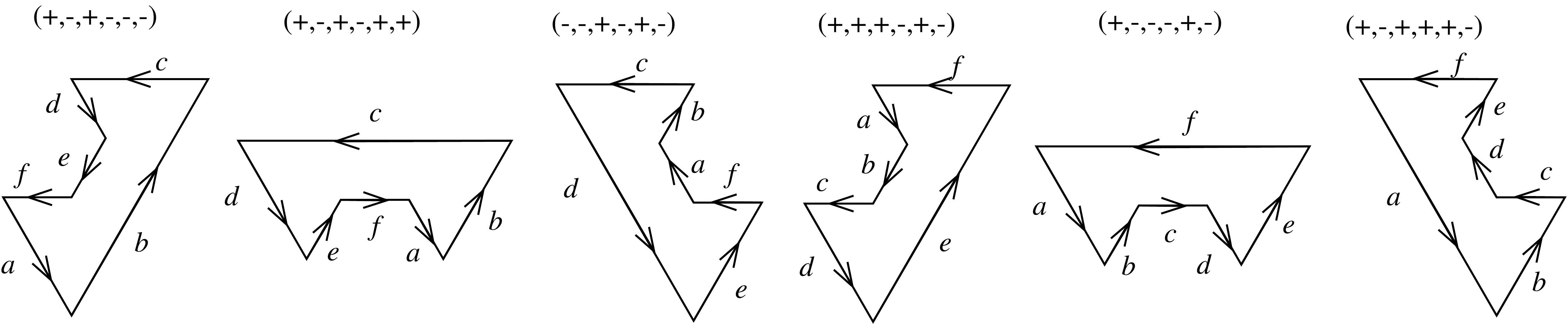}
\caption{Six bounded sign-unbalanced regions.}
\label{fig:shapes3}
\end{figure}

\begin{figure}
\includegraphics[width=2in]{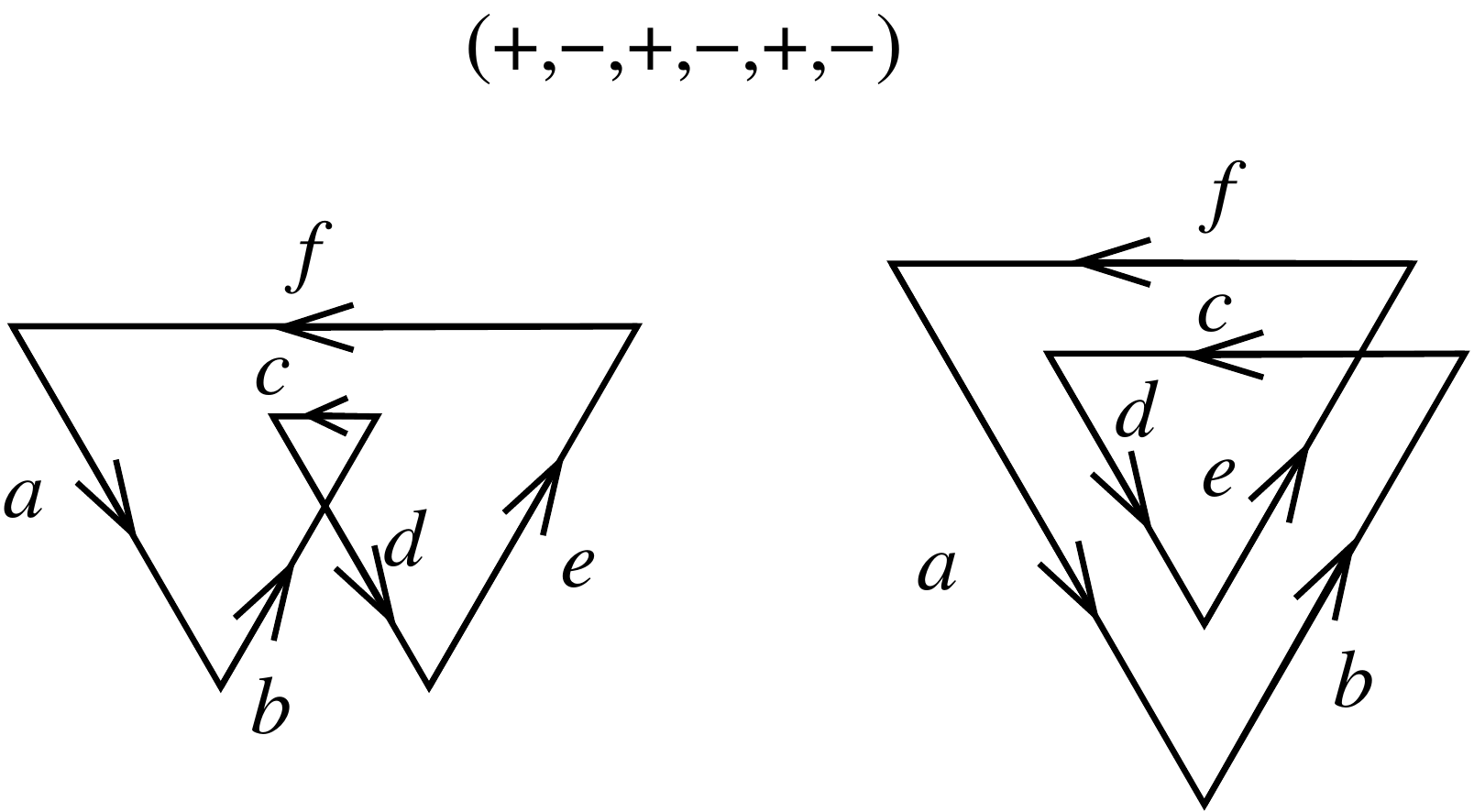}
\caption{One bounded sign-balanced region with the two types of self-intersections shown.}
\label{fig:shapes4}
\end{figure}

\begin{remark}
Figures \ref{fig:decomposition} and \ref{fig:decomposition2} illustrate how the map $\phi$ sends $(i,j,k)$ to these six-tuples.  For a fixed $k$, the map $\theta$ acts as $60^{\circ}$ clockwise rotation of $(i,j)$-plane.  These two-dimensional cross-sections motivates our terminology of unbounded and bounded regions for the various sign-patterns.
\end{remark}

\begin{remark}
If $ -2 \leq k \leq 3$, these bounded regions shrink to a point and some of the unbounded regions shrink to a line as there are no non-degenerate examples of those particular sign patterns.  This is why previous work \cite{LMNT}, which corresponded to the $k=0$ and $k=1$ cases involved only unbounded regions.  See Figure 5 of \cite{LMNT}. The decompositions for $k\leq -3$ and $k\geq 4$ are analogous to one another except for applying the negation map $(a,b,c,d,e,f) \to (-a,-b,-c,-d,-e,-f)$ everywhere.
\end{remark}

\begin{figure}
\includegraphics[width=5in]{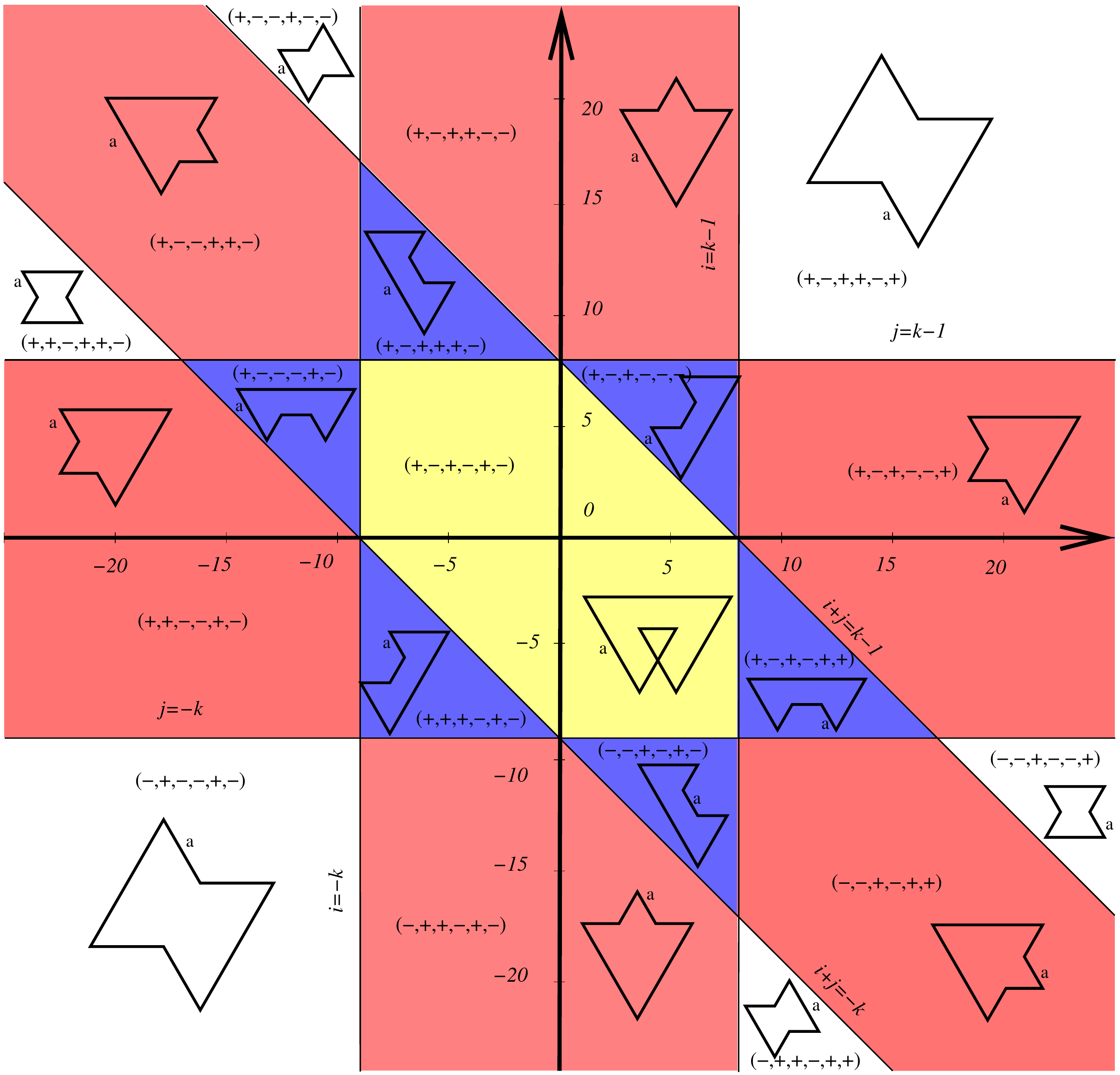}
\caption{Possible sign-patterns for a fixed $k \geq 1$, where the six lines illustrate the $(i,j)$-coordinates so that one of the elements of the $6$-tuple equals zero.}
\label{fig:decomposition}
\end{figure}

\begin{figure}
\includegraphics[width=5in]{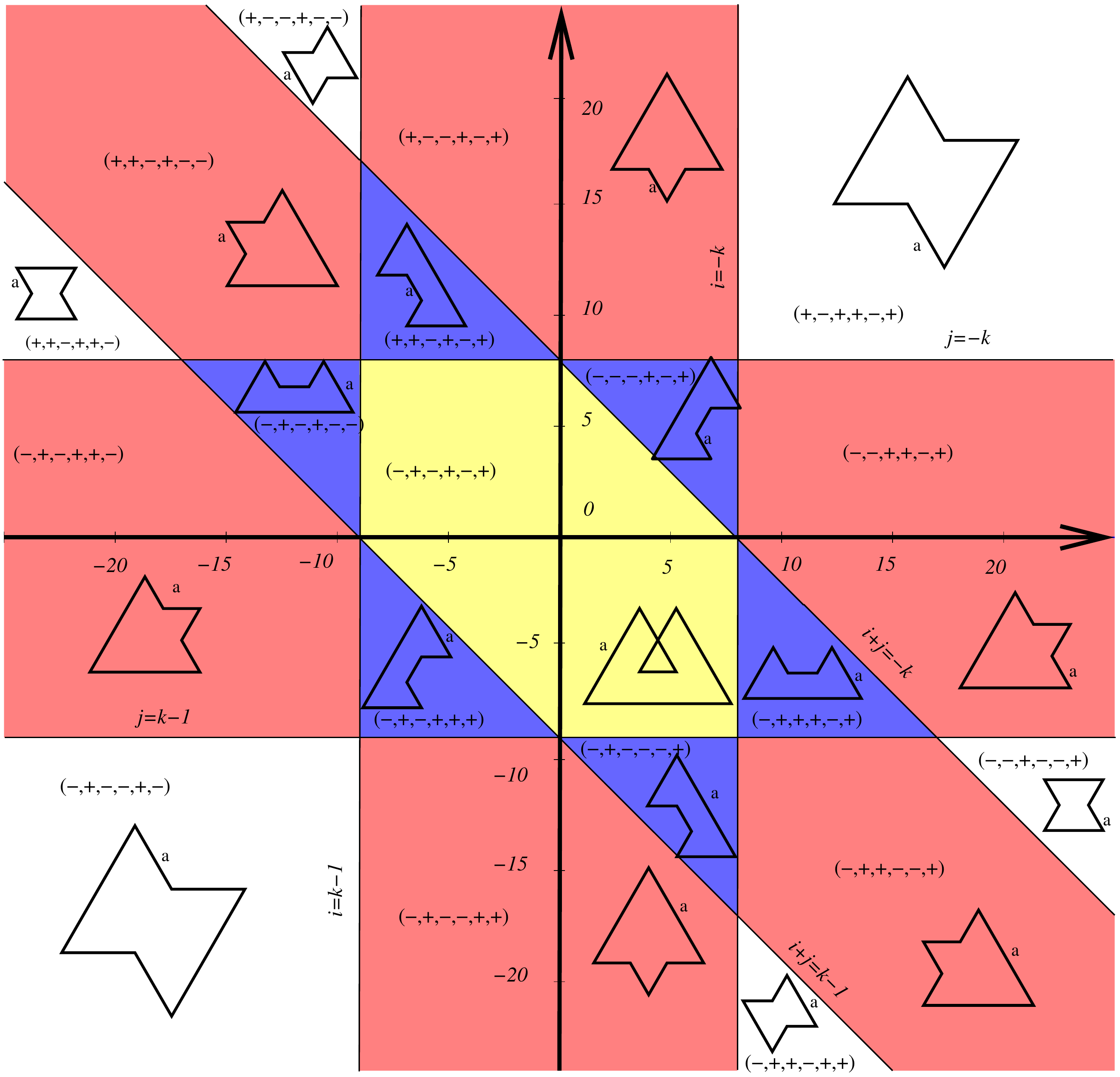}
\caption{Possible sign-patterns for a fixed $k \leq 0$.  Observe that the uncolored unbounded regions correspond to the same sign patterns as they did in the $k \geq 1$ case.}
\label{fig:decomposition2}
\end{figure}

\subsection{Combinatorial Interpretation of Ordered Clusters}\label{sec:comb}

In this section, we describe the weighting scheme that yields Laurent polynomials from the subgraphs defined in Section \ref{sec:subgraphs}.  In summary, given a toric mutation sequence $S$, we obtain a specific prism (or a well-defined deformation of a prism) in $\mathbb{Z}^3$, which corresponds to an ordered cluster $Z^S$.  Then each such lattice point $(i,j,k)$ corresponds via map $\phi$ to a $6$-tuple $(a,b,c,d,e,f)$.  Subsequently, the contour $\mathcal{C}_i^{j,k} = \mathcal{C}(a,b,c,d,e,f)$ yields a subgraph\footnote{As mentioned in Definition \ref{def:subgraphs}, this construction makes sense as long as $\mathcal{C}(a,b,c,d,e,f)$ has no self-intersections.} $\mathcal{G}(\mathcal{C}_i^{j,k}) = \mathcal{G}(a,b,c,d,e,f)$. 
 Finally, the weighting scheme leads to a Laurent polynomial $z(a,b,c,d,e,f)$ from $\mathcal{G}(a,b,c,d,e,f)$.  The main theorem of this section is that $z(a,b,c,d,e,f) = z_{i}^{j,k}$, the cluster variable in cluster $Z^S$ reached via the toric mutation sequence $S$.

We will adopt the weighting scheme on the brane tiling utilized in \cite{zhang}, \cite{speyer}, \cite{GK}, and \cite{LMNT}.  Associate the \textbf{weight} $\frac{1}{x_i x_j}$ to each edge bordering faces labeled $i$ and $j$ in the brane tiling. Let $\mathcal{M}(G)$ denote the set of perfect matchings of a subgraph $G$ of the brane tiling. We define the weight $w(M)$ of a perfect matching $M$ in the usual manner as the product of the weights of the edges included in the matching under the weighting scheme. Then we define the weight of $G$ as
\[
w(G) = \sum_{M \in \mathcal{M}(G)}w(M).
\]

We also define the \textbf{covering monomial}, $m(G)$, of any graph $G=\mathcal{G}(a,b,c,d,e,f)$, resulting from a contour without self-intersection, as follows.

\begin{definition} \label{def:covmon} First we define the covering monomial  $m(\widetilde{G})$ of  the graph $\widetilde{G}=\widetilde{\mathcal{G}}(a,b,c,d,e,f)$ as the product $x_1^{a_1}x_2^{a_2}x_3^{a_3}x_4^{a_4}x_5^{a_5}x_6^{a_6}$, where $a_j$ is the number of faces labeled $j$ restricted inside the contour $\mathcal{C}(a,b,c,d,e,f)$. Consider the edges in $\widetilde{G}$ that are adjacent to a valence one black vertex (this edge is not in $G$). Assume that $b_i$ is the number of faces labeled $i$ adjacent to such a forced edge. We now define the covering monomial, $m(G)$, of $G$ as the product $x_1^{a_1-b_1}x_2^{a_2-b_2}x_3^{a_3-b_3}x_4^{a_4-b_4}x_5^{a_5-b_5}x_6^{a_6-b_6}$.
 Figure \ref{fig: cov mon} illustrates an example of the quadrilaterals included in the covering monomial of a small subgraph, outlined in red.
\end{definition}

\begin{figure}[h]
    \centering
        \scalebox{0.75}{\includegraphics[keepaspectratio=true, width=50 mm]{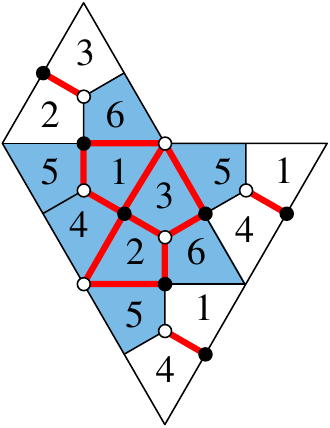}}
	\caption{\small The covering monomial of $\mathcal{G}(2,-2,1,1,-1,0) \subset \widetilde{\mathcal{G}}(2,-2,1,1,-1,0)$ outlined in red includes the blue quadrilaterals and is given by $x_{1}x_{2}x_{3}x_{4}x_{5}^{3}x_{6}^{2}$.}  	\label{fig: cov mon}
\end{figure}

Finally, to make notation more concise in later proofs, it will be useful to define the product of the covering monomial and weight of a subgraph $G$ as
\begin{equation*}
c(G) = w(G)m(G).
\end{equation*}
Similarly, we define $c(\widetilde{G})=w(\widetilde{G})m(\widetilde{G})$.

\begin{remark} \label{rem:cov}
We have $$c(\mathcal{G}(a,b,c,d,e,f)) = c(\widetilde{\mathcal{G}}(a,b,c,d,e,f))$$
for any contours without self-intersections.  Indeed, besides containing a larger set of faces, the difference between perfect matchings of $\mathcal{G}(a,b,c,d,e,f)$ and  $\widetilde{\mathcal{G}}(a,b,c,d,e,f)$ are the edges forced to be in a perfect matching by being adjacent to a valence one black vertex.  On the other hand, the contribution to the covering monomial of $\widetilde{\mathcal{G}}(a,b,c,d,e,f)$ from the faces incident to these edges balances out the weight of these edges. Then the equality follows.
\end{remark}

\begin{definition} [\bf Contour $6$-tuple] \label{def:6tuple}
Given a toric mutation sequence $S$, we may obtain a corresponding six-tuple of points
$$[(i_1,j_1,k_1),(i_2,j_2,k_2),(i_3,j_3,k_3),(i_4,j_4,k_4),(i_5,j_5,k_5),(i_6,j_6,k_6)]$$
in $\mathbb{Z}^3$ starting from the initial prism
{\footnotesize $[(0,-1,1), (0,-1,0), (-1,0,0), (-1,0,1), (0,0,1), (0,0,0)]$} and using the lattice moves illustrating in Figures \ref{fig:ModelII},  \ref{fig:ModelIII}, and \ref{fig:ModelsGeom}\footnote{In particular, when $S$ is specifically a generalized $\tau$-mutation sequence, this six-tuple of points is the prism $\Delta^S$ defined in Definition \ref{def:prism}.}.
We use these six $\mathbb{Z}^3$-points to define a corresponding $6$-tuple of Contours, which we denote as $\mathcal{C}^S$.
$$\mathcal{C}^{S} =
[\mathcal{C}_{i_1}^{j_1,k_1}, \mathcal{C}_{i_2}^{j_2,k_2}, \mathcal{C}_{i_3}^{j_3,k_3},
\mathcal{C}_{i_4}^{j_4,k_4}, \mathcal{C}_{i_5}^{j_5,k_5}, \mathcal{C}_{i_6}^{j_6,k_6}]$$
using  $\mathcal{C}_i^{j,k}= \mathcal{C}(j+k, -i-j-k, i+k, j+1-k, -i-j-1+k, i+1-k)$ from Definition  \ref{Def:LMNT}.
\end{definition}

In particular, notice that by Definition \ref{def:6tuple}, $\mathcal{C}^\emptyset = [\mathcal{C}_0^{-1,1}, \mathcal{C}_0^{-1,0}, \mathcal{C}_{-1}^{0,0}, \mathcal{C}_{-1}^{0,1}, \mathcal{C}_0^{0,1}, \mathcal{C}_0^{0,0}]$, which we abbreviate as $[C_1, C_2, C_3, C_4, C_5, C_6],$ and refer to as the initial contours.
\begin{eqnarray}
\label{contours1} C_1 = \mathcal{C}(0,0,1,-1,1,0), ~~C_2 = \mathcal{C}(-1,1,0,0,0,1), \\
\label{contours2} C_3 = \mathcal{C}(0,1,-1,1,0,0), ~~C_4 = \mathcal{C}(1,0,0,0,1,-1), \\
\label{contours3} C_5 = \mathcal{C}(1,-1,1,0,0,0), ~~C_6 = \mathcal{C}(0,0,0,1,-1,1).
\end{eqnarray}

\begin{theorem} \label{thm:main}
Let $S$ be a toric mutation sequence.  Build the $6$-tuple of contours  $\mathcal{C}^S$'s as in Definition \ref{def:6tuple}.  Then use these six contours to build six subgraphs of $\mathcal{T}$ as in Section \ref{Sec:Contours}, i.e.
$\mathcal{G}^S_i = \mathcal{G}(a,b,c,d,e,f)$ if $\mathcal{C}^S_i = \mathcal{C}(a,b,c,d,e,f)$ for $1 \leq i \leq 6$.
 If none of the contours have self-intersections, then the cluster obtained from the toric mutation sequence $S$ is
 $$Z^S = [ c(\mathcal{G}^S_1), c(\mathcal{G}^S_2), c(\mathcal{G}^S_3), c(\mathcal{G}^S_4), c(\mathcal{G}^S_5)
c(\mathcal{G}^S_6)].$$
\end{theorem}

We prove this Theorem in Section \ref{Sec:proof}.  Note that it is sufficient to prove the result when $S$ is a generalized $\tau$-mutation sequence, due to
Lemma \ref{lem:gentoric}.

By specializing $x_1=x_2=\dots=x_6=1$, Theorems  \ref{thm:explicit} and \ref{thm:main}  imply that the (unweighted) numbers of perfect matchings of our subgraphs are always certain products of a power of $2$ and a power of $3$. This also implies the work of the first author about tilings of the Dragon Regions in \cite[Theorem 2]{LaiNewDungeon}. In other words, our main results here provide a $6$-parameter refinement of Theorem 2 in \cite{LaiNewDungeon}.

\begin{example} \label{ex:1}
As an illustration of Theorem \ref{thm:main}, we consider the generalized $\tau$-mutation sequence
$S= \tau_1\tau_2\tau_3\tau_1\tau_2\tau_3 \tau_2 \tau_1 \tau_4$.  Letting $S_1 = \tau_1\tau_2\tau_3\tau_1\tau_2\tau_3\tau_2\tau_1$, and applying the corresponding alcove walk, we reach the triangle
$\{(1,3),(1,2),(0,3)\}$ in $L^\Delta$, written in order using $I-J \mod 3$.  We then apply $\tau_4$ to get
$$\mathcal{C}^S = [ \mathcal{C}_1^{3,-1}, \mathcal{C}_1^{3,0}, \mathcal{C}_1^{2,0}, \mathcal{C}_1^{2,-1},
\mathcal{C}_0^{3,-1}, \mathcal{C}_0^{3,0}], \mathrm{~which~equals~}$$
$$[\mathcal{C}(2, -3, 0, 5, -6, 3),  \mathcal{C}(3, -4, 1, 4, -5, 2), \mathcal{C}(2, -3, 1, 3, -4, 2),$$
$$\hspace{3em}\mathcal{C}(1, -2, 0, 4, -5, 3), \mathcal{C}(2, -2, -1, 5, -5, 2), \mathcal{C}(3, -3, 0, 4, -4, 1)].$$
These six contours respectively correspond to the six subgraphs appearing in Figure \ref{fig:ex1}.  The associated cluster variables are
$$\frac{(x_1x_3 + x_2x_4)^{4}(x_2x_5 + x_1x_6)^{6}(x_3x_5 + x_4x_6)^{7}(x_1x_3x_6 + x_2x_3x_5 + x_2x_4x_6)}{x_1^{7}x_2^{7}x_3^{6}x_4^{7}x_5^{5}x_6^{4}},$$
$$\frac{(x_1x_3 + x_2x_4)^{4}(x_2x_5 + x_1x_6)^{6}(x_3x_5 + x_4x_6)^{7}}{x_1^{7}x_2^{6}x_3^{6}x_4^{6}x_5^{4}x_6^{4}},$$
$$\frac{(x_1x_3 + x_2x_4)^{2}(x_2x_5 + x_1x_6)^{4}(x_3x_5 + x_4x_6)^{4}}{x_1^{4}x_2^{4}x_3^{3}x_4^{4}x_5^{2}x_6^{2}},$$
$$\frac{(x_1x_3 + x_2x_4)^{2}(x_2x_5 + x_1x_6)^{4}(x_3x_5 + x_4x_6)^{4}(x_1x_3x_6 + x_2x_3x_5 + x_2x_4x_6)}{x_1^{5}x_2^{4}x_3^{4}x_4^{4}x_5^{3}x_6^{2}},$$
$$\frac{(x_1x_3 + x_2x_4)^{3}(x_2x_5 + x_1x_6)^{4}(x_3x_5 + x_4x_6)^{5}(x_1x_3x_6 + x_2x_3x_5 + x_2x_4x_6)}{x_1^{6}x_2^{5}x_3^{4}x_4^{5}x_5^{3}x_6^{3}}, \mathrm{~and~}$$
$$\frac{(x_1x_3 + x_2x_4)^{3}(x_2x_5 + x_1x_6)^{4}(x_3x_5 + x_4x_6)^{5}}{x_1^{5}x_2^{5}x_3^{4}x_4^{4}x_5^{3}x_6^{2}}$$
containing 393216, 131072, 1024, 3072, 12288, and 4096 terms, respectively.
\end{example}

\begin{figure}
\includegraphics[width=5in]{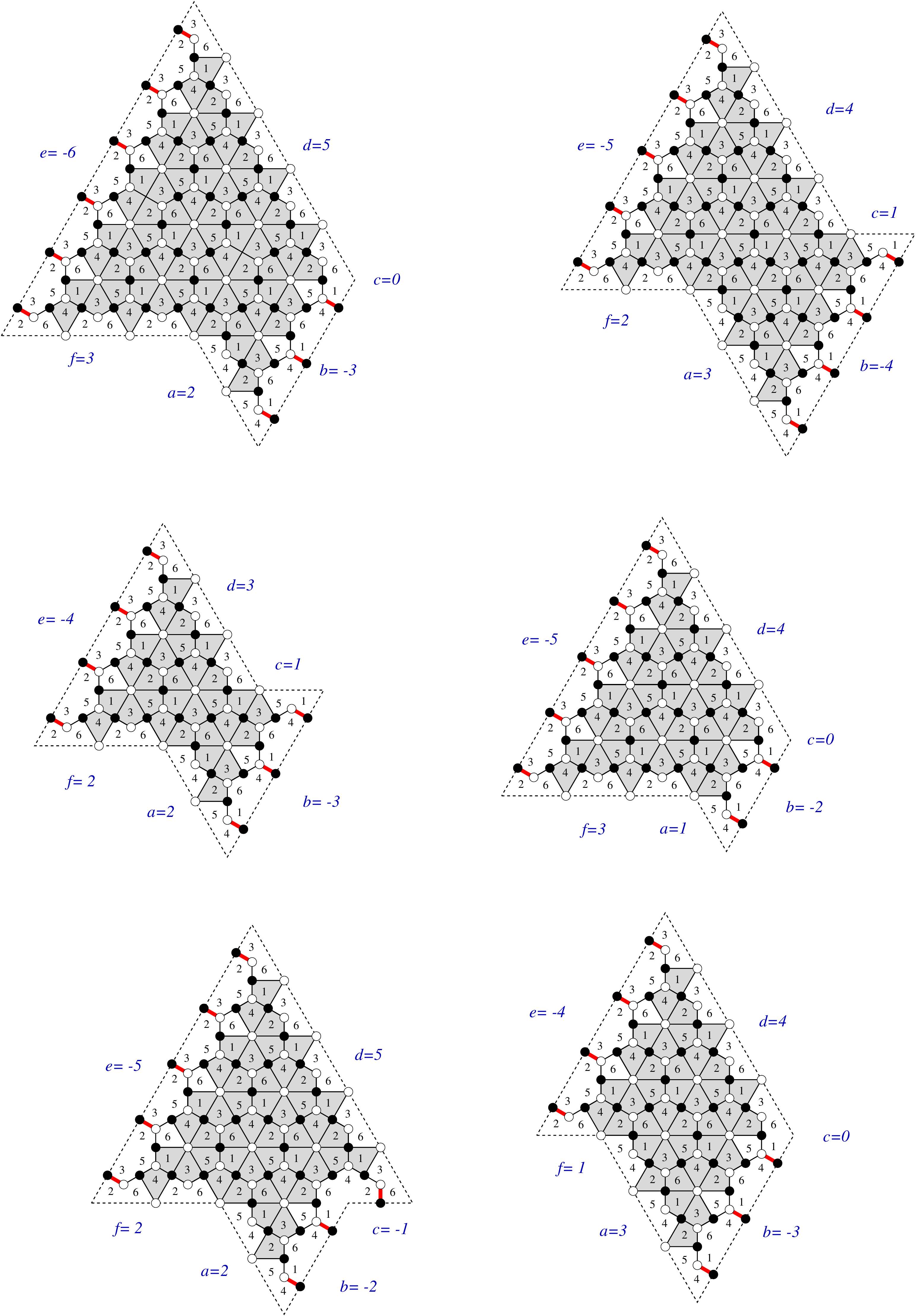}
\caption{The subgraphs obtained from the generalized $\tau$-mutation sequence of Example \ref{ex:1}.}
\label{fig:ex1}
\end{figure}

\begin{example} \label{ex:2}
As another example, we consider the generalized $\tau$-mutation sequence given by $S=\tau_1\tau_2\tau_3\tau_1 \tau_3\tau_2\tau_1 \tau_4\tau_5$.
We have $S_1 = \tau_1\tau_2\tau_3\tau_1 \tau_3\tau_2\tau_1$. Following the corresponding alcove walk, we reach the triangle $\{(2,1),(1,3),(1,1)\}$ in $L^\Delta$. Applying $\tau_4$ yields \[[ \mathcal{C}_2^{1,-1}, \mathcal{C}_2^{1,0}, \mathcal{C}_1^{2,0}, \mathcal{C}_1^{2,-1},
\mathcal{C}_1^{1,-1}, \mathcal{C}_1^{1,0}],\] and $\tau_5$ afterwards implies
$$[\mathcal{C}(0, -2, 1, 3, -5, 4),  \mathcal{C}(-1, -1 , 0, 4, -6, 5), \mathcal{C}(0, -1, -1, 5, -6, 4),$$
$$\hspace{3em}\mathcal{C}(1, -2, 0, 4, -5, 3), \mathcal{C}(0, -1, 0, 3, -4, 3), \mathcal{C}(-1, 0, -1, 4, -5, 4)].$$
These six  contours respectively correspond to the six subgraphs of Figure \ref{fig:ex2} with 3072, 27648, 27648, 3072, 96, and 864 perfect matchings, respectively.  Again this matches the number of terms in the corresponding cluster variables which we have omitted due to their size.
\end{example}

\begin{figure}
\includegraphics[width=5in]{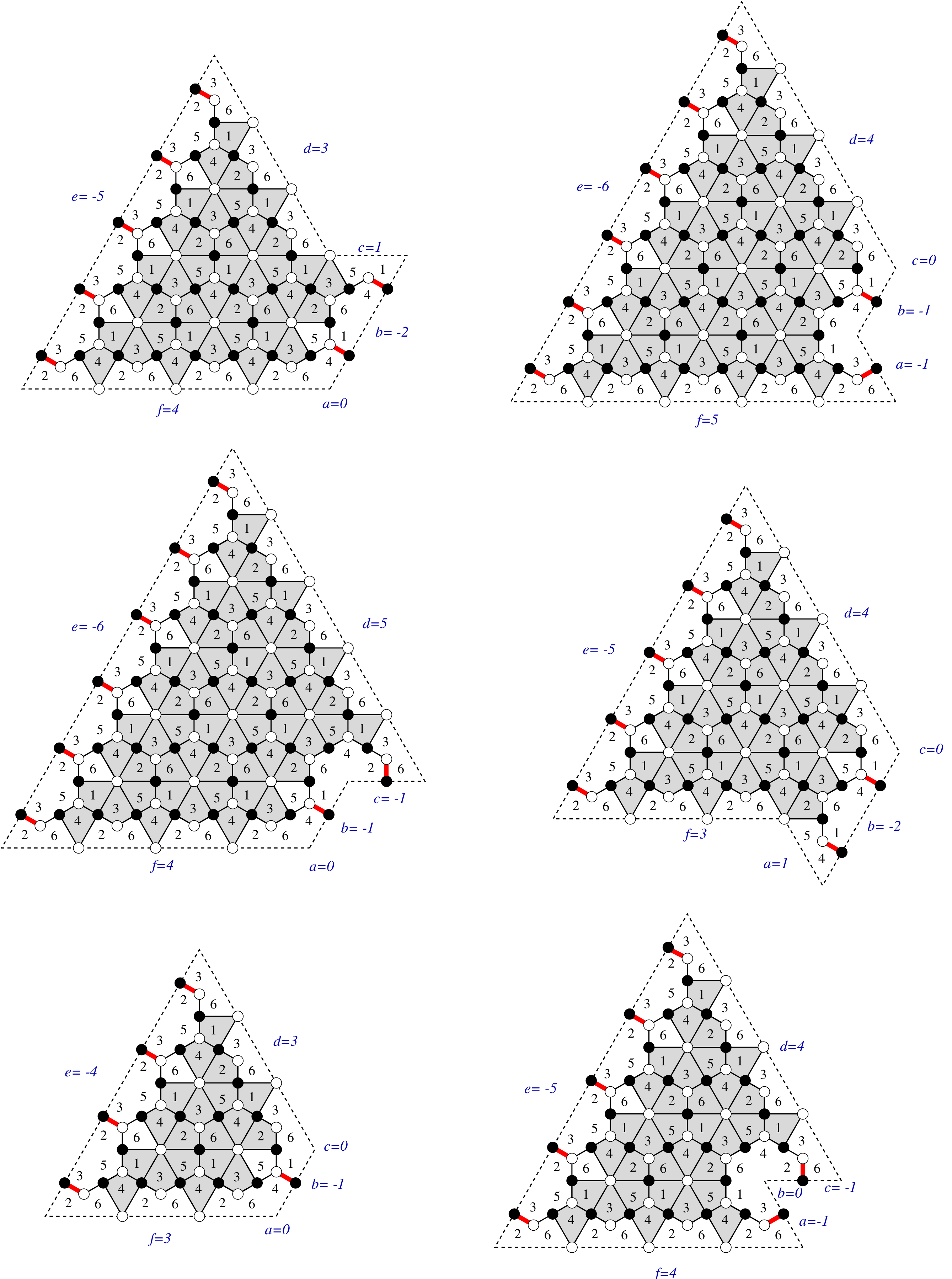}
\caption{The subgraphs obtained from the generalized $\tau$-mutation sequence of Example \ref{ex:2}.}
\label{fig:ex2}
\end{figure}

\section{Preparations for the Proof of Theorem \ref{thm:main}} \label{sec:directions}

In the proof of Theorem \ref{thm:explicit} in Section 3, we used three types of recurrences, e.g. (\ref{eq:Rec1}), (\ref{eq:kRec2}), and (\ref{eq:kRec3}). These recurrences correspond to the lattice moves in Figure \ref{fig:ModelIII}. We call them (R4), (R1) and (R2), respectively,  using the notations of \cite{LaiNewDungeon}.

The $(R4)$ recurrences correspond to replacing the cluster variable $z_i^{j,k}$ with one of twelve possibilities: $z_{i-1}^{j+2, k \pm 1}$, $z_{i+1}^{j-2,k \pm 1}$, $z_{i+2}^{j-1, k\pm 1}$, $z_{i-2}^{j+1, k\pm 1}$,
$z_{i-1}^{j-1, k\pm 1}$, or $z_{i+1}^{j+1, k\pm 1}$.  In terms of the $6$-tuples, this corresponds to replacing $(a,b,c,d,e,f)$ with
$$(a+1, b, c-2, d+3, e-2, f)$$ or a cyclic rotation or negation of this transformation.

The $(R1)$ recurrences  correspond to replacing the cluster variable $z_i^{j,k}$ with
$z_{i+1}^{j,k\pm 2}$, $z_{i-1}^{j,k\pm 2}$, $z_{i}^{ j+1, k\pm 2}$, $z_{i}^{ j-1, k \pm 2}$, $z_{i+1}^{j-1, k \pm 2}$ or $z_{i-1}^{j+1, k\pm 2}$.  In terms of $6$-tuples, the transformation is a cyclic rotation or negation of $$(a-2, b+3, c-3, d+2, e-1, f+1).$$

Finally, the $(R2)$ recurrences correspond to replacing the cluster variable $z_i^{j,k}$ with
$z_{i + 2}^{j, k}$, $z_{i - 2}^{j, k}$, $z_{i}^{j+2, k}$, $z_{i}^{j-2, k}$, $z_{i + 2}^{j-2, k}$, or $z_{i - 2}^{j+2, k}$.  This transformation is a cyclic rotation or negation of $$(a-2,b+2,c,d-2,e+2,f).$$  Notice that in this case, there are only three distinct cyclic rotations.

Comparing with the geometry in Section \ref{sec:gentoric}, any toric mutation corresponds to one of these transformations.  To recover the associated binomial exchange relations, we build a (possibly degenerate) octahedron in $\mathbb{Z}_3$ using the lattice points $(i,j,k)$ and $(i',j',k')$ as its antipodes. More precisely, see Figure \ref{fig:degoct}.

\vspace{2em}

\begin{figure}
\includegraphics[width=12cm]{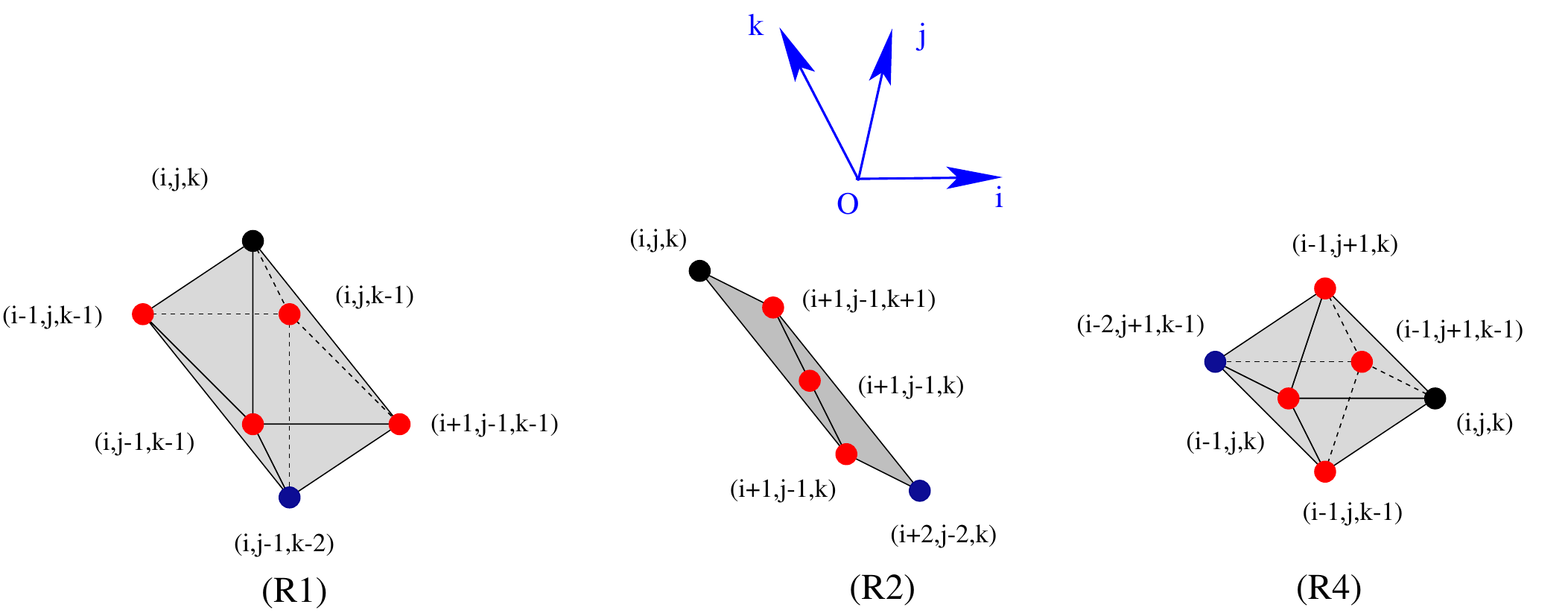}
\caption{Examples of (possibly degenerate) octahedron in $\mathbb{Z}^3$ induced by the lattice points $(i,j,k)$ and $(i',j',k')$.}
\label{fig:degoct}
\end{figure}

These algebraic recurrences agree with the three-term recurrences that appear in Eric Kuo's theory of graphical condensation \cite{kuo1,kuo2}.  This point of view is used in Section \ref{Sec:proof} to present the proof of Theorem \ref{thm:main}, which links the algebraic expressions of cluster variables as Laurent polynomials to a combinatorial interpretation as partition functions of perfect matchings of certain graphs.  Towards this end, we introduce a way to identify six special points in the graph $\mathcal{G}(a,b,c,d,e,f)$ and how to use these to build related subgraphs.

For the sides $a,b,c,d,e,f$, we define respectively the points $A,B,C,D,E,F$ as follows. The region restricted by the contour $\mathcal{C}(a,b,c,d,e,f)$ can be partitioned into equilateral triangles consisting of the faces $1, 4, 5$ or of the faces $2,3,6$. There are $|a|$ such triangles with bases resting on the side $a$ (see the the shaded triangles in Figures \ref{fig:part1}, \ref{fig:part2}, and \ref{fig:part3}).  If $a$ is positive, the point $A$ can be any of $|a|$ big black vertices located at the center of the near side of a shaded triangle as we follow the direction of the side $a$ (see the first picture of Figure \ref{fig:part1}). If $a$ is negative, we can pick $A$ as any of the $|a|$ big white vertices (that are the tops of the shaded triangles with bases resting on the side $a$ as in first picture in Figure \ref{fig:part3}). Similarly, we define the points $B,C,D,E,F$,  based on Figures \ref{fig:part1}, \ref{fig:part2}, and \ref{fig:part3}.

\begin{figure}
\includegraphics[width=4in]{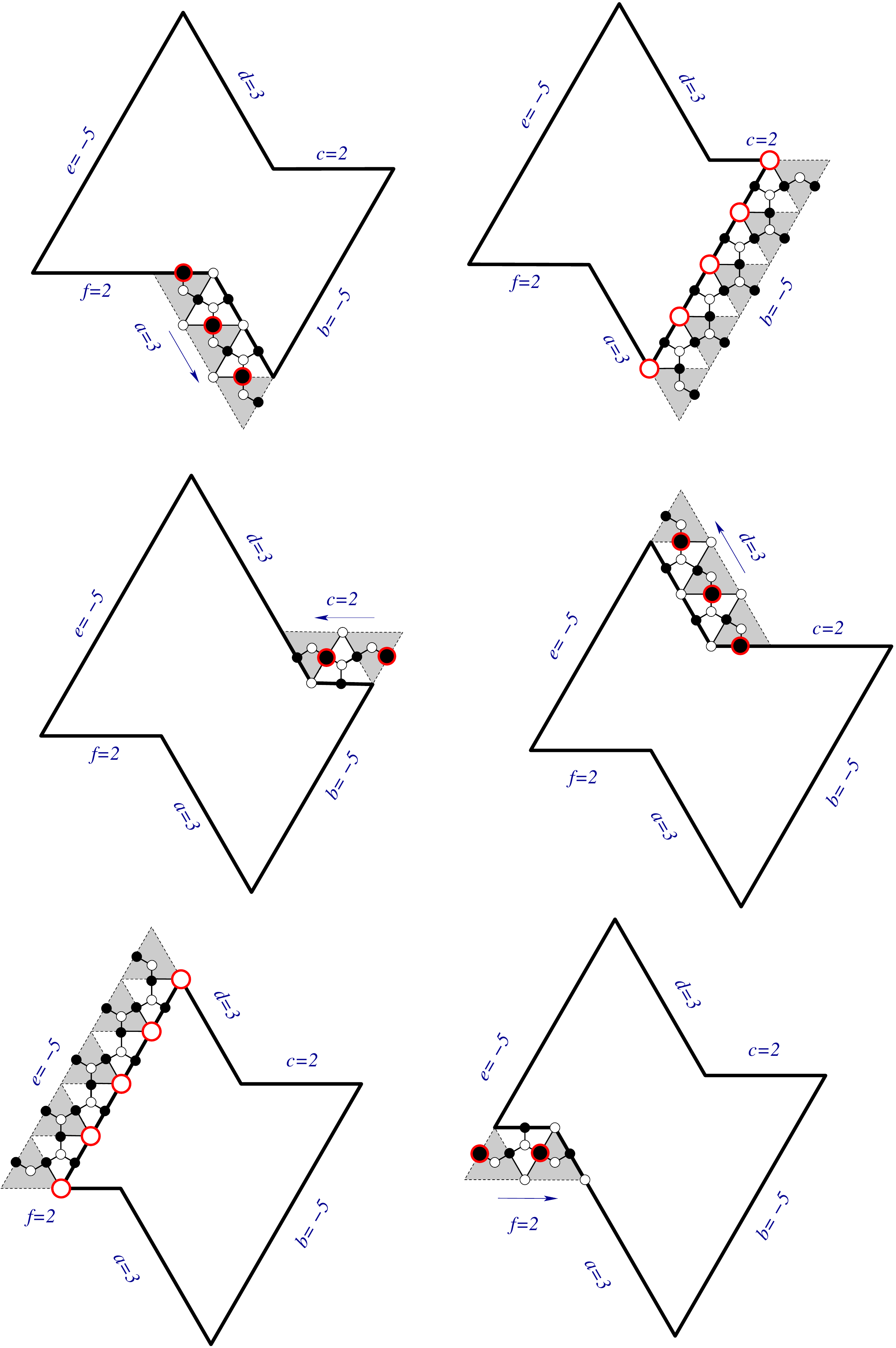}
\caption{How we pick the points $A,B,C,D,E,F$ in the case when $(a,b,c,d,e,f) = (+,-,+, +, -, +)$.}
\label{fig:part1}
\end{figure}

\begin{figure}
\includegraphics[width=4in]{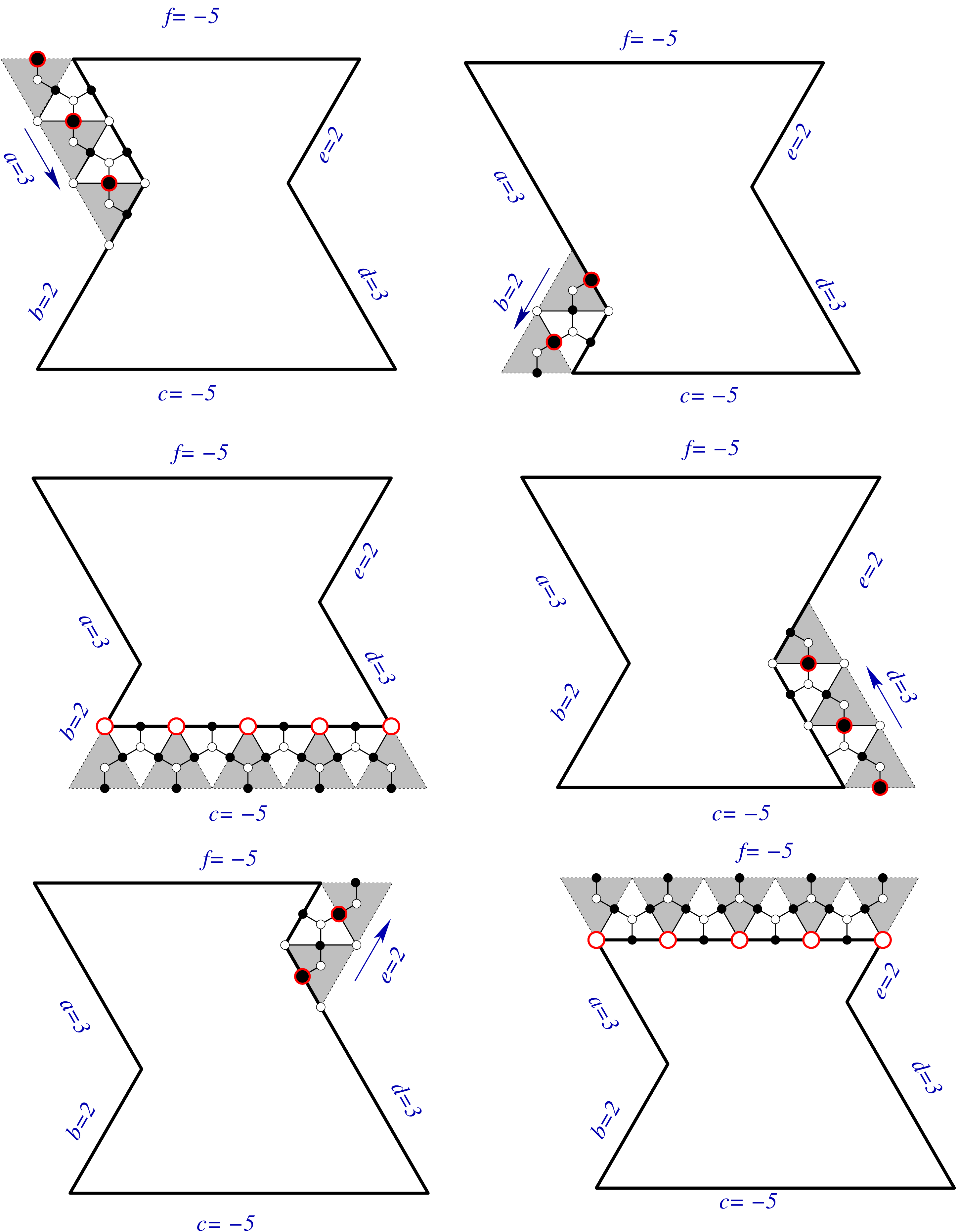}
\caption{How we pick the points $A,B,C,D,E,F$ in the case when $(a,b,c,d,e,f) = (+,+,-, +, +, -)$.}
\label{fig:part2}
\end{figure}

\begin{figure}
\includegraphics[width=5in]{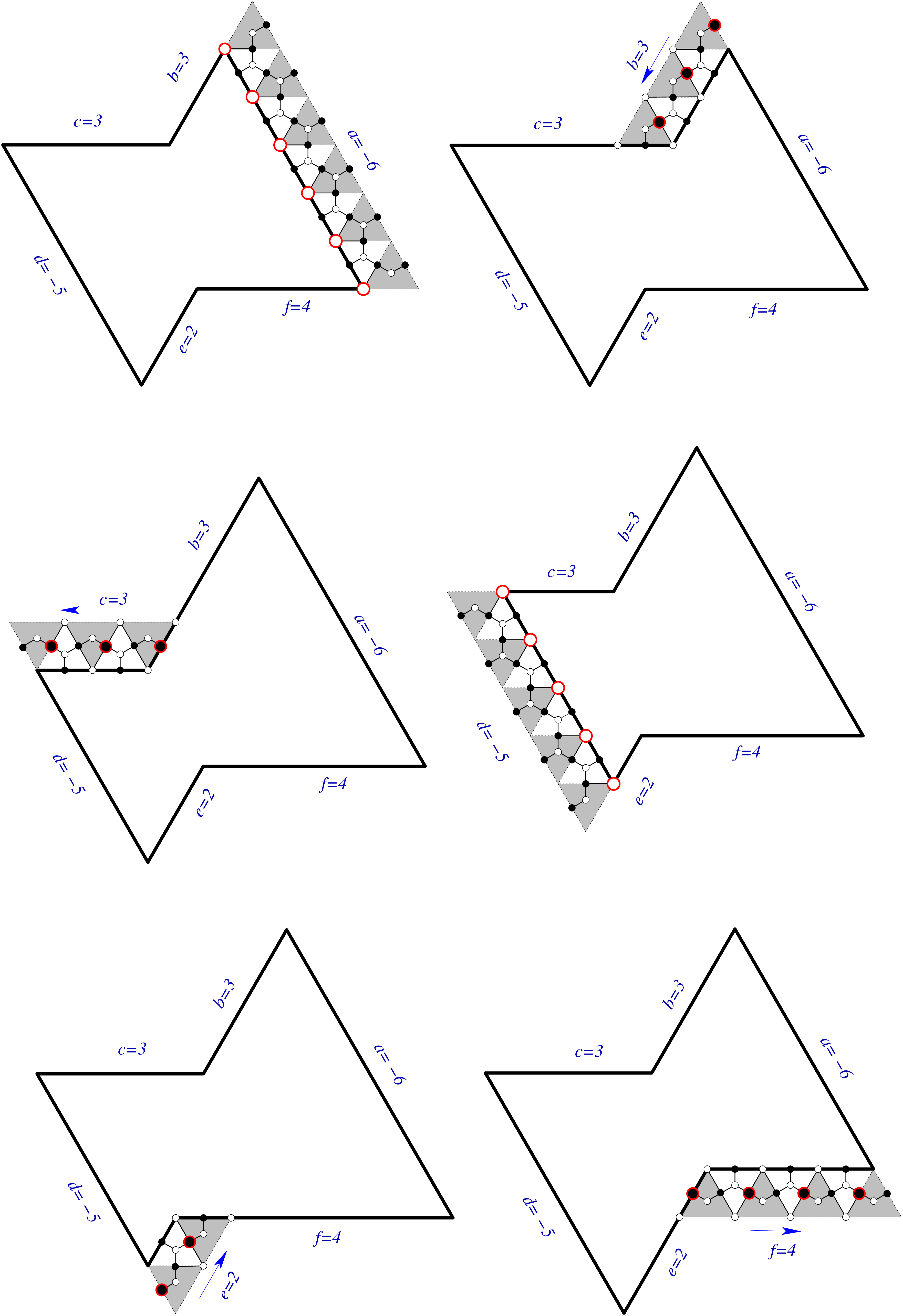}
\caption{How we pick the points $A,B,C,D,E,F$ in the case when $(a,b,c,d,e,f) = (-,+,+, -, +, +)$.}
\label{fig:part3}
\end{figure}

\begin{lemma} \label{claim:signtuple}
Given a subgraph $G$ of the $dP_3$ lattice corresponding to the contour\\ $\mathcal{C}(a,b,c,d,e,f)$ which contains the point $A$, the subgraph $G-\{A\}$ corresponds to the contour $\mathcal{C}(a-1, b+1, c, d, e, f+1)$ (resp.   $\mathcal{C}(a+1, b-1, c, d, e, f-1)$) if the point $A$  is black (resp. white).  Analogous results hold for points $B$, $C$, $D$, $E$, and $F$ up to cyclic rotations.
\end{lemma}

\begin{proof}
 The removal of $A$ from $G$ yields several edges which are forced in any perfect matching of the resulting graph. By removing these forced edges, we obtain a subgraph which coincides with the graph associated to the contour $\mathcal{C}(a-1, b+1, c, d, e, f+1)$ (see the first picture in Figure \ref{fig:forced}).

If $A$ is white, the removal of $A$ also yields several forced edges. By removing these forced edges, we get the graph associated to the contour $\mathcal{C}(a+1, b-1, c, d, e, f-1)$ (see the second picture in Figure \ref{fig:forced}). The removal of the forced edges also yields the removal of the trapezoid consisting of $2|a|-1$ equilateral triangles along the side $a$.

The arguments for $B$, $C$, $D$, $E$, and $F$ can be obtained by an analogous manner, based on Figures \ref{fig:forced} and \ref{fig:forced2}.  Figure \ref{fig:forced3} illustrates that the same procedure still works even in the degenerate case when some of the sides are of length $\pm 1$.
\end{proof}

\begin{figure}\centering
\includegraphics[width=4.5in]{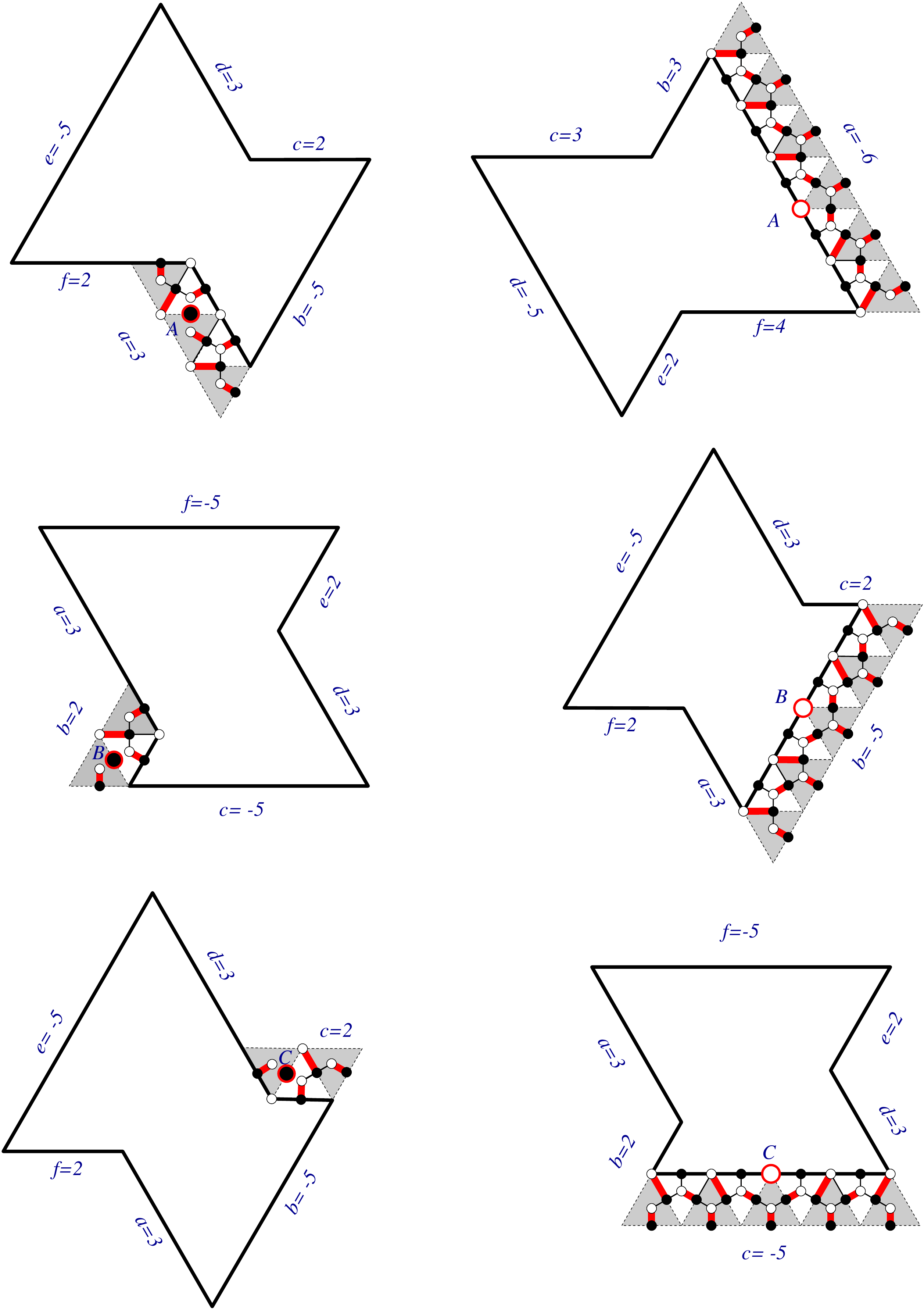}
\caption{Removal of $A$, $B$ and $C$ yields forced edges.}
\label{fig:forced}
\end{figure}

\begin{figure}
\includegraphics[width=4.5in]{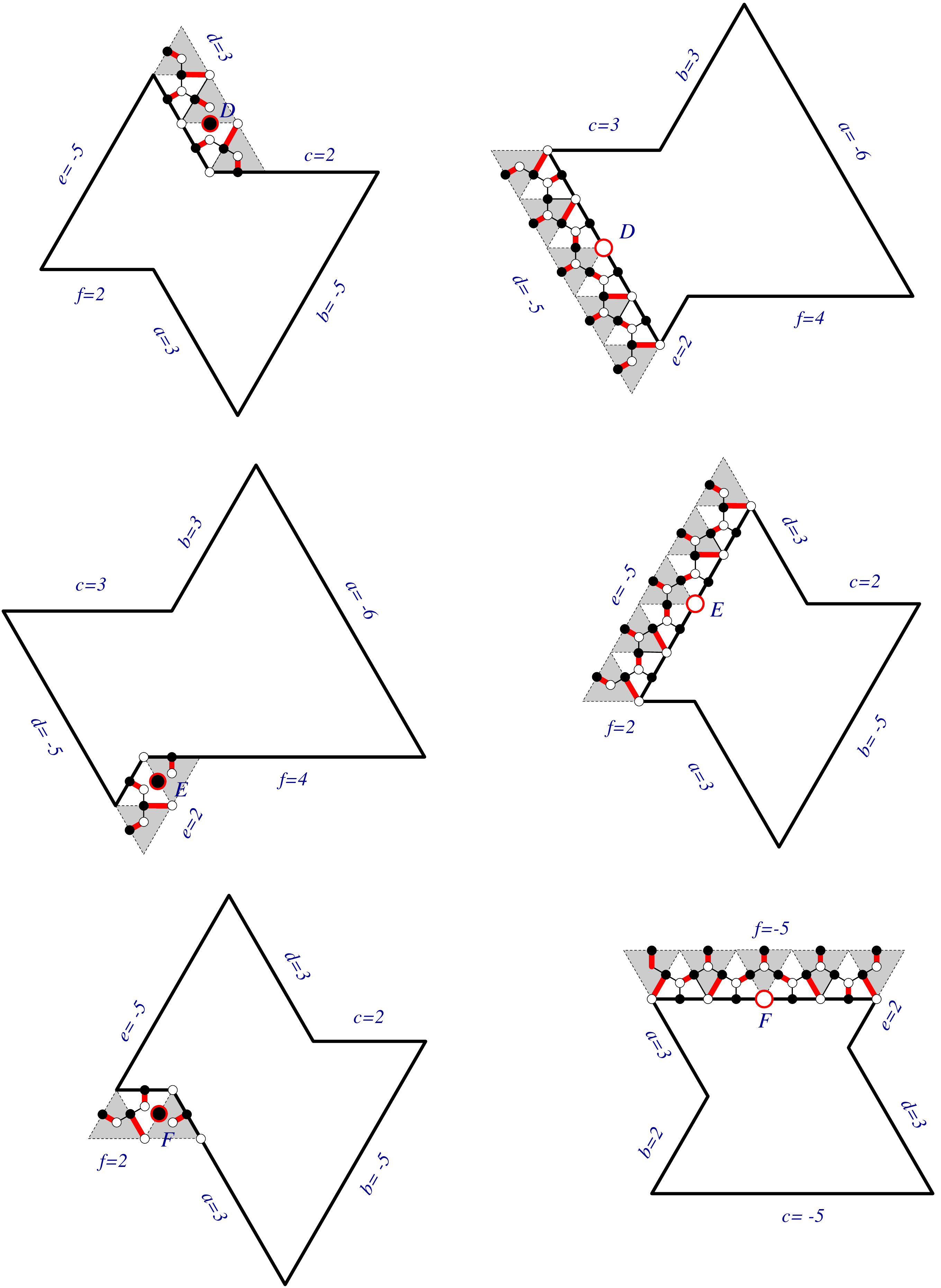}
\caption{Removal of $D$, $E$ and $F$ yields forced edges.}
\label{fig:forced2}
\end{figure}

\begin{figure}
\includegraphics[width=5in]{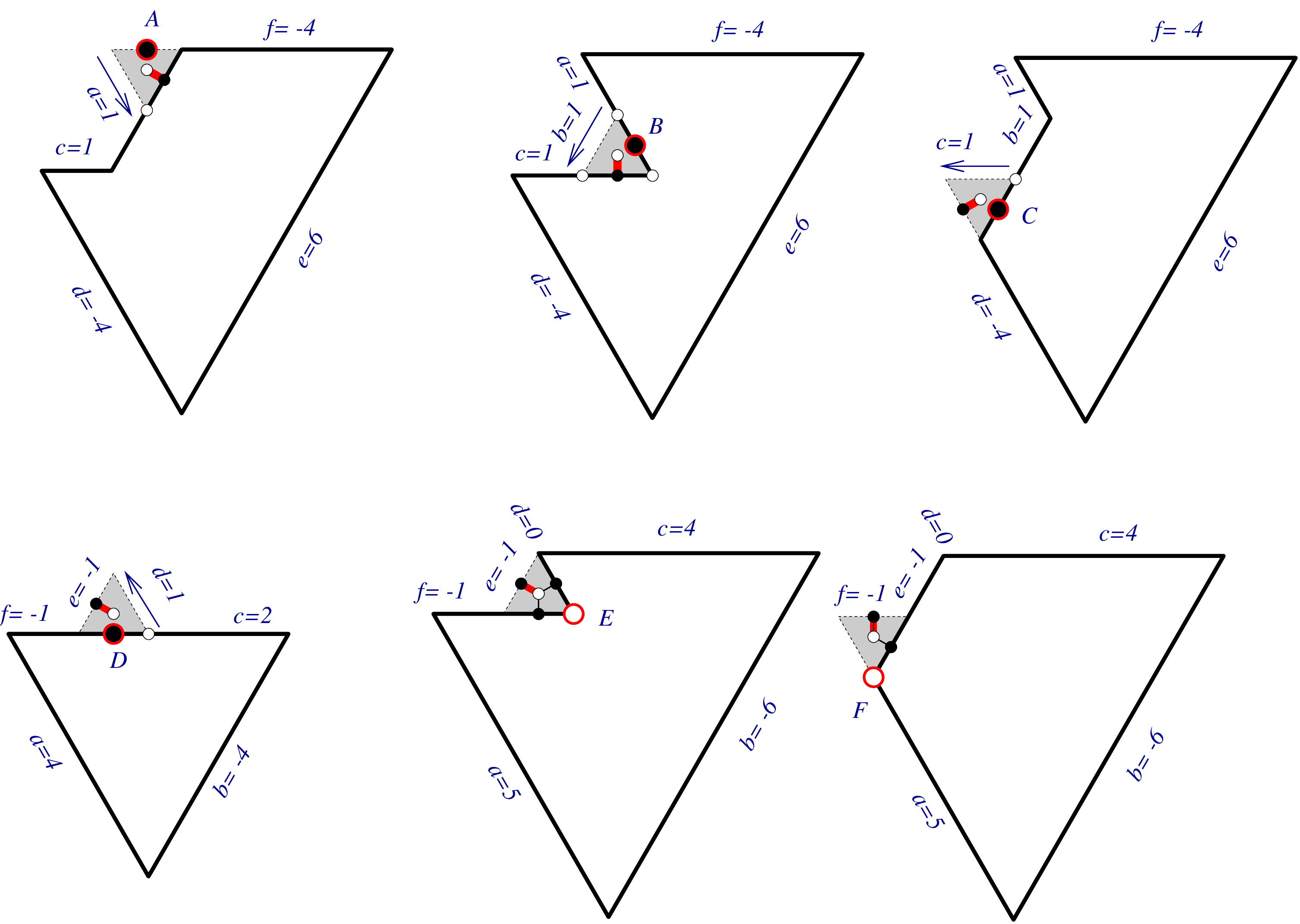}
\caption{Removal of $A,B,C,D,E,F$ when some sides equal $\pm 1$.}
\label{fig:forced3}
\end{figure}

Kuo condensation was first used by Eric H. Kuo \cite{kuo1} to (re)prove the Aztec diamond theorem by Elkies, Kuperberg, Larsen, and Propp \cite{Elkies}. Kuo condensation can be considered as a combinatorial interpretation of Dodgson condensation (or the Jacobi-Desnanot identity; see e.g  \cite{Mui}, pp.136--149) on determinants of matrices. See e.g. \cite{YZ}, \cite{kuo2}, \cite{Ciucu}, \cite{Ful}, \cite{speyer} for different aspects and generalizations of the method, and see e.g. \cite{lai'}, \cite{KW}, \cite{CF}, \cite{Trin} for recent applications of Kuo condensation.

In [Kuo], Kuo presented several different versions of Kuo condensation. For ease of reference, we list below the four versions employed in our proofs.

\begin{lemma}[Balanced Kuo Condensation; Theorem 5.1 in \cite{kuo1}]\label{Kuo1}
Let $G=(V_1,V_2,E)$ be a (weighted) planar bipartite graph with $|V_1|=|V_2|$. Assume that $p_1,p_2,p_3,p_4$ are four vertices appearing in a cyclic order on a face of $G$. Assume in addition that $p_1,p_3\in V_1$ and $p_2,p_4\in V_2$. Then
\begin{align}
w(G)w(G-\{p_1,p_2,p_3,p_4\})=&w(G-\{p_1,p_2\})w(G-\{p_3,p_4\})\notag\\
&+w(G-\{p_1,p_4\})w(G-\{p_2,p_3\}).
\end{align}
\end{lemma}

\begin{lemma}[Unbalanced Kuo Condensation; Theorem 5.2 in \cite{kuo1}]\label{Kuo2}
Let $G=(V_1,V_2,E)$ be a planar bipartite graph with $|V_1|=|V_2|+1$. Assume that $p_1,p_2,p_3,p_4$ are four vertices appearing in a cyclic order on a face of $G$. Assume in addition that $p_1,p_2,p_3\in V_1$ and $p_4\in V_2$. Then
\begin{align}
w(G-\{p_2\})w(G-\{p_1,p_3,p_4\})=&w(G-\{p_1\})w(G-\{p_2,p_3,p_4\})\notag\\
&+w(G-\{p_3\})w(G-\{p_1p_2,p_4\}).
\end{align}
\end{lemma}

\begin{lemma}[Non-alternating  Kuo Condensation;  Theorem 5.3 in \cite{kuo1}]\label{Kuo3}
Let $G=(V_1,V_2,E)$ be a planar bipartite graph with $|V_1|=|V_2|$. Assume that $p_1,p_2,p_3,p_4$ are four vertices appearing in a cyclic order on a face of $G$. Assume in addition that $p_1,p_2\in V_1$ and $p_3,p_4\in V_2$. Then
\begin{align}
w(G-\{p_1,p_4\})w(G-\{p_2,p_3\})=&w(G)w(G-\{p_1,p_2,p_3,p_4\})\notag\\
&+w(G-\{p_1,p_3\})w(G-\{p_2,p_4\}).
\end{align}
\end{lemma}

\begin{lemma}[Monochromatic  Kuo Condensation;  Theorem 5.4 in \cite{kuo1}]\label{Kuo4}
Let $G=(V_1,V_2,E)$ be a planar bipartite graph with $|V_1|=|V_2|+2$. Assume that $p_1,p_2,p_3,p_4$ are four vertices appearing in a cyclic order on a face of $G$. Assume in addition that $p_1,p_2,p_3,p_4\in V_1$. Then
\begin{align}
w(G-\{p_1,p_3\})w(G-\{p_2,p_4\})=&w(G-\{p_1,p_2\})w(G-\{p_3,p_4\})\notag\\
&+w(G-\{p_2,p_3\})w(G-\{p_4,p_1\}).
\end{align}
\end{lemma}

\section{Proof of Theorem \ref{thm:main}} \label{Sec:proof}

We begin by verifying that $Z^\emptyset = [x_1,x_2,x_3,x_4,x_5,x_6]$ equals
$$[c(\mathcal{G}(C_1)), c(\mathcal{G}(C_2)), c(\mathcal{G}(C_3)), c(\mathcal{G}(C_4)), c(\mathcal{G}(C_5)), c(\mathcal{G}(C_6))],$$
or equivalently, by Remark \ref{rem:cov}, $$[c(\widetilde{\mathcal{G}}(C_1)), c(\widetilde{\mathcal{G}}(C_2)), c(\widetilde{\mathcal{G}}(C_3)), c(\widetilde{\mathcal{G}}(C_4)), c(\widetilde{\mathcal{G}}(C_5)), c(\widetilde{\mathcal{G}}(C_6))],$$ where the $C_i$'s are the initial contours from (\ref{contours1})-(\ref{contours3}).

Each  graph $\widetilde{\mathcal{G}}(C_i)$ is a unit triangle in the brane tiling $\mathcal{T}$ with all but two vertices removed.  In all six of these cases, there is a single edge remaining which is the lone contribution to a perfect matching of $\widetilde{\mathcal{G}}(C_i)$.  Thus multiplying the covering monomial and weight of the unique perfect matching together we get
$$m(\widetilde{\mathcal{G}}(C_1)) = m(\widetilde{\mathcal{G}}(C_4)) = m(\widetilde{\mathcal{G}}(C_5))= x_1x_4x_5,$$
$$m(\widetilde{\mathcal{G}}(C_2)) = m(\widetilde{\mathcal{G}}(C_3)) = m(\widetilde{\mathcal{G}}(C_6)) = x_2x_3x_6,$$
$$w(\widetilde{\mathcal{G}}(C_1)) = \frac{1}{x_4x_5},~~ w(\widetilde{\mathcal{G}}(C_2)) = \frac{1}{x_3x_6},~~ w(\widetilde{\mathcal{G}}(C_3)) = \frac{1}{x_2x_6},$$
$$w(\widetilde{\mathcal{G}}(C_4)) = \frac{1}{x_1x_5},~~ w(\widetilde{\mathcal{G}}(C_5)) = \frac{1}{x_1x_4},~~ w(\widetilde{\mathcal{G}}(C_6)) = \frac{1}{x_2x_3}.$$
In all these cases $c(\mathcal{G}(C_i)) =c(\widetilde{\mathcal{G}}(C_i))= m(\widetilde{\mathcal{G}}(C_i))w(\widetilde{\mathcal{G}}(C_i)) = x_i$ as desired.
See Figure \ref{fig:C12}.
This verifies the desired combinatorial interpretation of the cluster variables $z_{i}^{j,k}$ for $(i,j,k) \in \{
(0,-1,0), (0,-1,1), (-1,0,0), (-1,0,1), (0,0,0), (0,0,1)\}$.

\begin{figure}
\includegraphics[width=1.5in]{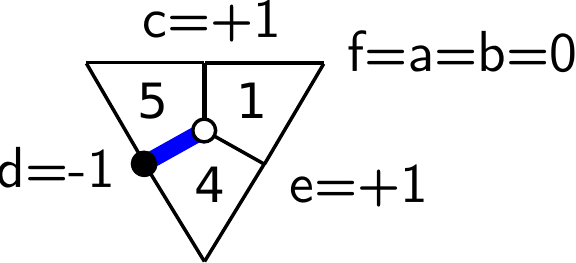} \hspace{3em}
\includegraphics[width=1.5in]{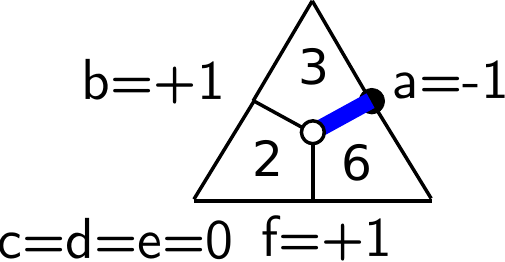}
\caption{Examples of $\widetilde{\mathcal{G}}(C_1)$ and $\widetilde{\mathcal{G}}(C_2)$.}
\label{fig:C12}
\end{figure}

The proof continues by induction.  As in Section \ref{sec:directions}, toric mutations correspond to $30$ possible transformations $(i,j,k) \to (i',j',k')$.   It suffices to turn each of these geometric moves into an algebraic recurrence.  We proceed to accomplish this as follows:

Step 1: Let $(d_A,d_B,d_C,d_D,d_E,d_F)$ denote the difference between the six-tuple given as $\phi(i,j,k) = (a,b,c,d,e,f)$ and the six-tuple $\phi(i',j',k') = (a',b',c',d',e',f')$. Based on the description in Section \ref{sec:directions}, the $30$ possibilities for $(d_A,d_B,d_C,d_D,d_E,d_F)$ are
$\pm (1, 0, -2, 3, -2, 0)$, $\pm (-2, 1 ,-1, 2, -3, 3)$, $\pm (2,-2, 0, 2, -2, 0)$, or one of these up to rotation, categorized as (R4), (R1), or (R2), respectively.  We let $\mathcal{C}$ and $\mathcal{C}'$ denote the contours
$$\mathcal{C} = {\mathcal{C}}(a,b,c,d,e,f) \mathrm{~~and~~} \mathcal{C}' = {\mathcal{C}}(a',b',c',d',e',f').$$
Step 2:  We create a new contour $\mathcal{O}$ by superimposing a shift of $\mathcal{C}'$ on top of $\mathcal{C}$ and drawing the contour obtained by taking the outer boundary.  We shift according to the following rules:

(a) If $d_A \geq 2$ (resp. $d_A \leq -2$) we first shift $\mathcal{C}'$ to the left (resp. right) by $(-1,0)$ (resp. $(1,0)$), i.e. parallel to side $f$.

(b) If $d_F \geq 2$ (resp. $d_F \leq -2$) we instead or then shift $\mathcal{C}'$ diagonally by $(\frac{1}{2}, -\frac{\sqrt{3}}{2})$ (resp. $(-\frac{1}{2}, \frac{\sqrt{3}}{2})$), i.e. parallel to side $a$.

(c) In the cases when $|d_A|$ and $|d_F| \geq 2$, both of these two shifts will occur.  However we observe that $d_A$ and $d_F$ cannot have the same sign and thus we get a total shift of $\mathcal{C}'$ of either $(\frac{3}{2}, -\frac{\sqrt{3}}{2})$ or $(-\frac{3}{2}, \frac{\sqrt{3}}{2})$.

(d) In the remaining cases, we also shift $\mathcal{C}'$ by $(-1,0)$ (resp. $(1,0)$) if $d_A = 1$ (resp. $d_A = -1$) and $d_F = 0$.  Similarly, we shift by
$(\frac{1}{2}, -\frac{\sqrt{3}}{2})$ (resp. $(-\frac{1}{2}, \frac{\sqrt{3}}{2})$) if $d_A = 0$ and $d_F = 1$ (resp. $d_F=-1$).  Lastly, we do not shift
$\mathcal{C}'$ at all if $\{d_A,d_F\} = \{-1,1\}$.

Based on these rules, the local configuration around the corner where sides $f$ and $a$ meet in $\mathcal{C}$ (and sides $f'$ and $a'$ meet in $\mathcal{C}'$) are one of the possibilities illustrated in Figure \ref{fig:AA} (if $|d_A| = 2$, $|d_A|= 3$, or both $d_A = \pm 1$ and $d_F=0$), Figure \ref{fig:FF} (if $|d_F| = 2$, $|d_F|= 3$, or both $d_F = \pm 1$ and $d_A=0$), Figure \ref{fig:NN} (if neither of these happen), or Figure \ref{fig:AF} (if both of these conditions occur).

For brevity, we only illustrate the cases where $d_A \geq 0$ and $d_F \leq 0$ since allowing $d_A < 0$ or $d_F > 0$ merely switches the role of $\mathcal{C}$ and $\mathcal{C}'$.  We also only illustrate the case where side lengths $a$ and $f$ are positive (except in Figure \ref{fig:NN}) since changing their signs or values does not affect the relative position of the endpoints of sides $a$, $a'$, $f$, and $f'$.  This is clear in Figure \ref{fig:NN} and extends to the other cases as well.

\begin{figure}
\includegraphics[width=4.5in]{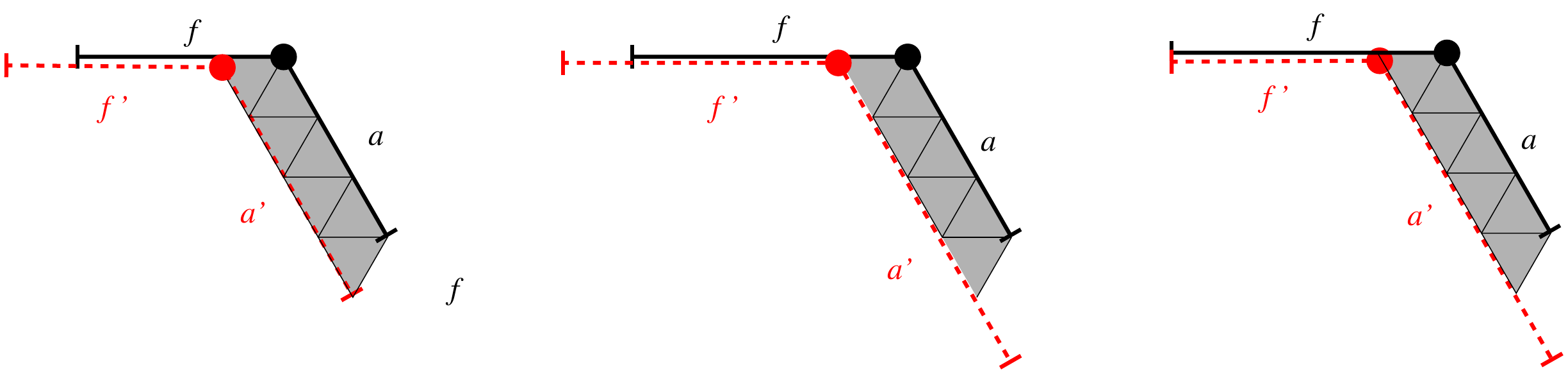}
\caption{The cases with a shift of $(-1,0)$ for $\mathcal{C}'$.  From Left to Right: (i) $d_F=0$, $d_A=1$, (ii) $d_F=0$, $d_A=2$, and (iii) $d_F=-1$, $d_A=2$.}
\label{fig:AA}
\end{figure}

\begin{figure}
\includegraphics[width=4.5in]{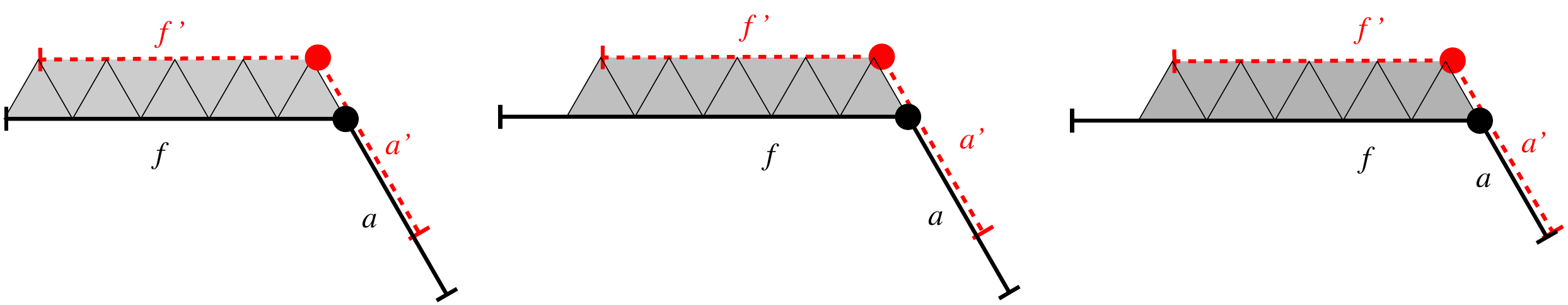}
\caption{The cases with a shift of $(1/2,-\sqrt{3}/2)$ for $\mathcal{C}'$.  From Left to Right: (i) $d_F=-1$, $d_A=0$, (ii) $d_F=-2$, $d_A=0$, and (iii) $d_F=-2$, $d_A=1$.}
\label{fig:FF}
\end{figure}

\begin{figure}
\includegraphics[width=4.5in]{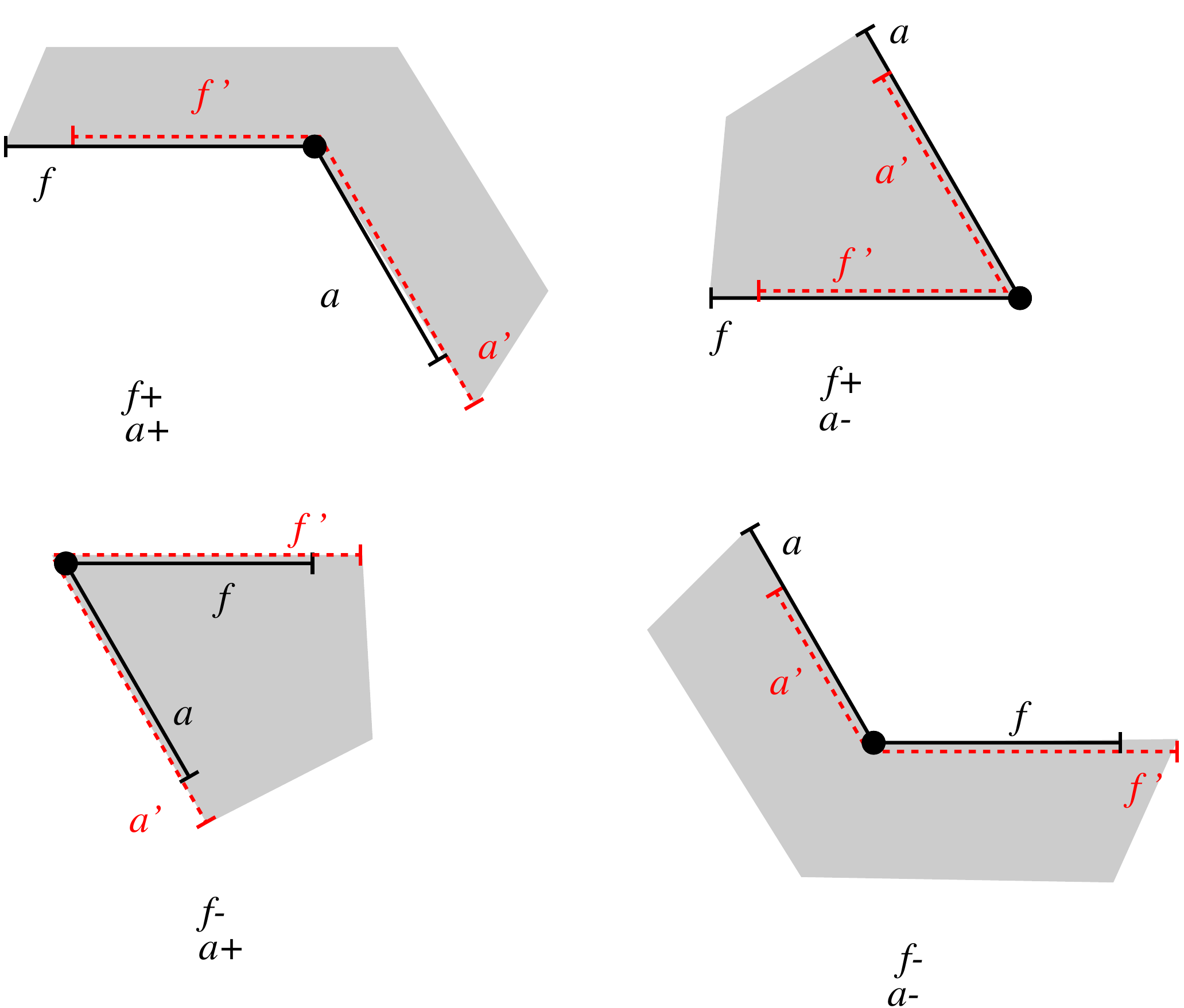}
\caption{The case with no shift.  There are only two possibilities, $(d_F,d_A)=(1,-1)$ or $(d_F,d_A)=(-1,1)$.  We illustrate the latter but unlike Figures \ref{fig:AA}, \ref{fig:FF}, and \ref{fig:AF}, illustrate the shape of the corner as $a$ and $f$ vary in sign.  }
\label{fig:NN}
\end{figure}

\begin{figure}
\includegraphics[width=4.5in]{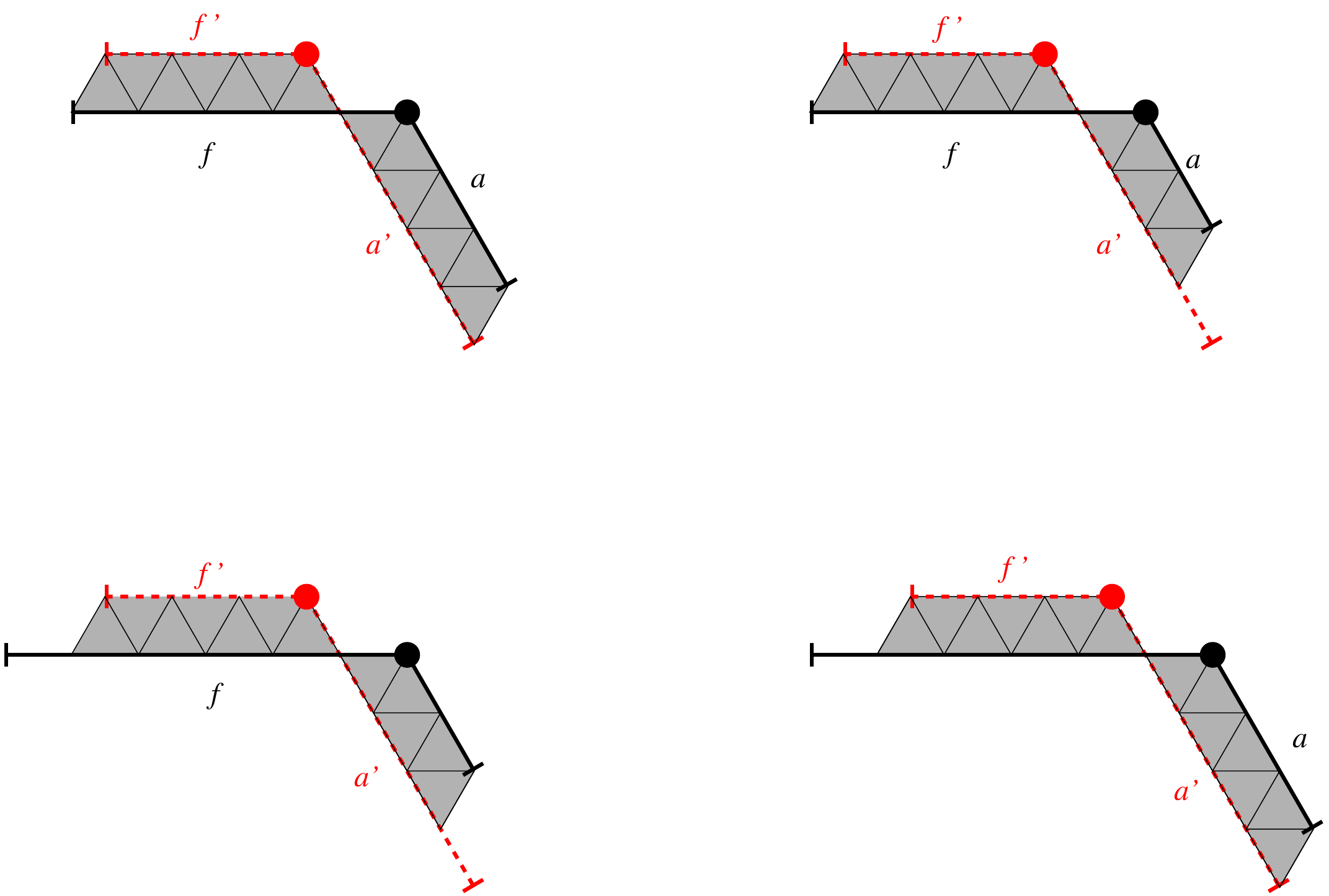}
\caption{Lastly, we illustrate the cases with a shift of $(3/2,-\sqrt{3}/2)$ for $\mathcal{C}'$.  (i) In the top-left corner, $d_F=-2$, $d_A=2$, (ii) in the top-right, $d_F=-2$, $d_A=3$, (iii)
in the bottom-left, $d_F=-3$, $d_A=3$, and (iv) in the bottom-right, $d_F=-3, d_A=2$.}
\label{fig:AF}
\end{figure}

Step 3: By inspecting these figures, we claim that this outer contour $\mathcal{O}$ either coincides or is shifted one unit away from the contour $\mathcal{C}$ (resp. $\mathcal{C}'$) along side $a$ (resp. $a'$), as well as $f$ (resp. $f'$).  Its behavior is also completely determined by the pair $(|d_A|, |d_F|)$.

Considering instead the local configuration in the neighborhoods of the corners where $a$ and $b$ meet in $\mathcal{C}$ (as well as where $a'$ and $b'$ meet in $\mathcal{C}'$), we obtain rotations of Figures \ref{fig:AA}, \ref{fig:FF}, \ref{fig:NN}, and \ref{fig:AF} by $60^\circ$ clockwise.  Based on the possible six-tuples for $(d_A,d_B,d_C,d_D,d_E,d_F)$ given in Step 1, we see that the ordered pairs $(d_F,d_A)$ lead to the possibilities for $(d_A,d_B)$ as given by the following tables\footnote{We only show the cases where $d_A \geq 0$ and $d_F \leq 0$ since the case where $d_A < 0$ or $d_F > 0$ is analogous.}:

$$\begin{array}{r|c|c|c|c|c|c}
(d_F,d_A) & (0,1) & (0,2) &                       (-1,2) & (-1,0) & (-2,0) &                        (-2,1)  \\
\hline (d_A,d_B) & (1,0) & (2,-2) \mathrm{~or~} (2,-3) & (2,-3) & (0,2) & (0,1) \mathrm{~or~} (0,2) & (1,-1)
\end{array}$$

$$\begin{array}{r|c|c|c|c|c}
(d_F,d_A) &  (-1,1) & (-2, 2) & (-2, 3) &                     (-3, 3) & (-3,2) \\
\hline (d_A,d_B) &  (1,-2) & (2, 0) & (3, -2) \mathrm{~or~} (3,-3) & (3, -2) & (2,0) \mathrm{~or~} (2,-1)
\end{array}$$

Based on the rules of Step 2 and these two tables, we conclude that these two local configurations can be consistently glued together into a shape involving the three sides $\{f,a,b\}$ (resp. $\{f',a',b'\}$).  Inductively, all six sides can be glued together in this way, and we see from these figures that the outer contour $\mathcal{O}$ is built by taking the longer side at each corner (when the two parallel sides do not overlap).  It follows that $\mathcal{O}$ is a (six-sided) contour just as defined in the beginning of Section \ref{Sec:Contours}.

Step 4: Comparing the outer contour $\mathcal{O}$ to the contours $\mathcal{C}$ and $\mathcal{C}'$, respectively, some of the sides of $\mathcal{O}$ lie to the strict left of $\mathcal{C}$ (call these positive), others to the strict right (call these negative).  We compare $\mathcal{O}$ and $\mathcal{C}'$ analogously.

As explained in Step 3, it is sufficient to look at the six ordered pairs
$(d_F,d_A)$, $(d_A,d_B)$, $(d_B,d_C)$, $(d_C,d_D)$, $(d_D,d_E)$, and $(d_E,d_F)$ to determine which sides are positive, which sides are negative, and which are neither.  Based on the description in Section \ref{sec:directions} and in Step 1, we conclude that in all cases, there are exactly four sides which are positive or negative.  Hence by taking the appropriate linear combination of $6$-tuples associated to sides, as appearing in Lemma \ref{claim:signtuple}, we can recover the contour $\mathcal{C}$ or the contour $\mathcal{C}'$ from $\mathcal{O}$.

Step 5: By case-by-case analysis, we see that the possibilities for $(d_A,d_B,d_C,d_D,d_E,d_F)$ discussed in Step 1 can be written as the linear combination
$$c_A (-1,1,0,0,0,1) + c_B(1,-1,1,0,0,0) + c_C(0,1,-1,1,0,0)$$ $$ \hspace{5em} + ~~ c_D(0,0,1,-1,1,0) + c_E(0,0,0,1,-1,1) + c_F(1,0,0,0,0,1,-1),$$
where two of these coefficients are $+1$, two are $-1$, and two are $0$.
For example, $$(1,0,-2,3,-2,0) \leftrightarrow c_A=-1, c_C=+1, c_D=-1, c_E=+1,$$
$$(-2,1,-1,2,-3,3) \leftrightarrow c_A=+1, c_D=-1, c_E=+1, c_F=-1, \mathrm{~ and}$$
$$(2,-2,0,2,-2,0) \leftrightarrow c_A=-1, c_B=+1, c_D=-1, c_E=+1.$$
The set of nonzero coefficients exactly match up with the set of special sides identified in Step 4.  Thus may also use from Lemma \ref{claim:signtuple} to switch our point of view from the contours $\mathcal{C}(a,b,c,d,e,f)$ to the actual subgraphs $\widetilde{\mathcal{G}}(a,b,c,d,e,f)$.

Consequently, we now build the graph $H$, which is defined as $H= \widetilde{\mathcal{G}}(\mathcal{O})$. We are able to use the methods of Section \ref{sec:directions} to pick four points $X$, $Y$, $W$, $Z$ (in $\{A,B,C,D,E,F\}$) on $H$ with appropriate colors based on the special sides identified in Step 4.  The colors are determined by the sign of the side.  Because the contour $\mathcal{O}$ always overlaps with either $\mathcal{C}$ or $\mathcal{C}'$ (or both) along each side, we get a set partition $\{S_1, S_2\}$ of the points $\{X,Y,W,Z\}$ such that the removal of the points $S_1$ from $H$ yields the graph $G = \widetilde{\mathcal{G}}(\mathcal{C})$ and the removal of the points $S_2$ from $H$ yields the graph $G' = \widetilde{\mathcal{G}}(\mathcal{C}')$.

Step 6: Based on this set partition and the color pattern of the four points involved, one of the four possible versions of Kuo Condensation applies.  By construction, the left-hand-side will involve graphs $G = \widetilde{\mathcal{G}}(\mathcal{C})$ and $G' = \widetilde{\mathcal{G}}(\mathcal{C}')$. The appropriate application of Kuo condensation, i.e. Lemma \ref{Kuo1}, \ref{Kuo2}, \ref{Kuo3}, or \ref{Kuo4}, dictates the appropriate graphs on the right-hand-side accordingly.
Assume that we get the general Kuo recurrence:

\begin{align}\label{eq:kuogen1}
w(H-S_1)w(H-S_2)=w(H-S_3)w(H-S_4)+w(H-S_5)w(H-S_6),
\end{align}
where $S_1,S_2,\dotsc, S_6$ are certain subsets of $\{X,Y,W,Z\}$, and $S_{i+1}=\{X,Y,W,Z\}-S_i$, for $i=1,3,5$. By applying Lemma \ref{claim:signtuple}, we are able to obtain a contour $\mathcal{C}_i$ by adding or subtracting the appropriate six-tuples from $\mathcal{O}$ as dictated by the subset $S_i$.  After removing the forced edges from the graph in the above recurrence, we get the graphs $G=G_1=\widetilde{\mathcal{G}}(\mathcal{C}_1),G'=G_2=\widetilde{\mathcal{G}}(\mathcal{C}_2),G_3=\widetilde{\mathcal{G}}(\mathcal{C}_3),G_4=\widetilde{\mathcal{G}}(\mathcal{C}_4),G_5=\widetilde{\mathcal{G}}(\mathcal{C}_5),G_6=\widetilde{\mathcal{G}}(\mathcal{C}_6)$.

 It is easy to see that the removal of the point $X$ (resp. $Y,W,Z$) reduces the covering monomial of $H$ by exactly one face (see Figure \ref{fig:6-cov-mon}). In particular, this face can be determined uniquely from Figures \ref{fig:forced} and \ref{fig:forced2}. If $X$ (resp. $Y,W,Z$) is white then the face is the one inside the shaded triangle adjacent to $X$ (resp. $Y,W,Z$); if $X$ (resp. $Y,W,Z$) is black then the face is the one inside the shaded triangle adjacent to $X$ (resp. $Y,W,Z$) and the side $x$ (resp. $y,w,z$). We denote this face by $t_X$ (resp. $t_Y,t_W,t_Z$).
In particular, the face $t_X$ has label $2,5,3,1,6,4$ when $X$ is white $A,B,C,D,E,F$, respectively; and $t_X$ is labeled by $5,3,1,6,4,2$ when $X$ is black $A,B,C,D,E,F$, respectively.

\begin{figure}
\includegraphics[width=12cm]{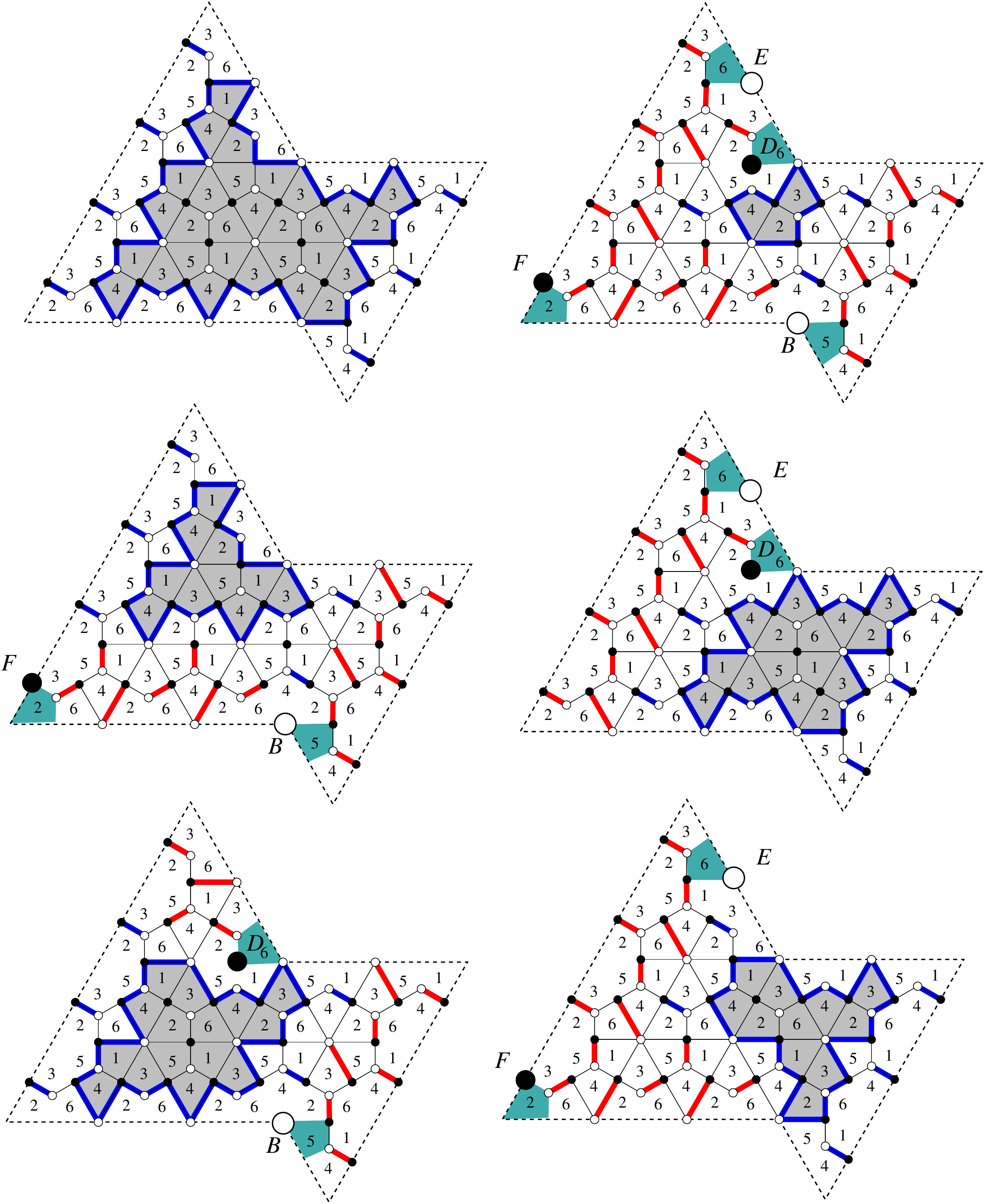}
\caption{Illustrating how covering monomials are affected by the removal of points.}
\label{fig:6-cov-mon}
\end{figure}

Just like the case for $\widetilde{\mathcal{G}}(a,b,c,d,e,f)$, we define the covering monomial, $m(H)$, of $H$ as the product of weight of all faces restricted inside the contour $\mathcal{O}$ associated to $H$. We also define the covering monomial, $m(H-S_i)$, of the graph $H-S_i$ by
\[m(H-S_i)=\frac{m(H)}{wt(S_i)},\]
where $wt(S_i)$ is the product of the weights of all faces corresponding to the vertices in $S_i$; and let $c(H-S_i)=m(H-S_i)w(H-S_i)$.  By definition, we have
\begin{align}
m(H-S_1)m(H-S_2)&=m(H-S_3)m(H-S_4)=m(H-S_5)m(H-S_6)\notag\\
&=\frac{m(H)^2}{wt(t_X)wt(t_Y)wt(t_W)wt(t_Z)}.\end{align}
Thus, (\ref{eq:kuogen1}) is equivalent to
\begin{align}\label{eq:kuogen2}
c(H-S_1)c(H-S_2)=c(H-S_3)c(H-S_4)+c(H-S_5)c(H-S_6).
\end{align}

Since $G_i$ is obtained from $H-S_i$ by removing forced edges, we have $w(H-S_i)/w(G_i)$ is the product of the weights of all forced edges, and $m(H-S_i)/m(G_i)$ is the product of the weights of all faces adjacent to these forced edges. However, the two products cancel each other out. It means that $c(H-S_i)=c(G_i)$. Therefore, equation (\ref{eq:kuogen2}) implies
\begin{align}\label{eq:kuogen3}
c(G)c(G')=c(G_3)c(G_4)+c(G_5)c(G_6).
\end{align}

In conclusion, for any contour without self-intersection, we have $z_{i}^{j,k} =c(\widetilde{\mathcal{G}}(\mathcal{C}_i^{j,k}))=c(\mathcal{G}(\mathcal{C}_i^{j,k})) = z(a,b,c,d,e,f)$ (using the notation of Section \ref{sec:comb}) as desired, finishing the proof of the theorem.

\section{Further Examples} \label{sec:examp}

In this section, we provide a number of graphics illustrating the methods used in Sections \ref{sec:directions} and \ref{Sec:proof} to prove Theorem \ref{thm:main}.
Firstly, we provide a sample of the graphics obtained in \cite{sage} for a variety of examples of the contours $\mathcal{C}$ and shifted $\mathcal{C}'$.  In these examples, we start with $\mathcal{C}=\mathcal{C}_i^{j,k}$ and $\mathcal{C}' = \mathcal{C}_{i+d_i}^{j+d_j,k+d_k}$.  We visualize the superposition and the outer contour $\mathcal{O}$ via the command  ~  {\tt SuperO(i,j,k,[$d_i$,$d_j$,$d_k$])}.  The code also outputs the six-tuples $(a,b,c,d,e,f)$, $(a',b',c',d',e',f')$, the type of recurrence involved, and the associated shift.

\vspace{10em}

\begin{sageblock}
SuperO(-7,-7,-7,[1,-1,2])
\end{sageblock}

\begin{center}\includegraphics[width=4in]{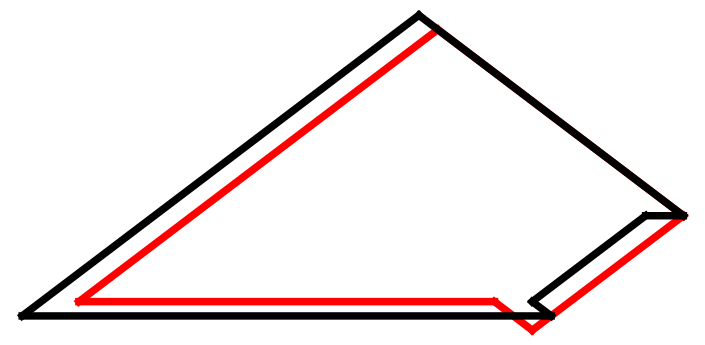}
\end{center}

$$(-14, 21, -14, 1, 6, 1), (-13, 19, -11, -2, 8, 0), R1, (0, 0)$$

\begin{sageblock}
SuperO(-10,-10,8,[-1,2,-1])
\end{sageblock}

\begin{center}\includegraphics[width=4in]{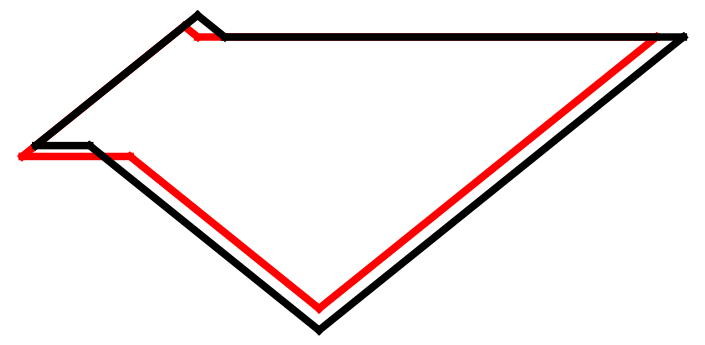}
\end{center}

$$(-2, 12, -2, -17, 27, -17), (-1, 12, -4, -14, 25, -17), R4, (-1,0)$$

\begin{sageblock}
SuperO(-10,-10,8,[-1,2,1])
\end{sageblock}

\begin{center}\includegraphics[width=4in]{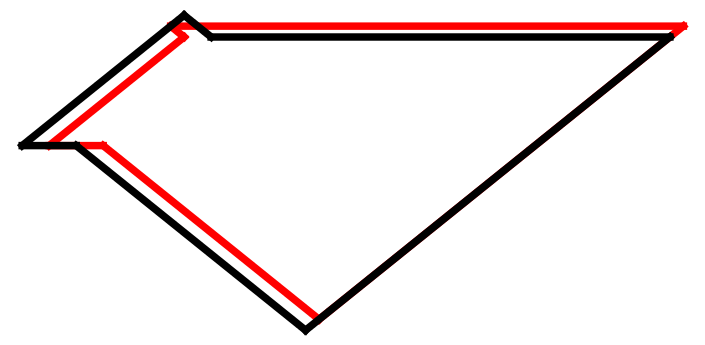}
\end{center}

$$(-2, 12, -2, -17, 27, -17), (1, 10, -2, -16, 27, -19), R4, (-\frac{3}{2},\frac{\sqrt{3}}{2})$$

\vspace{5em}

\begin{sageblock}
SuperO(-4,-10,8,[2,0,0])
\end{sageblock}

\begin{center}\includegraphics[width=4in]{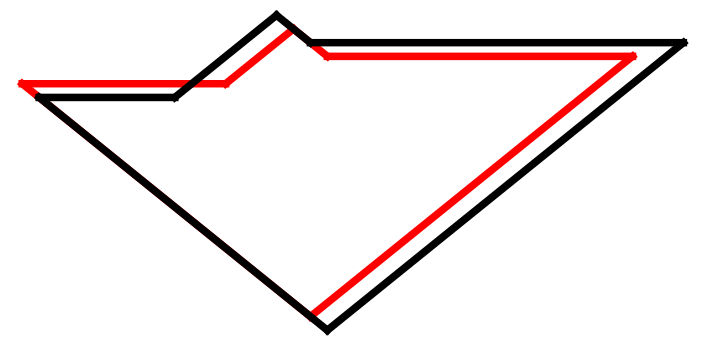}
\end{center}

$$(-2, 6, 4, -17, 21, -11), (-2, 4, 6, -17, 19, -9), R2, (\frac{1}{2},-\frac{\sqrt{3}}{2})$$

\begin{sageblock}
SuperO(8,-9,8,[2,0,0])
\end{sageblock}

\begin{center}\includegraphics[width=4in]{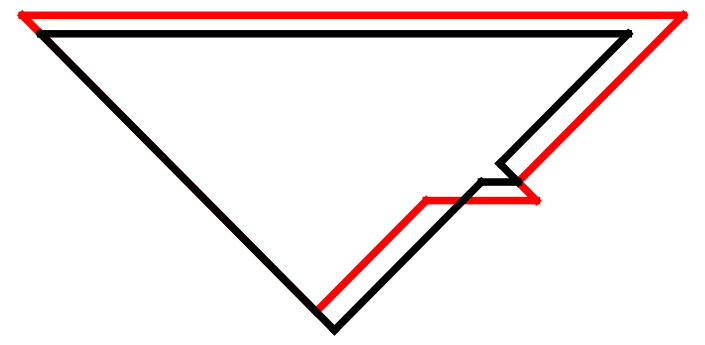}
\end{center}

$$(-1, -7, 16, -16, 8, 1), (-1, -9, 18, -16, 6, 3), R2,  (\frac{1}{2},-\frac{\sqrt{3}}{2})$$

\begin{sageblock}
SuperO(20,-6,8,[1,0,2])
\end{sageblock}

\begin{center}\includegraphics[width=4in]{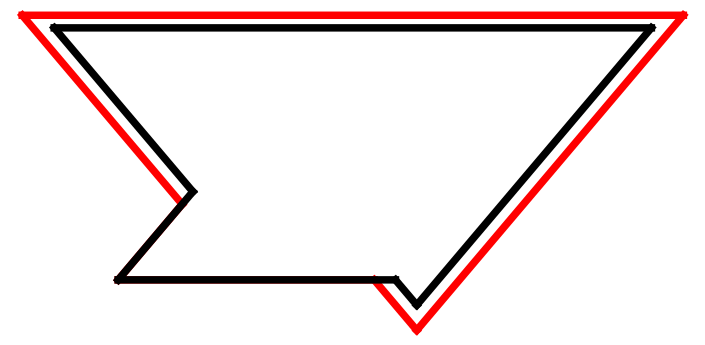}
\end{center}

$$(2, -22, 28, -13, -7, 13), (4, -25, 31, -15, -6, 12), R1, (-1,0)$$

In Figures \ref{fig:egg} - \ref{fig:eggs}, we provide additional examples and fill in the associated contours with the corresponding subgraphs of the brane tiling $\mathcal{T}$.  We illustrate our use of the four types of Kuo Condensation used to get our algebraic recurrences in Step 6 of the proof of Theorem \ref{thm:main}.

\begin{figure}
\includegraphics[width=12cm]{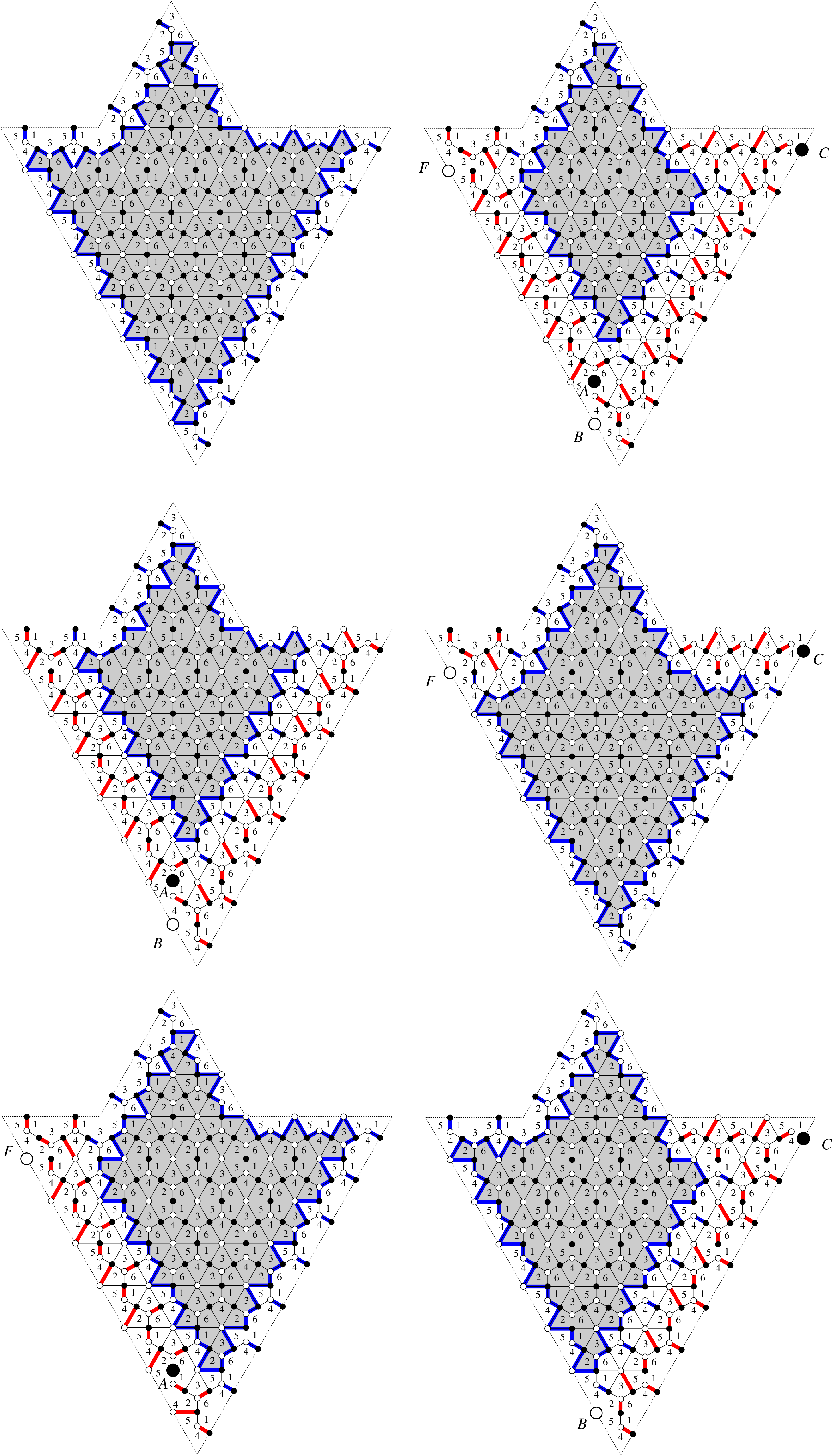}
\caption{Example of the recurrence (R1) using Balanced Kuo Condensation with $i=0, j=5, k=3$.}
\label{fig:egg}
\end{figure}

\begin{figure}
\includegraphics[width=12cm]{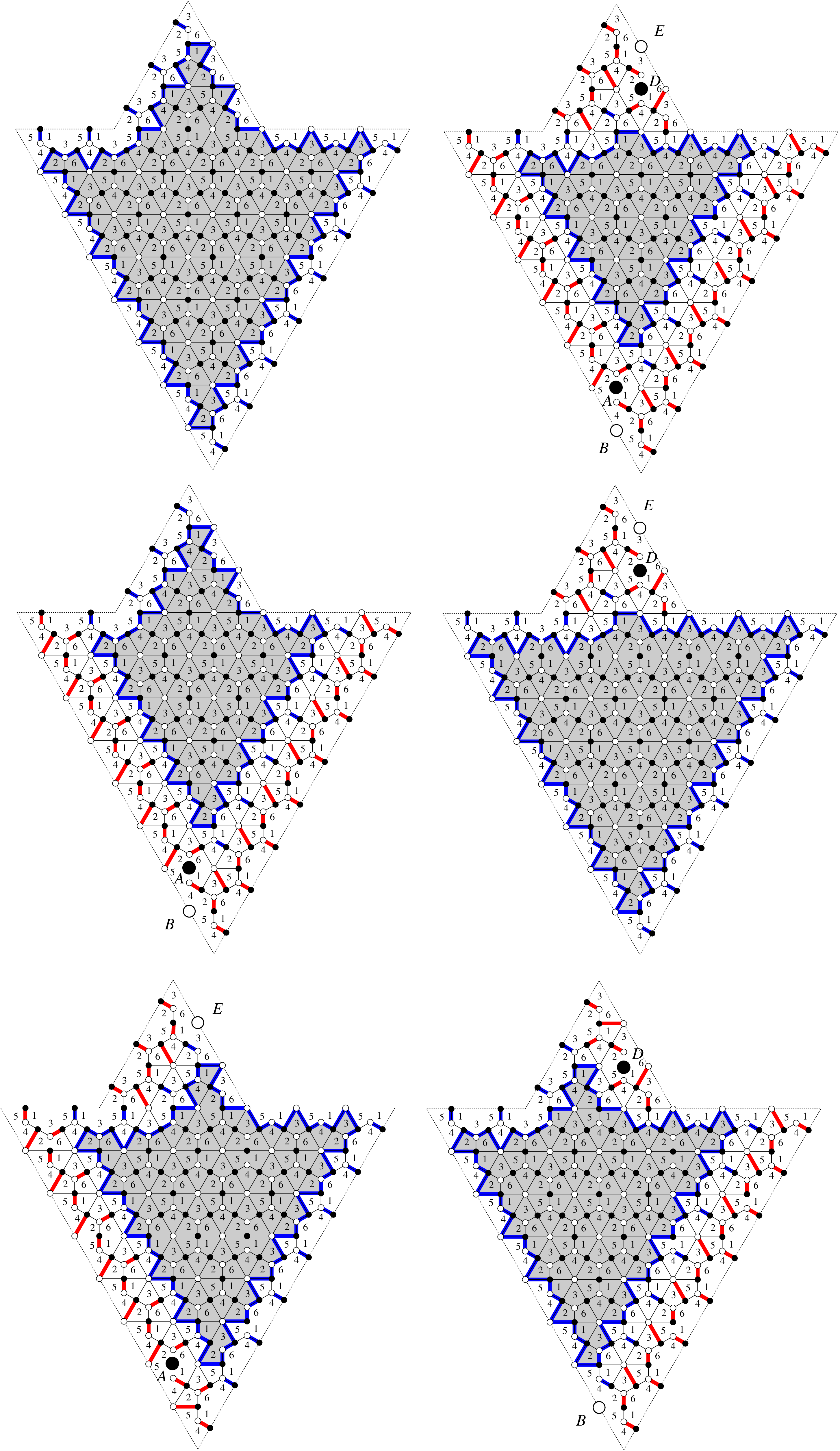}
\caption{Example of the recurrence (R2) using Balanced Kuo Condensation with $i=0, j=5,k=3$.}
\end{figure}

\begin{figure}
\includegraphics[width=12cm]{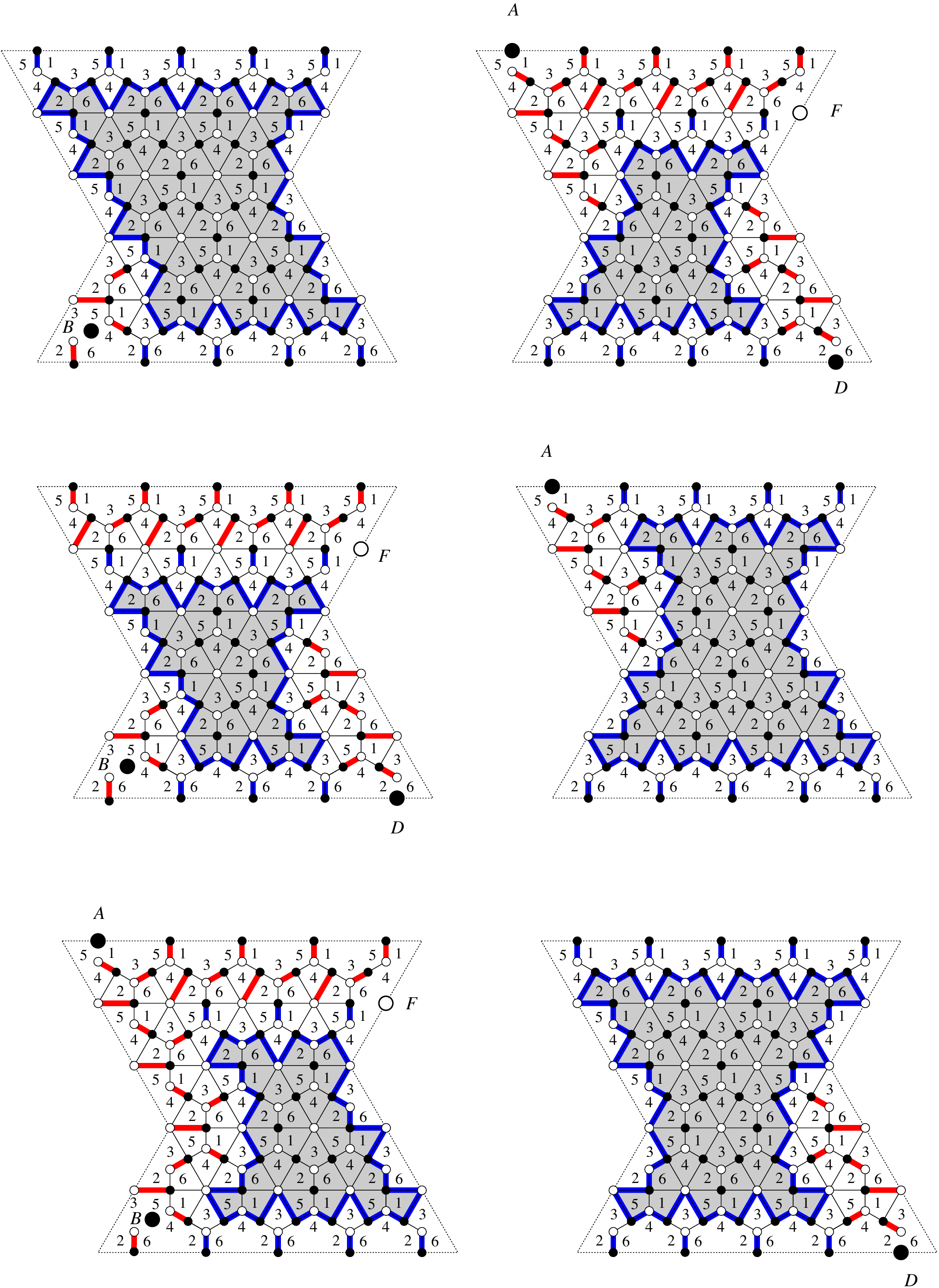}
\caption{Example of  the recurrence (R4) using Unbalanced Kuo Condensation with $i=-5, j=3,k=1$.}
\end{figure}

\begin{figure}
\includegraphics[width=11cm]{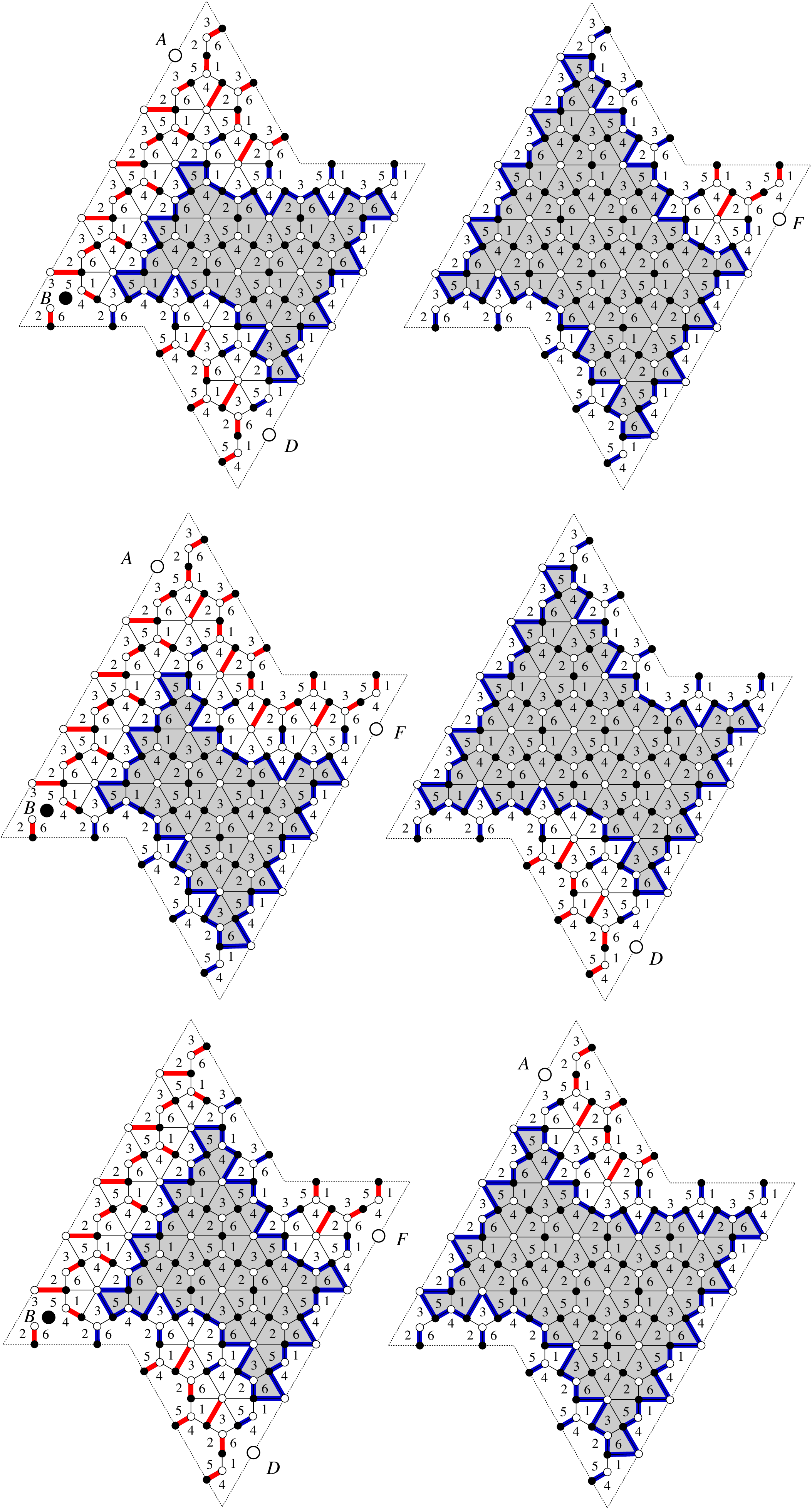}
\caption{Example of (R4) using Unbalanced Kuo Condensation with $i=-3, j=-2,k=1$.}
\end{figure}

\begin{figure}
\includegraphics[width=12cm]{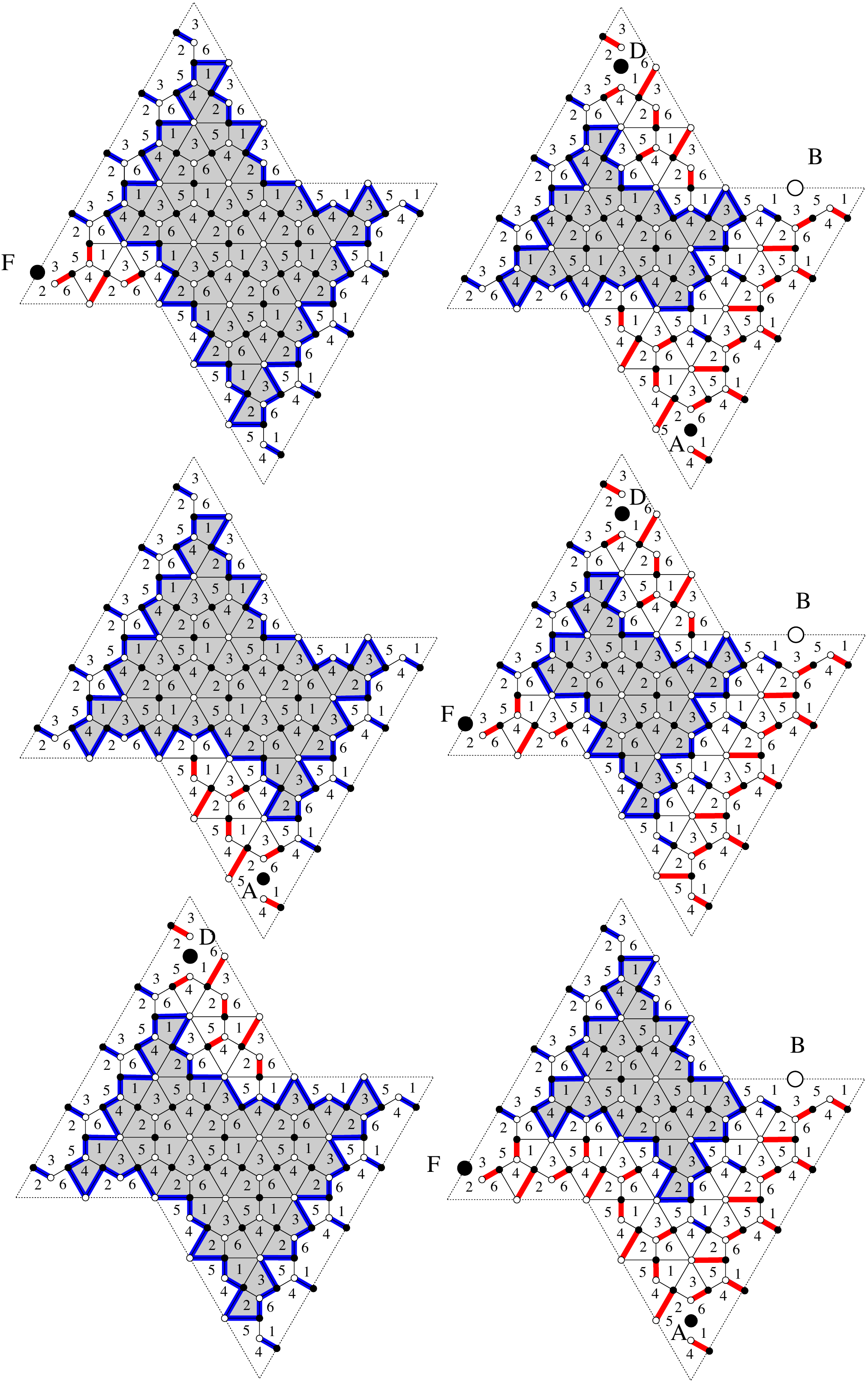}
\caption{Another example of (R4) using Unbalanced Kuo Condensation with $i=1, j=3, k=1$.}
\end{figure}

\begin{figure}
\includegraphics[width=12cm]{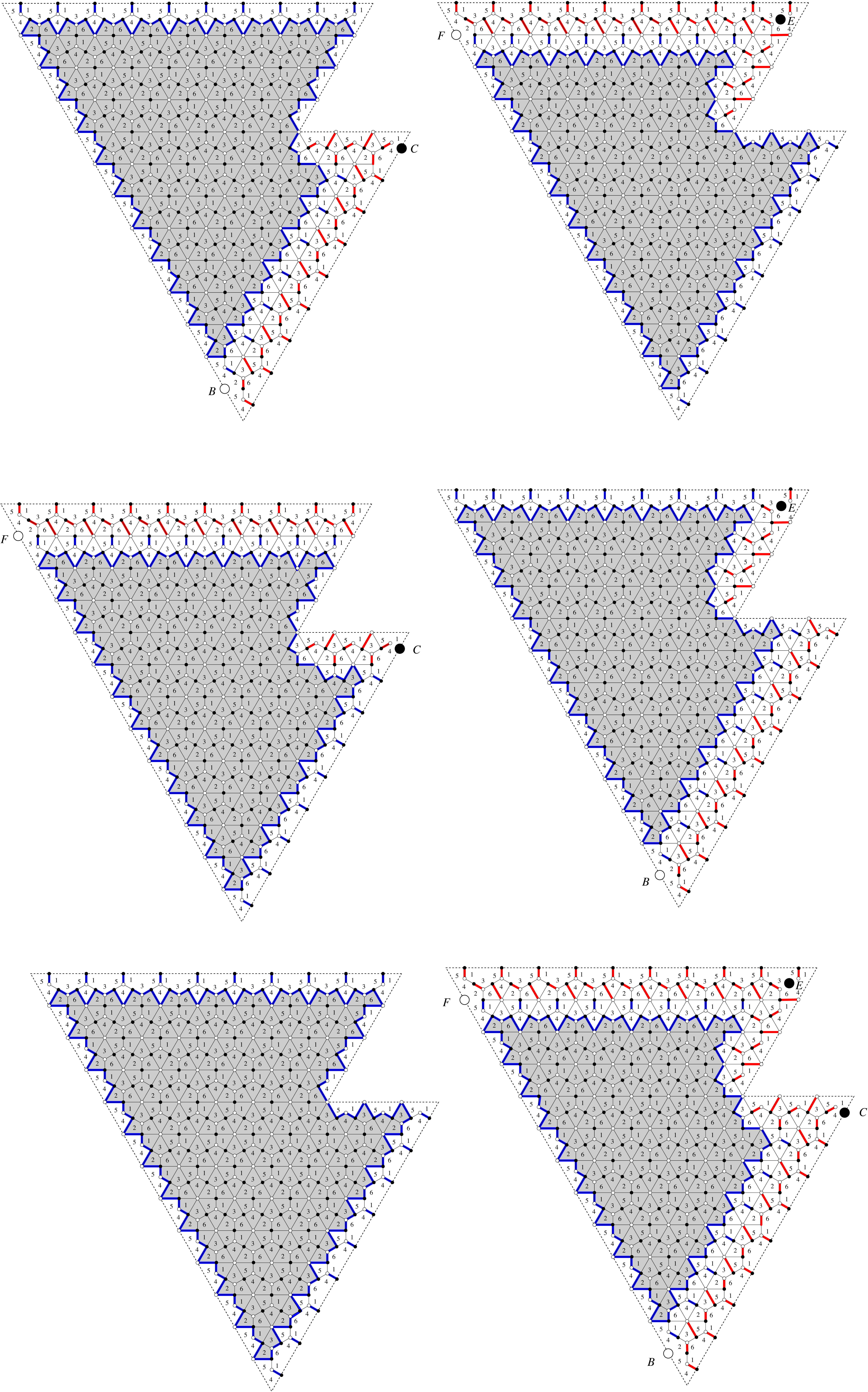}
\caption{Example of (R2) using Non-alternating Kuo Condensation with $i=-5, j=6,k=6$.}
\end{figure}

\begin{figure}
\includegraphics[width=12cm]{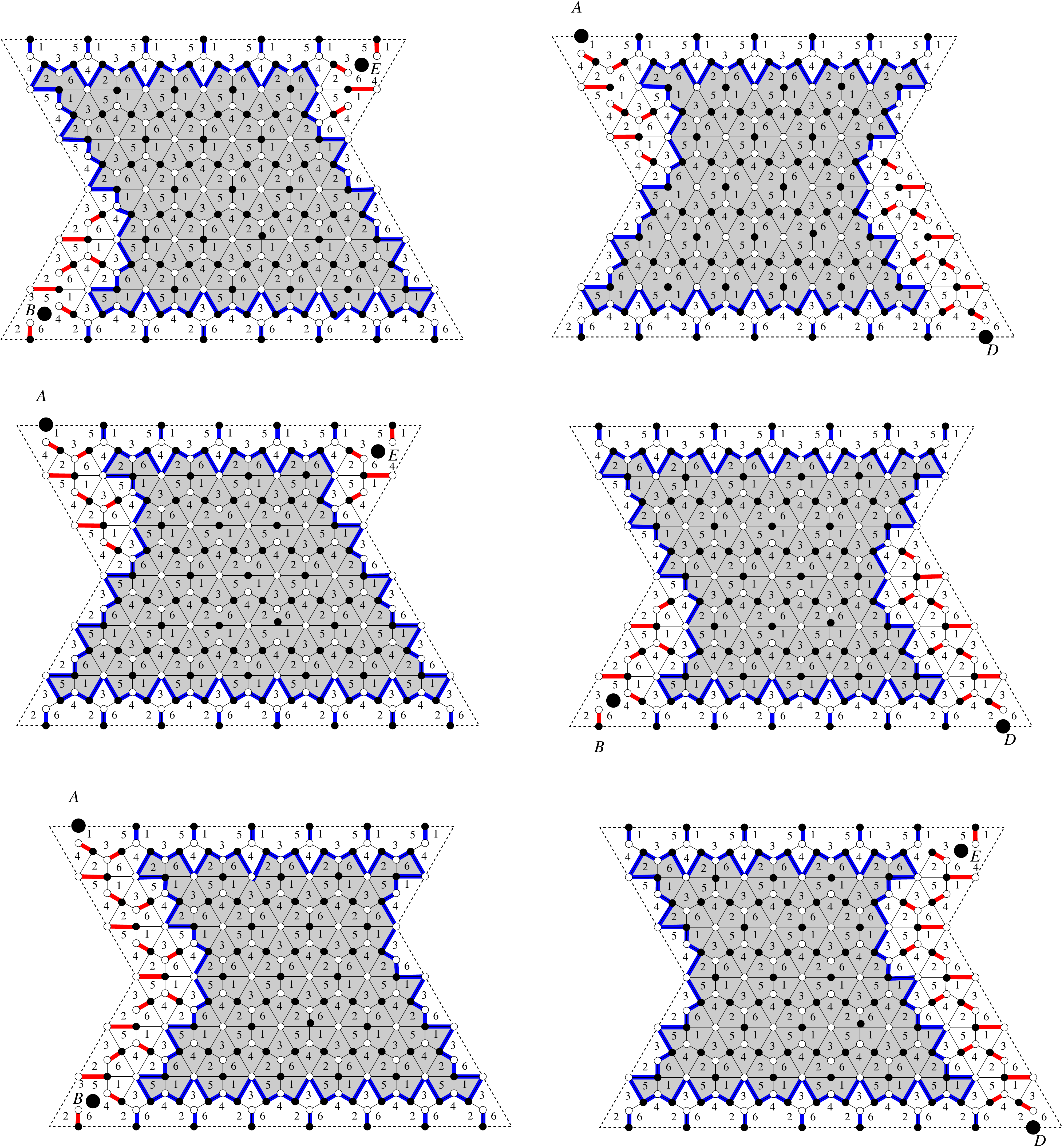}
\caption{Example of (R2) using Monochromatic Kuo Condensation with $i=0, j=4,k=0$.}
\label{fig:eggs}
\end{figure}

\begin{remark}
Balanced Kuo Condensation, as in Lemma \ref{Kuo1}, is utilized to prove Lemma 3.6 of \cite{LMNT}, Proposition 2 of \cite{zhang}, and the Recurrence (R4) in \cite{LaiNewDungeon}.  The proof of Theorem \ref{thm:main} shares similarities to the methods in those papers but using all four types of Kuo Condensation theorems provided us with a greater toolbox.  For instance, in \cite{LMNT}, a consequence of relying solely on only one of the four types of Kuo condensations was the need to consider only certain subsequences of $\tau$-mutation sequences along a so-called canonical path as opposed to mutations in any direction.  This was handled by induction and taking advantage of symmetries but required a non-intuitive decomposition of the $(i,j)$-plane into twelve regions.

Analogously, the first author's previous work in \cite{LaiNewDungeon} had to rely on a large number of base cases without utilizing the additional three Kuo Condensation Theorems.  We believe that more widely used application of all four of these Kuo condensation recurrences will allow more elegant attacks on other perfect matching enumeration problems that had previously appeared too daunting.
\end{remark}

\section{Conclusions and Open Questions} \label{sec:open}

In this paper, we succeeded in starting from the Model 1 $dP_3$ quiver and providing a $\mathbb{Z}^3$-parameterization of cluster variables that are reachable via sequences of toric mutations (i.e. two incoming arrows and two outgoing arrows at every vertex when it is mutated).  For such cluster variables, we presented an explicit algebra formula for their Laurent expansions in terms of this parameterization.  Then for most of these cluster variables (the ones corresponding to contours without self-intersections) we also gave a combinatorial interpretation to these Laurent expansions in terms of subgraphs of the $dP_3$ brane tiling $\mathcal{T}$.

We suspect that the methods of this paper should generalize to other quivers associated to brane tilings, i.e. those that can be embedded on a torus.  One interesting feature of the $dP_3$ case was the fact that the associated toric diagram and contours had six possible sides rather than the four sides that show up in the study of the octahedron recurrence \cite{speyer}, $T$-systems \cite{DiF}, Gale-Robinson Sequences and pinecones \cite{BMPW,JMZ}.  This led to new phenomena such as self-intersecting contours which we wish to investigate further in future work.

On the other hand, the $dP_3$ provides an example with a lot of symmetry coming from the toric diagram consisting of three pairs of anti-parallel sides.  Additionally, Model 1 gives rise to a brane tiling which is {\bf isoradial}, meaning that {\bf alternating-strand diagrams} \cite{Post,Scott}, also known as {\bf Postnikov diagrams} or {\bf zig-zag paths} \cite{GK}, behave (and can be drawn as) straight lines \cite{Vafa,HV}.  For less symmetric cases, the contours might not follow the lines of the brane tiling and instead require a description directly in terms of zig-zag paths or alternating-strands instead.

In particular, the boundaries of the subgraphs in this paper are actually segments of zig-zag-paths (as observed by R. Kenyon) and the alternating-strand diagrams would naturally oscillate along the contour lines thereby separating the line of white and black vertices on the strands' two sides.  However, translating lengths of these zig-zags into coordinates was not as well-behaved (see \cite{LMNT}) as the contour-coordinates of this current paper.  Based on conversations with David Speyer and Dylan Thurston, they have unpublished work that works with zig-zag coordinates for general brane tilings to yield cluster structures.

\begin{problem} \label{prob:self-int}
How do we make sense of contours that self-intersect, and their corresponding subgraphs so that we get the desired cluster variables and Laurent polynomials?  As indicated by conversations with R. Kenyon, D. Speyer, and B. Young, we suspect there is some sort of double-dimer interpretation for the Laurent expansions of such cluster variables.  Initial conjectured combinatorial interpretations in this direction did not produce the correct formulas, but more recent computations by the second author and D. Speyer showed more promise.
\end{problem}

\begin{problem} We saw in Remarks \ref{rem:RG} and \ref{rem:zono} that the generalized $\tau$-mutation sequences have particular properties of note to physicists that more general toric mutation sequences lack.  Mathematically, we wonder if there are other important ways that the generalized $\tau$-mutation sequences are special.  We have already seen that they satisfy Coxeter relations in this example.  Additionally, based on explorations by Michael Shapiro and Michael Gehktman \cite{GS}, it seems natural to ask if the generalized $\tau$-mutation sequences represent the integrable directions of motions in the discrete integrable system induced by this cluster algebra.\end{problem}

\begin{problem}
How well do the results and methods of this paper extend to other quivers and subgraph interpretations?  For example the $F_0 (\mathbb{P}^1 \times \mathbb{P}^1)$, Gale-Robinson Quivers, and others arising from the Octahedron Recurrence or $T$-systems have similar looking combinatorial interpretations using Aztec Diamonds or Pinecones.  Can we construct the Aztec Diamonds and Pinecones using a parameterization and contour coordinates?  In principle, this should agree with the methods in \cite{speyer}, but the transformation between the two coordinate systems and parameterization of cluster variables is still not written down.
\end{problem}

\begin{problem}
In the first authors' work on Blum's conjecture with Ciucu \cite{lai'}, a number of other subgraphs of doubly-periodic tilings of the plane appear.  Some of these are called Needle Families and Hexagonal Dungeons.  In fact the Hexagonal Dungeons are the dual graph for the Model 4 $dP_3$ quiver, so it is expected that the weighted enumeration of perfect matchings of these regions should indeed yield Laurent expansions of cluster variables starting with a different initial cluster.

However, one mystery is why contours have coordinates $(a,b,c,d,e,f)$ summing up to $1$ in the present paper (see Lemma \ref{lem:closeup}), but summing up to $0$ in the case of Hexagonal Dungeons.  For the Needle family, this seems to be a genuinely different quiver.  Can we find other quivers such that the weighted enumeration of perfect matchings of those subgraphs also yield Laurent expansions of cluster variables reachable by certain mutation sequences in the corresponding cluster algebra?
\end{problem}

\begin{problem}
In work in progress with Di Francesco, the equations of Section \ref{sec:explicit} seem well-suited for analyzing limit shapes.  This analysis uses techniques similar to those used in the case of T-systems by Di Francesco and Soto-Garrido for studying arctic curves \cite{diFSG}.  Several nuances appear to make this case quite interesting such as multi-dimensional time and the role of non-convex regions.
\end{problem}

\begin{problem}
In his study of Admissible $W$-graphs, Stembridge \cite{Stembridge} introduced a family of directed graphs (equivalently quivers) $\Lambda \times \Lambda$ known as a {\bf twist}.  Here $\Lambda$ is an orientation of a Dynkin diagram of type $A$, $D$, or $E$.  Such twists were recently studied by Galashin and Pylyavskyy in \cite{GaPy} as part of their classification of Zamolodchikov periodic quivers.  Such quivers and their associated cluster algebras have the property that one can define certain sequences of mutations such that their compositions with one another satisfy the Coxeter relations of $\Lambda$, when thought of as an action on cluster variables.  As pointed out to us by Pylyavskyy, the $dP_3$ quiver is an example of such a quiver, in particular this is a twist of the Weyl group Affine $\tilde{A}_2$.  The relevant Coxeter relations were seen in (\ref{eq: tau_relations}).  This thus motivates an exploration for combinatorial formulas for cluster variables associated to the quivers constructed as the twists of other Weyl groups or other Zamolodchikov periodic quivers.  
\end{problem}

{\bf Acknowledgments:} The authors are grateful to Mihai Ciucu, Philippe Di Francesco, Richard Eager, Sebastian Franco, Michael Gehktman, Rinat Kedem, Richard Kenyon, Pasha Pylyavskyy, Michael Shapiro, David Speyer, and Dylan Thurston for a number of inspirational discussions.  We are both appreciative of the hospitality of the Institute of Mathematics and its Applications (IMA) for providing support for this project.  The second author was supported by NSF Grants DMS-\#1148634 and DMS-\#13692980. Much of this research was also aided by the open source mathematical software \cite{sage}.  We are both appreciative of the hospitality of the Institute of Mathematics and its Applications (IMA) for providing support for this project. Much of this project was done during the time the first author worked at the IMA as a postdoctoral associate (2014--2016).

\bibliographystyle{alphaurl}

\newcommand{\etalchar}[1]{$^{#1}$}

\end{document}